%% file: main.tex
\documentclass[a4paper,12pt]{amsart}
\usepackage{bbm}
\usepackage{units}
\usepackage{mathtools}
\usepackage{cancel}
\usepackage{amsmath}
\usepackage{amsfonts}
\usepackage{latexsym}
\usepackage{amssymb}
\usepackage{mathrsfs}
\usepackage{amscd}
\usepackage{hyperref}
\usepackage{soul}
\usepackage{psfrag}
\usepackage{graphicx}
\usepackage{ulem}
\usepackage{fullpage}
\usepackage[all]{xy}
\usepackage{rotating}
\usepackage{url}
\usepackage{color}
\usepackage{enumerate}
\usepackage{cleveref}
\usepackage{dsfont}
\usepackage{tikz}
\usetikzlibrary{arrows,backgrounds,arrows.meta,decorations.markings}
\usepackage{tikz-cd}

\numberwithin{equation}{section}
\numberwithin{figure}{section}

\theoremstyle{plain}
\newtheorem{thm}{Theorem}[section]

\theoremstyle{plain}
\newtheorem{lem}[thm]{Lemma}

\theoremstyle{remark}
\newtheorem{rem}[thm]{Remark}

\theoremstyle{plain}

\theoremstyle{definition}
\newtheorem{defn}[thm]{Definition}

\theoremstyle{definition}

\theoremstyle{definition}

\theoremstyle{plain}
\newtheorem{prop}[thm]{Proposition}

\theoremstyle{plain}
\newtheorem{fact}[thm]{Fact}

\theoremstyle{definition}

\theoremstyle{plain}

\newcommand{\comments}[1]{}
\newcommand{\ra}{\rightarrow}
\newcommand{\rab}{\rangle}
\newcommand{\lra}{\longrightarrow}

\newcommand{\lab}{\langle}

\newcommand{\mcal}{\mathcal}

\newcommand{\N}{\mathbb N}

\newcommand{\C}{\mathbb{C}}

\newcommand{\mscr}{\mathscr}
\newcommand{\vlon}{\varepsilon}

\newcommand{\btimes}{\boxtimes}

\title{Unitary connections on Bratteli diagrams}
\author{Paramita Das, Mainak Ghosh, Shamindra Ghosh and Corey Jones}
\newcommand{\Contact}{{
		\bigskip
		\footnotesize
		
		Paramita Das, \textsc{Stat-Math Unit, Indian Statistical Institute}\par\nopagebreak
		\textit{E-mail address}: \texttt{paramita.isi@gmail.com}
		
		\medskip

		Mainak Ghosh, \textsc{Stat-Math Unit, Indian Statistical Institute}\par\nopagebreak
		\textit{E-mail address}: \texttt{main\_ghosh@rediffmail.com}

		\medskip

		Shamindra Ghosh, \textsc{Stat-Math Unit, Indian Statistical Institute}\par\nopagebreak
		\textit{E-mail address}: \texttt{shamindra.isi@gmail.com}

		\medskip
		
		Corey Jones, \textsc{Department of Mathematics, North Carolina State University}\par\nopagebreak
		\textit{E-mail address}: \texttt{cmjones6@ncsu.edu}
}}

\begin{document}
	\global\long\def\vlon{\varepsilon}
	\global\long\def\bt{\bowtie}
	\global\long\def\ul#1{\underline{#1}}
	\global\long\def\ol#1{\overline{#1}}
	\global\long\def\norm#1{\left\|{#1}\right\|}
	\global\long\def\os#1#2{\overset{#1}{#2}}
	\global\long\def\us#1#2{\underset{#1}{#2}}
	\global\long\def\ous#1#2#3{\overset{#1}{\underset{#3}{#2}}}
	\global\long\def\t#1{\text{#1}}
	\global\long\def\lrsuf#1#2#3{\vphantom{#2}_{#1}^{\vphantom{#3}}#2^{#3}}
	\global\long\def\tr{\triangleright}
	\global\long\def\tl{\triangleleft}
	\global\long\def\cc90#1{\begin{sideways}#1\end{sideways}}
	\global\long\def\turnne#1{\begin{turn}{45}{#1}\end{turn}}
	\global\long\def\turnnw#1{\begin{turn}{135}{#1}\end{turn}}
	\global\long\def\turnse#1{\begin{turn}{-45}{#1}\end{turn}}
	\global\long\def\turnsw#1{\begin{turn}{-135}{#1}\end{turn}}
	\global\long\def\fusion#1#2#3{#1 \os{\textstyle{#2}}{\otimes} #3}
	
	\global\long\def\abs#1{\left|{#1}\right	|}
	\global\long\def\red#1{\textcolor{red}{#1}}
\maketitle
\input{absint}
\input{prelim}
\input{unicon2bimod}
\input{tracialcase}
\input{flat2intwin}
\input{eg}

\Contact
\end{document}

%% file: absint.tex
\section{Introduction}

\begin{abstract}
In this paper, we extend Ocneanu's theory of connections on graphs to define a 2-category whose 0-cells are tracial Bratteli diagrams, and whose 1-cells are generalizations of unitary connections. We show that this 2-category admits an embedding into the 2-category of hyperfinite von Neumann algebras, generalizing fundamental results from subfactor theory to a 2-categorical setting.
\end{abstract}

Jones seminal results on the index for subfactors gave rise to the modern theory of subfactors \cite{Jo2}. Popa proved that amenable finite index subfactors of the hyperfinite $\rm{II}_{1}$ factor are completely classified by their standard invariant \cite{Po1}, which are axiomatized in general by standard $\lambda$-lattices \cite{Po2} or planar algebras \cite{Jo1}. This has led to remarkable progress in the classification of finite index hyperfinite subfactors, by transforming a large part this fundamentally analytic problem to the (essentially) algebraic problem of classifying abstract standard invariants \cite{JMS}, \cite{AMP}.

In the finite depth setting, Ocneanu introduced and established the theory of biunitary connections on 4-partite graph as an essential tool for constructing hyperfinite subfactors. Biunitary connections feature in his paragroup axiomatization of finite depth standard invariants (\cite{O0}, \cite{EvKaw}) but can also be used to construct infinite depth hyperfinite subfactors from finite graphs \cite{Sch}. While the other approaches to standard invariants are now more common, the theory of biunitary connections remain an important ingredient in the construction and classification of hyperfinite subfactors \cite{EvKaw}, \cite{JMS}. Many features of subfactor theory now have a clear higher-categorical interpretation (\cite{Mu1}, \cite{CPJP}, \cite{JMS}), and while there is some work investigating biunitary connections from a categorical viewpoint (\cite{Ch}), the general theory of biunitary connections and particularly their role in hyperfinite subfactor construction has remained mysterious from the categorical viewpoint.

In this paper, we shed some light on this problem by showing that graphs and bi-unitary connections can be viewed naturally as part of a larger W* 2-category $\textbf{UC}^\t{tr}$ (see \Cref{unicontr}). We then build a 2-functor to the 2-category of tracial von Neumann algebras, which puts the hyperfinite subfactor construction from biunitary connections into a larger categorical context. The 0-cells (or objects) of the 2-category $\textbf{UC}^{tr}$ are Bratteli diagrams equipped with tracial weighting data. These generalize the Bratteli diagrams arising from taking the tower of relative commutants of a finite index subfactor. 1-cells in our 2-category are \textit{unitary connections} between Bratteli diagrams which are compatible with the tracial data. These naturally generalize Ocneanu's biunitary connections from subfactor theory to our Bratteli diagram setting. Finally, the 2-cells of our category are built as certain fixed points under a UCP map, strongly resembling a noncommutative Poisson boundary as in \cite{Izm}, \cite{NY}.

Recall $\textbf{vNAlg}$ denotes the 2-category of von Neumann algebras, bimodules, and intertwiners. The following is the main theorem of the paper.

\begin{thm}\label{mainthm}
There is a {\normalfont W*} $ 2 $-functor $\normalfont \mathcal{PB}:\textbf{UC}^\t{tr}\rightarrow \textbf{vNAlg}$ which is fully faithful at the level of $ 2 $-cells.
\end{thm}

We note that the von Neumann algebras in the image of $\mathcal{PB}$ are always hyperfinite by construction. To see how the usual subfactor theory construction fits into this story, from a 4-partite graph and a biuntary connection we build a pair of tracial Bratteli diagrams by repeatedly reflecting the ``vertical" bipartite graphs, and taking the Markov trace as data. The horizontal graphs and the biunitary connection assemble into a 1-morphism in $\textbf{UC}^\t{tr}$ from this pair of tracial Bratteli diagrams. By carefully choosing the initial vertex data, we can build a unital inclusion of hyperfinite von Neumann algebras from this data, which we will see is just a special case of our construction \Cref{subfactors}. One way of looking at our result is that we are generalizing connections to be 1-morphisms between graphs that can be composed. Our main result is that a ``compositional'' version of the subfactor construction holds, and many of the results from the subfactor setting are true for bimodules as well. For example, it is well known that the relative commutants of the subfactor constructed as above can be computed as the ``flat part'' of the initial biunitary connection. We prove a generalization of this \Cref{periodicOcneanucompact}. 

To motivate our definitions in $\textbf{UC}^\t{tr}$, we first consider a purely algebraic category $\textbf{UC}$ consisting of Bratteli diagrams (without tracial data), unitary connections between them, and natural intertwiners between connections which we call flat sequences \Cref{unicon}. This 2-category is essentially equivalent to the 2-category studied in \cite{CPJ} in the context of fusion category actions on AF-C*-algebras, with only minor differences at the level of 0-cells and 2-only. As in \cite{CPJ}, from a 0-cell we define an AF-algebra\footnote{we are slightly abusing terminology: by AF-algebra we mean inductive limit of finite dimensional C*-algebras in the category of *-algebras, so we do not complete in norm}. We see that the 1-morphisms in $\textbf{UC}$ are precisely the data we need to define inductive limit bimodules between the AF-algebras built from the 0-cells, and the 2-cells in $\textbf{UC}$ are precisely the intertwiners between the resulting bimodules. Then picking a tracial state on the AF-algebras, we ask which 1-cell bimodules extend to the corresponding von Neumann completion, and if they do, what are the morphisms between them? This consideration leads us precisely to our definitions of 1-cell and 2-cell in $\textbf{UC}^\t{tr}$, which answers this question and proves our main theorem simultaneously.

There are at least two natural questions arising from our investigations. First, which bimodules between hyperfinite von Neumann algebras can be realized by this construction? This is deeply related to the question about possible values of the index for irreducible hyperfinite subfactors and the recent work of Popa \cite{Po5}. Second, how is the story modified if we pick arbitrary states on our Bratteli diagrams instead of tracial ones? This could have interesting applications for the study of defects and categorical symmetries in 1-D spin chains.

The outline of the paper is as follows. In Section 2, we record some categorical preliminaries and introduce notation. In Section 3, we detail the 2-category $\textbf{UC}$, and its realization as a category of bimodules over AF-algebras. In Section 4, we introduce $\textbf{UC}^\t{tr}$ and prove our main theorem in Section 5. In Section 6, we investigate flatness and in Section 7 we consider some examples, including the relationship between our work and classical subfactor constructions, as well as the work of Izumi \cite{Izm}. 

\subsection*{Acknowledgements}
The authors would like to thank David Penneys for helpful comments and enlightening conversations. Corey Jones was supported by 
NSF Grant DMS-2100531.

%% file: prelim.tex
\section{Preliminaries and notations}

In this section, we set up several notations and state well-known facts which will be used in the latter part of this article.
Although, we have not used any pictures in this section, we urge the reader to translate the expressions in terms of pictures using standard graphical calculus of morphisms and see what pictorial moves are given by the equations and maps.

\subsection{Categorical trace}\label{CaTr}
Let $ \mcal M $ be a semisimple C*-category and $ V $ be a maximal set of mutually non-isomorphic simple objects
in $ \mcal M $.
For all $ v \in V $, $ x \in \t{ob} (\mcal M) $, consider the inner product $ \lab \cdot , \cdot \rab_{v,x} $ on $ \mcal M (v ,x) $ defined by $ \tau^* \sigma = \lab \sigma , \tau \rab_{v,x} \ 1_v $.
An orthonormal basis for such spaces is basically a maximal orthogonal family of isometries in $ \mcal M (v,x) $.

\noindent\textbf{Convention.} If a statement is independent of the choice of orthonormal basis for $ \mcal M (v,x) $, then we denote it by ONB$ (v,x) $.
For instance, $ 1_x = \us {v \in V} \sum\ \us {\sigma \in \t{ONB} (v,x)} \sum \sigma \sigma^* $.

Given a map $ \mu : V \ra (0,\infty) $ (referred as a \textit{weight function on $ \mcal M $}), consider the linear functional
\[
\t{End} (x) \ni \alpha \os {\displaystyle \t{Tr}_x} \longmapsto \ \us {v \in V} \sum\ \us {\sigma \in \t{ONB} (v,x)} \sum \mu_v \ \left\lab \alpha \sigma , \sigma \right \rab_{v,x} \ \in \C \ .
\]
Clearly, $ \t{Tr}_x $ is a faithful, positive functional.
Moreover, $ \t{Tr} = \left(\t{Tr}_x\right)_{x\in \t{ob} (\mcal M)}  $ is a `categorical' trace, namely it satisfies $ \t{Tr}_x  (\alpha \beta) = \t{Tr}_y  (\beta \alpha )$ for all $ \alpha \in \mcal M (y,x) $, $ \beta \in \mcal M (x,y)$.
We refer Tr as the \textit{categorical trace associated to the weight function $ \mu $}.

\subsection{Graphs and functors}\label{GraFunc}

Let $ \Gamma = (V_\pm,E)$ (also denoted by $ V_- \os {\displaystyle \Gamma} \lra V_+ $) be a bipartite graph with vertex sets $ V_\pm$ and edge sets $ E_{v_+ , v_-} $ for $ (v_+ , v_-) \in V_+ \times V_- $, such that the set of edges attached to any vertex is non-empty and finite.
Consider the semisimple C*-category $ \mcal M_\pm $ whose objects consists of finitely supported $ V_\pm $-graded finite dimensional Hilbert spaces.
Note that $ \Gamma $ induces the following pair of faithful functors
\[
\mcal M_- \ni \left(H_{v_-}\right)_{v_-\in V_-}  \os {F_+ }\longmapsto \left( \us {v_- \in V_-} \oplus H_{v_-} \otimes \ell^2(E_{ v_+ , v_- }) \right)_{v_+ \in V_+} \in \mcal M_+
\]
\[
\mcal M_+ \ni \left(H_{v_+}\right)_{v_+\in V_+}  \os {F_- }\longmapsto \left( \us {v_+ \in V_+} \oplus H_{v_+} \otimes \ell^2(E_{v_+ , v_- }) \right)_{v_- \in V_-} \in \mcal M_-
\]
where the action of each of the functors on a morphism is obtained by distributing it over the direct sum and tensor product keeping the edge vectors (in $ \ell^2 (E_{v_+ , v_- }) $'s) fixed.
One can easily show that such $ F_\pm $ is $ * $-linear, \textit{bi-faithful} (that is, both itself and it adjoint are faithful), and $ F_+ $ and $ F_- $ are adjoints of each other.
Conversely, every adjoint pair of $ * $-linear faithful functors $ F_\pm : \mcal M_\mp \ra \mcal M_\pm  $ between semisimple C*-categories $ \mcal M_\pm $, gives rise to such a bipartite graph by setting the vertex set $ V_\pm $ as a maximal set of mutually non-isomorphic simple objects in $ \mcal M_\pm $
, and edge set $ E_{v_+,v_-} $ as a choice of orthonormal basis in $ \mcal M_+ (v_+ , F_+ v_-) $ with respect to $ \lab \cdot , \cdot \rab_{v_+ , F_+ v_-} $ (defined in \Cref{CaTr}).

For $ F_\pm $, $ \mcal M_\pm  $, $ V_\pm $ as before, we will try to characterize the set of solutions to conjugate equations implementing the duality of $ F_\pm $.
At this point, it will be useful for us to introduce some pictorial notation for morphisms and natural transformations which are quite standard in articles appearing in category theory.
\subsection*{Pictorial notation}
(i) A morphism $ f:C\ra D $ will be denoted by
\raisebox{-8mm}{
	\begin{tikzpicture}
		\node[draw,thick,rounded corners, fill=white] at (0,0) {$f$};
		\begin{scope}[on background layer]
			\draw[dashed,thick] (0,-.9) to (0,.9);
		\end{scope}	
		\node[right] at (-0.1,0.6) {$D$};
		\node[right] at (-0.1,-0.6) {$C$};
	\end{tikzpicture}
}, and composition of two morphisms will be represented by two vertically stacked labelled boxes.

(ii) Let $\mcal C$ and $\mcal D$ be two categories and $F,G: \mcal C \ra \mcal D$ be two functors.
Then a natural transformation $\eta:F \ra G$ will be denoted by 
\raisebox{-8mm}{
	\begin{tikzpicture}
		\node[draw,thick,rounded corners,fill = white] at (0,0) {$\eta$};
		\begin{scope}[on background layer]
			\draw (0,-.8) to (0,.8);
		\end{scope}
		\node[right] at (-0.1,0.6) {$G$};
		\node[right] at (-0.1,-0.6) {$F$};
	\end{tikzpicture}
}.
For an object $x$ in $\mcal C$, the morphism $\eta_{x} :Fx \ra Gx$ will be denoted by
\raisebox{-8mm}{
	\begin{tikzpicture}
		\node[draw,thick,rounded corners,fill=white] at (0,0) {$\eta_{x}$};
		\begin{scope}[on background layer]
			\draw (-.2,-.8) to (-.2,.8);
			\draw[thick,dashed] (.2,-.8) to (.2,.8);
		\end{scope}
		\node[left] at (-0.07,0.6) {$G$};
		\node[left] at (-0.07,-0.6) {$F$};
		\node[right] at (0.07,0.6) {$x$};
		\node[right] at (0.07,-0.6) {$x$};
		\node[left=] at (1,0) {$ = $};
		\node[draw,thick,rounded corners,fill=white] at (1.3,0) {$\eta$};
		\begin{scope}[on background layer]
			\draw (1.3,-.8) to (1.3,.8);
		\end{scope}
		\node[right] at (.8,0.6) {$G$};
		\node[right] at (.8,-0.6) {$F$};
		\draw[thick,dashed] (1.7,-0.8) to (1.7,0.8);
		\node[right] at (1.6,0) {$x$};
	\end{tikzpicture}
}.

(iii)
For a $ * $-linear functor $ F: \mcal C \ra \mcal D$ between two semisimple C*-categories categories, we will denote a solution to conjugate equation by
\[
\rho = \raisebox{-2.5mm}{
	\begin{tikzpicture}
		\draw (-.6,.3) to[out=-90,in=180] (0,-.3) to[out=0,in=-90] (.6,.3);
		\node[right] at (0.4,0.1) {$F'$};
		\node[left] at (-0.5,0.1) {$F$};
	\end{tikzpicture}
}\!\!\! : \t{id}_{\mcal D} \lra FF' \ \ \t{ and } \ \ 
\rho' = \raisebox{-2.5mm}{
	\begin{tikzpicture}
		\draw (-.6,.3) to[out=-90,in=180] (0,-.3) to[out=0,in=-90] (.6,.3);
		\node[right] at (0.4,0.1) {$F$};
		\node[left] at (-0.5,0.1) {$F'$};
	\end{tikzpicture}
}\!\!\! : \t{id}_{\mcal C} \lra F 'F 
\]
\[
\rho^* = \raisebox{-2.5mm}{
	\begin{tikzpicture}
		\draw (-.6,-.3) to[out=90,in=180] (0,.3) to[out=0,in=90] (.6,-.3);
		\node[right] at (0.5,-0.1) {$F'$};
		\node[left] at (-0.45,-0.1) {$F$};
	\end{tikzpicture}
}\!\!\! : FF' \lra \t{id}_{\mcal D} \ \ \t{ and } \ \ 
\left[\rho' \right]^* = \raisebox{-2.5mm}{
	\begin{tikzpicture}
		\draw (-.6,-.3) to[out=90,in=180] (0,.3) to[out=0,in=90] (.6,-.3);
		\node[right] at (0.5,-0.1) {$F$};
		\node[left] at (-0.4,-0.1) {$F'$};
	\end{tikzpicture}
}\!\!\! : F 'F \lra  \t{id}_{\mcal C}
\]
where $ F' : \mcal D \ra \mcal C $ is an adjoint functor of $ F $.

We will extend the above dictionary (between things appearing in the category world and pictures) in an obvious way.
For instance, composition of morphisms and natural transformations will be pictorially represented by stacking the boxes vertically whereas tensor product (resp., composition) of objects (resp., functors) by parallel vertical strings.
For simplicity, sometimes we will not label all of the strings (with any object or functor) emanating from a box (labelled with a morphism or a natural transformation) when it can be read off from the context.
We urge the reader to get used to the various picture moves which are induced by relations satisfied by operations, such as, composition, tensor product, etc. between objects, morphisms, functors and natural transformations.
In fact, the main purpose of introducing this graphical language is because of the ease of working with these moves instead of long equations.\\

\begin{fact}
If $ \rho^\pm : \t{id}_{\mcal M_\pm} \ra F_\pm F_\mp $ is a solution to the conjugate equation for $ F_\pm $, then for each $ (v_+,v_-) \in V_+ \times V_- $, there exists an orthonormal basis $ E_{v_+,v_-} $ of $ \mcal M_+( v_+ , F_+ v_-) $ and a `weight' function $ \kappa_{v_+,v_-} : E_{v_+,v_-} \ra (0,\infty) $ and  satisfying the following:
\begin{equation}\label{sol2conj}
\left(\rho^-_{v_-}\right)^* \  F_- \left(\sigma \ \tau^*\right) \ \rho^-_{v_-}\ =\ \delta_{\sigma = \tau} \ \kappa_{v_+,v_-} (\sigma) \  1_{v_-}\ \t{ for all } \ \sigma , \tau \in E_{v_+,v_-}, \ v_\pm\in V_\pm.
\end{equation}
Conversely, to every such family of orthonormal bases and weight functions, one can associate a solution to the conjugate equations implementing the duality of $ F_\pm $ satisfying \Cref{sol2conj}.
\end{fact}
The above easily follows from the spectral decomposition of the faithful positive functional $\left[\left(\rho^-_{v_-}\right)^* \  F_- \left( \bullet \right) \ \rho^-_{v_-}\right] : \t{End} (F_+ v_- ) \os{ }{\longmapsto} \t{End} (v_-) = \C 1_{v_-} $.
Thus, from the adjoint pair of $ * $-linear faithful functors, we not only get a bipartite graph, the solution to the conjugate equations puts a positive scalar weight on each edge.
Further, the set $E_{v_-,v_+ } \coloneqq  \left\{ \left(\kappa_{v_+,v_-} (\sigma)\right)^{- \frac 1 2} \ \left[F_-\ \sigma^* \right]\ \rho^-_{v_-} \ : \ \sigma \in E_{v_+, v_-} \right\} $ turns out to be an orthonormal basis of $ \mcal M_- (v_- , F_- v_+) $ and satisfies an equation analogous to \Cref{sol2conj} with weight function $ \kappa_{v_- , v_+} \coloneqq \displaystyle \frac 1 {\kappa_{v_+ , v_-}}  $.

The solution $ \rho^\pm $ will be called `\textit{tracial}' if the weight function is constant on edges for every fixed pair of vertices.
Indeed, for tracial solution $ \rho^\pm $, \Cref{sol2conj} becomes
\begin{equation}\label{trsol2conj}
\begin{split}
\left(\rho^-_{v_-}\right)^* \  F_- \left(\sigma \ \tau^*\right) \ \rho^-_{v_-}\ =\ \kappa_{v_+,v_-} \ \lab \sigma , \tau \rab_{v_+,F_+ v_-} \  1_{v_-}\ \t{ for all } \ \sigma , \tau \in \mcal M_+ (v_+ , F_+ v_-)\\
\left(\rho^+_{v_+}\right)^* \  F_+ \left(\sigma \ \tau^*\right) \ \rho^+_{v_+}\ =\ \kappa_{v_-,v_+} \ \lab \sigma , \tau \rab_{v_-,F_- v_+} \  1_{v_+}\ \t{ for all } \ \sigma , \tau \in \mcal M_- (v_- , F_- v_+)
\end{split}
\end{equation}
and the map $\left[\left(\rho^-_{v_-}\right)^* \  F_- \left( \bullet \right) \ \rho^-_{v_-}\right] $ is tracial and so is $\left[\left(\rho^+_{v_+}\right)^* \  F_+ \left( \bullet \right) \ \rho^+_{v_+}\right]$.
We also get a conjugate linear unitaries
\[
\begin{split}
\mcal M_+ (v_+ , F_+ v_- ) \ni \sigma \ \os {\displaystyle J_{v_+,v_-} } {\longmapsto}\  {\sqrt{\kappa_{v_-,v_+} }} \ \left[F_-\ \sigma^* \right]\ \rho^-_{v_-} \in \mcal M_- (v_- , F_- v_+) \ , \\
\mcal M_- (v_- , F_- v_+ ) \ni \sigma \ \os {\displaystyle J_{v_-,v_+} } {\longmapsto}\  {\sqrt{\kappa_{v_+,v_-} }} \ \left[F_+\ \sigma^* \right]\ \rho^+_{v_+} \in \mcal M_+ (v_+ , F_+ v_-) \ .
\end{split}
\]
The two `loops' are given by:
\begin{equation}\label{loop-anticlock}
\left(\rho^+_\bullet \right)^* \circ \rho^+_\bullet \ =\   \left( \left\{\sum_{v_- \in V_-} \ N_{v_+,v_-} \ \kappa_{v_- , v_+}\right\} \ 1_{v_+} \right)_{v_+ \in V_+} \ \in \t{End} (\t{id}_{\mcal M_+}),
\end{equation}
\begin{equation}\label{loop-clock}
\left(\rho^-_\bullet \right)^* \circ \rho^-_\bullet \ =\   \left( \left\{\sum_{v_+ \in V_+} \ N_{v_+,v_-} \ \kappa_{v_+ , v_-}\right\} \ 1_{v_-} \right)_{v_- \in V_-} \ \in \t{End} (\t{id}_{\mcal M_-}),
\end{equation}
where $ N_{v_+,v_-} \coloneqq \dim_\C \left( \mcal M_+ (v_+ , F_+  v_-) \right)  = \dim_\C \left( \mcal M_- (v_- , F_-  v_+) \right)$ (that is, the number of edges between $ v_+ $ and $ v_- $ in the bipartite graph). (Note that a natural linear transformation between *-linear functors from one semisimple C*-category to another, is captured fully by its components corresponding to the simple objects.)

\subsection{Trace on natural transformations}\label{NaTrace}$ \ $

In this article, we will be working with the $ 2 $-category of \textit{weighted semisimple C*-categories}, denoted by WSSC*Cat, whose $ 0 $-cells are finite semisimple C*-categories along with a weight function on it (that is, a positive real valued map from the isomorphism classes of simple objects, as considered in \Cref{CaTr}), $ 1 $-cells are $ * $-linear bi-faithful functors and $ 2 $-cells are natural linear transformations.
Further, for the duality of the adjoint pair of $ 1 $-cells
$ \left(\mcal M_- ,  \ul \mu^-\right)
\stackrel{\os{\displaystyle F_+}{\displaystyle\lra}}{\us {\displaystyle F_-}{ _{\displaystyle \longleftarrow}}} \left(\mcal M_+ ,  \ul \mu^+\right)$, we will consider tracial solution $ \rho^\pm : \t{id}_{\mcal M_\pm} \ra F_\pm F_\mp $ to the conjugate equations associated to the constant weight on edges given by $ \kappa_{v_+ , v_-} = \displaystyle \frac{\mu^+_{v_+}}{\mu^-_{v_-}} $ for $ (v_+,v_-)  \in V_+ \times V_-$; we refer such a solution to be \textit{commensurate with the weight functions (on the simple objects) $ \left(\ul \mu^- , \ul \mu^+\right) $}.
Using the categorical trace Tr associated to the weight function $ \ul \mu^\pm $, one may obtain the relation:
\begin{equation}\label{TrCondExp}
	\t{Tr}_x \left(\left(\rho^\pm_x \right)^* F_\pm (\alpha) \rho^\pm_x \right) \ = \ \t{Tr}_{F_\mp (x)} (\alpha) \ \t{ for all } x \in \t{ob} (\mcal M_\pm), \alpha \in \t{End} (F_\mp (x)).
\end{equation}

We will now exhibit a similar categorial trace on the endomorphism space of every $ 1 $-cell between two `finite' $ 0 $-cells (that is, there is finitely many isomorphism classes of simple objects in the semisimple C*-category of the $ 0 $-cell); further, this trace will be compatible with the tracial solution  commensurate with the weight function in the $ 0 $-cells.
\begin{prop}\label{NTTr}
Let $ \left(\mcal M , \ul \mu \right)$, $\left( \mcal N , \ul \nu\right)$, $\left( \mcal Q , \ul \pi \right)$ be finite $ 0 $-cells in {\normalfont WSSC*Cat}, and $ \Lambda : \mcal M \ra \mcal N $, $ \Sigma: \mcal N \ra \mcal Q $ be $ * $-linear bi-faithful functors.
Suppose $ V_{\mcal M} $, $ V_{\mcal N} $, $ V_{\mcal Q} $ are maximal sets of mutually non-isomorphic simple objects in $ \mcal M $, $ \mcal N $, $ \mcal Q $ respectively.
	
(a)	The map
	\[
	\normalfont\t{End} (\Lambda) \ni \eta \os{\t{Tr}^\Lambda} \longmapsto \sum_{u \in V_{\mcal M}} \mu_u \; \t{Tr}^{\;{\ul \nu}}_{\Lambda u} \left( \eta _u \right) \; \in \C
	\]
	is a positive faithful trace.

\noindent We will refer $\normalfont \t{Tr}^\Lambda $ as the `trace on $ \normalfont \t{End} (\Lambda) $ commensurate with $ \left(\ul \mu , \ul \nu\right) $'.
	
	(b) If $ \left(\t{id}_{\mcal M} \os{\displaystyle \rho}\ra \ol \Lambda \Lambda\, , \, \t{id}_{\mcal N} \os{\displaystyle \ol  \rho}\ra \Lambda \ol  \Lambda\right) $ (resp., $ \left(\t{id}_{\mcal N} \os{\displaystyle \beta}\ra \ol \Sigma \Sigma\, , \, \t{id}_{\mcal Q} \os{\displaystyle \ol  \beta}\ra \Sigma \ol  \Sigma\right) $)  is a solution to conjugate equations for the duality of $ \Lambda $ (resp., $ \Sigma $) commensurate with $ (\ul \mu , \ul \nu) $ (resp., $ (\ul \nu , \ul \pi) $), then
	\[
	\normalfont \t{Tr}^{\, \Lambda} \left(\beta^*_\Lambda\; \ol \Sigma (\eta)\; \beta_\Lambda\right) = \t{Tr}^{\, \Sigma \Lambda} (\eta) = \t{Tr}^{\, \Sigma} \left( \Sigma \left(\ol \rho^*\right)\; \eta_{\, \ol \Lambda} \;  \Sigma \left(\ol \rho\right) \right) \; \t{ for } \; \eta \in  \t{End}(\Sigma \Lambda) \;.
	\]
\end{prop}
\begin{proof}
	(a) To each $ \alpha \in \t{End} (\Lambda x) $ for $ x \in \t{Ob} (\mcal M) $, we associate the natural transformation $ \left[\alpha \right] \coloneqq \left( \us{u \in V_{\mcal M}}{\sum} \, \us{\substack{\sigma \in \t{ONB} (u, x) \\ \tau \in \t{ONB} (u,y) } }{\sum} \Lambda \left(\tau \, \sigma^*\right) \, \alpha \, \Lambda \left(\sigma \, \tau^*\right) \right)_{y \in \t{Ob} (\mcal M)} $.
	In terms of this association, we may express $ \t{End} (\Lambda) $ as a direct sum of full matrix algebras indexed by $ V_{\mcal M} \times V_{\mcal N} $, and a system of matrix units of the summand corresponding to $ (u,v) \in V_{\mcal M} \times V_{\mcal N}$ is given by $ \left\{ [\sigma \tau^*] : \sigma , \tau \in \t{ONB} (v, \Lambda u)  \right\} $.
	Note that $ \t{Tr}^{\Lambda} \left([\sigma \tau^*]\right) = \delta_{\sigma = \tau}\, \mu_u \, \nu_v$ which is positive and independent of the choice of $ \sigma  $ and $ \tau $.
	
	(b) From the definition, the left side turns out to be $ \us {u \in V_{\mcal M}}\sum \mu_u \; \t{Tr}^{\;{\ul \nu}}_{\Lambda u} \left( \beta^*_{\Lambda u }\; \ol \Sigma (\eta_u)\; \beta_{\Lambda u} \right) $ which is equal to (applying \Cref{TrCondExp}) $ \us {u \in V_{\mcal M}}\sum \mu_u \; \t{Tr}^{\;{\ul \pi}}_{\Sigma \Lambda u} \left( \eta_u \right) = \t{Tr}^{\, \Sigma \Lambda} (\eta)$.
	
Pictorially the right side can be expressed as
\[
\us{v \in V_{\mcal N}}{\sum} \nu_v\, {\t{Tr}}^{\ul \pi}_{\Sigma v}
\raisebox{-.7cm}{
\begin{tikzpicture}
\draw[dashed,thick] (.8,-.8)to (.8,.8);
\draw[->] (-.1,-.8) -- (-.1,-.3);
\draw[->] (-.1,-.3) -- (-.1,.8);
\draw[->] (.6,0) -- (.6,-.6) to[in=-90,out=-90] (.1,-.6) -- (.1,.6) to[in=90,out=90] (.6,.6) --  (.6,0);
\node[draw,thick,rounded corners, fill=white] at (0,0) {$\eta$};
\node at (-.25,.5) {$ \Sigma$};
\node at (-.25,-.6) {$ \Sigma$};
\node[right] at (0,.45) {$ \Lambda$};
\node at (.95,0) {$ v$};
\end{tikzpicture}
} = \us{\substack{u \in V_{\mcal M} \\ v \in V_{\mcal N} \\ \sigma \in \t{ONB} (u , \ol \Lambda v) }}{\sum} \nu_v\, \t{Tr}^{\ul \pi}_{\Sigma v} 
\raisebox{-1.1cm}{
\begin{tikzpicture}
\draw[dashed,thick] (.7,-.8)to (.7,.8);
\draw[dashed,thick] (.8,.6)to (.8,1.2);
\draw[dashed,thick] (.8,-.6)to (.8,-1.2);
\draw[->] (-.1,-1.2) -- (-.1,-.3);
\draw[->] (-.1,-.3) -- (-.1,1.2);
\draw[->] (.1,-.8) -- (.1,-.3);
\draw[->] (.1,-.3) -- (.1,.8);
\draw (.1,.8) to[in=90,out=90] (.6,.8);
\draw (.1,-.8) to[in=-90,out=-90] (.6,-.8);
\node[draw,thick,rounded corners, fill=white] at (0,0) {$\eta$};
\node[draw,thick,rounded corners, fill=white] at (.7,.6) {$\sigma$};
\node[draw,thick,rounded corners, fill=white] at (.7,-.6) {$\sigma^*$};
\node at (-.25,.7) {$ \Sigma$};
\node at (-.25,-.7) {$ \Sigma$};
\node[right] at (0,.45) {$ \Lambda$};
\node at (.85,0) {$u$};
\node at (.95,1) {$v$};
\node at (.95,-1.05) {$v$};
\end{tikzpicture}
} = \us{\substack{u \in V_{\mcal M} \\ v \in V_{\mcal N} \\ \tau \in \t{ONB} (v , \Lambda u) }}{\sum} \mu_u\, \t{Tr}^{\ul \pi}_{\Sigma v}
\raisebox{-1.35cm}{
	\begin{tikzpicture}
		\draw[dashed,thick] (.8,-.9)to (.8,.9);
		\draw[dashed,thick] (.7,.9)to (.7,1.5);
		\draw[dashed,thick] (.7,-.9)to (.7,-1.5);
\draw[->] (-.1,-1.5) -- (-.1,-.3);
\draw[->] (-.1,-.3) -- (-.1,1.5);
\draw[->] (.1,.3) to[in=-90,out=90] (.6,.6);
\draw[<-] (.1,-.3) to[in=90,out=-90] (.6,-.65);
		\node[draw,thick,rounded corners, fill=white] at (0,0) {$\eta$};
		\node[draw,thick,rounded corners, fill=white] at (.7,.9) {$\tau^*$};
		\node[draw,thick,rounded corners, fill=white] at (.7,-.9) {$\tau$};
		\node at (-.25,.7) {$ \Sigma$};
		\node at (-.25,-.7) {$ \Sigma$};
		\node at (0.45,.3) {$ \Lambda$};
		\node at (0.2,-.6) {$ \Lambda$};
		\node at (.95,0) {$u$};
		\node at (.85,1.35) {$v$};
		\node at (.85,-1.3) {$v$};
	\end{tikzpicture}
}
\]
	\[
	= \us{\substack{u \in V_{\mcal M} \\ v \in V_{\mcal N} \\ \sigma \in \t{ONB} (v , \Lambda u) }}{\sum} \mu_u\, \t{Tr}^{\ul \pi}_{\Sigma \Lambda  u} \left( \eta \; \Sigma \left(  \tau \, \tau^* \right)  \right) = \t{Tr}^{\Sigma \Lambda} (\eta)\; .
	\]
\end{proof}
\begin{rem}
	The trace in \Cref{NTTr} (a), is `categorical', that is, $ \tilde \Lambda : \mcal M \ra \mcal N $ is another functor with the same PF vectors as that of $ \Lambda $, then $ \t
	{Tr}^{\Lambda} \left(\gamma \, \eta \right) = \t
	{Tr}^{\tilde \Lambda} \left(\eta \, \gamma \right) $ for $ \eta \in \t{NT} (\Lambda , \tilde \Lambda) $, $ \gamma \in \t{NT}(\tilde \Lambda , \Lambda) $.
\end{rem}

%% file: unicon2bimod.tex
\section{Unitary connections and right correspondences}\label{unicon}

Bratteli diagrams are incredibly useful tools for studying inductive limits of semisimple algebras (also called locally semisimple algebras). In this section we introduce a combinatorial 2-category whose objects are Bratteli diagrams and 1-cells are generalizations of Ocneanu's connections. Our perspective is that our 1-cells can naturally be viewed as ``Bratteli diagrams for bimodules" between locally semisimple algebras. Thus as we describe our 2-category $\textbf{UC}$, we will explain its relationship to algebras and bimodules. As a consequence, we build a fully faithful 2-functor $\textbf{UC}$ into the 2-category of algebras, bimodules, and intertwiners.

\vspace*{4mm}

\subsection{The $ 0 $-cells}\label{0cellspreC*}$ \ $

These are sequences consisting of finite bipartite graphs $\; V_0 \os {\displaystyle \Gamma_1} \lra V_1 \os {\displaystyle \Gamma_2} \lra V_2 \os {\displaystyle \Gamma_3} \lra V_3 \cdots  $ (where $ V_j$'s are the vertex sets) such that none of the vertices is isolated.
As described in the \Cref{GraFunc}, given such a data, we will often work with the corresponding $ * $-linear, bi-faithful functor $ \Gamma_k : \mcal M_{k-1} \ra \mcal M_k $ (where $ \mcal M_k $ is a semisimple C*-category whose isomorphism classes of the simple objects are indexed by the vertex set $ V_k $).
We will denote such a $ 0 $-cell by $ \left\{ \mcal M_{k-1}  \os {\displaystyle \Gamma_k} \lra \mcal M_k \right\}_{k\geq 1}$ or sometimes simply $ \Gamma_\bullet $.

Given such a $ 0 $-cell, we fix an object $ m_0 \coloneqq \us{v \in V_0} \bigoplus v \in \t{ob} (\mcal M_0)$.
Consider the sequence of finite dimensional C*-algebras $ \left\{ A_k \coloneqq \t{End} (\Gamma_k \cdots \Gamma_1 m_0) \right\}_{k\geq 0}$ (assuming $ A_0 = \t{End}(m_0) $) along with the unital $ * $-algebra inclusions given by $ A_{k-1} \ni \alpha \hookrightarrow \Gamma_k \, \alpha \in A_{k} $.
Indeed, the Bratteli diagram of $ A_{k-1} $ inside $ A_{k} $ is given by the graph $ \Gamma_k $.
To the $ 0 $-cell $ \Gamma_\bullet $, we associate the $ * $-algebra $ A_\infty \coloneqq \us{k\geq 0}{\cup} A_k $.

\vspace*{4mm}

\subsection{The $ 1 $-cells}\label{1cellspreC*}$ \ $

\begin{defn}
A $ 1 $-cell from the $ 0 $-cell $ \left\{\mcal M_{k-1}  \os{\displaystyle \Gamma_k}\lra \mcal M_k \right\}_{k\geq 1}$ to the $ 0 $-cell $ \left\{\mcal N_{k-1} \os{\displaystyle \Delta_k}\lra \mcal N_k \right\}_{k\geq 1}$ consists of a sequence of $ * $-linear bi-faithful functors $ \left\{\Lambda_k : \mcal M_k \ra \mcal N_k\right\}_{k\geq 0} $ and natural unitaries $ W_k: \Delta_k \Lambda_{k-1} \ra \Lambda_k \Gamma_k $ for $ k\geq 1 $. Such a $ 1 $-cell will be denoted by $ \left(\Lambda_\bullet, W_\bullet \right)$ or simply by $ \Lambda_\bullet $, and $ W_\bullet $ will be referred as a \textit{unitary connection associated to $ \Lambda_\bullet $}.
Denote the set of $ 1 $-cells from $ \Gamma_\bullet$ to $\Delta_\bullet $ by $ \textbf{UC}_1 \left(\Gamma_\bullet , \Delta_\bullet\right) $.
\end{defn}
We will abuse the notation $ \Lambda_k $ to denote the functor $ \Lambda_k : \mcal M_k \ra \mcal N_k $ as well as its associated adjacency matrix ($ V_{\mcal N_k} \times V_{\mcal M_k}$), and the same will be done for $ \Gamma_k $'s and $ \Delta_k $'s.
From the context, it will be clear whether we are using it as a functor or a matrix.
Pictorially, the natural unitary $ W_k $ appearing in the $ 1 $-cell will be represented by
\raisebox{-5mm}{
	\begin{tikzpicture}
		\draw[white,line width=1mm,out=-90,in=90] (-.25,.5) to (.25,-.5);
		\draw[out=-90,in=90,red] (-.25,.5) to (.25,-.5);
		\begin{scope}[on background layer]
			\draw[out=90,in=-90] (-.25,-.5) to (.25,.5);
		\end{scope}
		\node[right] at (0.15,-0.4) {$\Lambda_{k-1}$};
		\node[right] at (0.15,0.3) {$\Gamma_k$};
		\node[left] at (-0.15,-0.4) {$\Delta_k$};
		\node[left] at (-0.15,0.3) {$\Lambda_k$};
	\end{tikzpicture}
}
and $ W^*_k $ by
\raisebox{-5mm}{
	\begin{tikzpicture}
		\draw[white,line width=1mm,out=-90,in=90] (.25,.5) to (-.25,-.5);
		\draw[out=-90,in=90,red] (.25,.5) to (-.25,-.5);
		\begin{scope}[on background layer]
			\draw[out=90,in=-90] (.25,-.5) to (-.25,.5);
		\end{scope}
		\node[right] at (0.15,-0.4) {$\Gamma_k$};
		\node[right] at (0.15,0.3) {$\Lambda_{k-1}$};
		\node[left] at (-0.15,-0.4) {$\Lambda_k$};
		\node[left] at (-0.15,0.3) {$\Delta_k$};
	\end{tikzpicture}
}.

\vspace*{4mm}

To each such $ 1 $-cell $ (\Lambda_\bullet , W_\bullet)$, we will associate an $ A_\infty $-$ B_\infty $ right correspondence where $ n_0 $ and $ B_k $'s are related to $ \left\{\mcal N_{k-1} \os{\displaystyle \Delta_k}\lra \mcal N_k \right\}_{k\geq 1}$ exactly the way $ m_0 $ and $ A_k $'s are related to $ \left\{ \mcal M_{k-1} \os{\displaystyle \Gamma_k}\lra \mcal M_k \right\}_{k\geq 1}$ respectively.
For $ k\geq 0 $, set $ H_k \coloneqq \mcal N_k \left(\Delta_k \cdots \Delta_1 n_0 , \Lambda_k \Gamma_k \cdots \Gamma_1 m_0\right) $.
We have an obvious $ A_k $-$ B_k $-bimodule structure on $ H_k $ in the following way:
\[
A_k \times H_k \times B_k \ni (\alpha , \xi , \beta) \longmapsto \Lambda_k (\alpha) \circ \xi \circ \beta \ \in H_k \ .
\]
Again, there is a $ B_k $-valued inner product on $ H_k $ given by
\[
H_k \times H_k \ni (\xi,\zeta ) \os{\displaystyle \left\lab \cdot , \cdot \right \rab_{B_k}} \longmapsto \left\lab \xi , \zeta  \right \rab_{B_k} \coloneqq \zeta^* \circ \xi \ \in B_k \ .
\]
Next, observe that $ H_k $ sits inside $ H_{k+1} $ via the map
\[
H_k \ni \xi \longmapsto \left[\left(W_{k+1}\right)_{\Gamma_k \cdots \Gamma_1 m_0}\right] \circ \left[\Delta_{k+1} \xi \right] \ =
\raisebox{-.9cm}{
\begin{tikzpicture}
\draw[thick,dashed] (.5,-.7) to (.5,1);
		\draw[in=-90,out=90] (-.9,-.7) to (-.9,.3) to (-.6,1);
		\draw[in=-90,out=90] (-.3,-.7)to (-.3,1);
		\draw[in=-90,out=90] (.3,-.7)to (.3,1);
		\draw[white,line width=1mm,out=-90,in=90] (-.6,0.3) to (-.9,1);
		\draw[red,in=-90,out=90] (-.6,0.3) to (-.9,1);
		\node[draw,thick,rounded corners, fill=white,minimum width=45] at (0,0) {$\xi$};
		\node[right] at (-.4,.6) {$ \cdots$};
		\node[right] at (-.4,-.6) {$ \cdots$};
		\node[left] at (-.7,.8) {$ \Lambda_{k+1}$};
		\node[left] at (-.8,-.3) {$ \Delta_{k+1}$};
		\node[right] at (.4,-.6) {$ n_0$};
		\node[right] at (.4,.6) {$ m_0$};
	\end{tikzpicture}
} \in H_{k+1} \ .
\]
\begin{lem}\label{inclusionactioncompatible}
	The inclusions $ H_k \hookrightarrow H_{k+1} $, $ A_k \hookrightarrow A_{k+1} $, $ B_k \hookrightarrow B_{k+1} $ and the corresponding actions are compatible in the obvious sense.
\end{lem}
\begin{proof}
Naturality of $ W $ implies
\[
\begin{split}
&\Lambda_{k+1} \Gamma_{k+1} \alpha \ \circ\ \left[\left(W_{k+1}\right)_{\Gamma_k \cdots \Gamma_1 m_0} \ \circ \  \Delta_{k+1} \xi \right]\ \circ\ \Delta_{k+1} \beta \\
= &\left[\left(W_{k+1}\right)_{\Gamma_k \cdots \Gamma_1 m_0} \right]\ \circ\ \Delta_{k+1} \left( \Lambda_k \alpha \circ \xi \circ \beta\right)
\end{split}
\]
for all $ \xi \in H_k $, $ \alpha \in A_k $, $ \beta \in B_k $.
\end{proof}

Set $ H_\infty \coloneqq \us {k \geq 0} \cup H_k $ which clearly becomes an $ A_\infty $-$ B_\infty $ right correspondence.
Further, we will exhibit a Pimsner-Popa (PP) basis of the right-$ B_\infty $-module $ {H_\infty }$ with respect to the $ B_\infty $-valued inner product.
\begin{lem}\label{PPbasispreC*}
There exists a subset $ \mscr S $ of $ H_0 $ such that $ \displaystyle \sum_{\sigma \in \mscr S} \sigma \circ  \sigma^* \  = \ 1_{\Lambda_0 m_0}$; moreover, any such $ \mscr S $ is a PP-basis for the right $ B_\infty $-module $ H_\infty $.
\end{lem}
\begin{proof}
Since $ n_0 $ contains every simple object of $ \mcal N_0 $ as a subobject, therefore expressing the identity of $ \t{End} (\Lambda_0 m_0) $ as a sum of minimal projections, we have a resolution of identity $ 1_{\Lambda_0 m_0}  $ factoring through $ n_0 $, that is, there exists a subset $ \mscr S$ of $\mcal N (n_0, \Lambda_0 m_0) = H_0 $ satisfying:

(i) $ \sigma^* \sigma $ is a minimal projection of $ \t{End} (n_0) $, and

(ii) $ \displaystyle \sum_{\sigma \in \mscr S} \sigma \sigma^* \  = \ 1_{\Lambda_0 m_0}$ .

Condition (ii) and the definition of $ B_\infty $-valued inner product directly imply that $ \mscr S $ is indeed a PP-basis for the right $ B_\infty $-module $ H_\infty $.
\end{proof}

\vspace*{4mm}
\subsection{The $ 2 $-cells}\label{2cellspreC*}$ \ $

Let $ \Lambda_\bullet $ and $ \Omega_\bullet$ be two $ 1 $-cells from the $ 0 $-cell $ \left\{\mcal M_{k-1} \os{\displaystyle \Gamma_k}\lra \mcal M_k \right\}_{k\geq 1}$ to $ \left\{\mcal N_{k-1} \os{\displaystyle \Delta_k}\lra \mcal N_k \right\}_{k\geq 1}$.
The natural way to define a $ 2 $-cell will be considering a sequence of natural linear transformations from $ \Lambda_k $ to $ \Omega_k $ which are compatible with the natural unitaries $ W^\Gamma_k $ and $ W^\Omega_k $ for $ k \geq 1 $.
We define such compatibility in the following way.

\begin{defn}\label{XreldefpreC*}
A pair $ (\eta , \kappa) \in \t{NT}(\Lambda_k , \Omega_k) \times \t{NT}(\Lambda_{k+1} , \Omega_{k+1}) $ is said to satisfy \textit{exchange relation} if the condition
\raisebox{-1cm}{
	\begin{tikzpicture}
		\draw[in=-90,out=90] (0,0) to (0,1.2);
		\draw[in=-90,out=90] (0,1.2) to (0.6,2);
		\draw[in=-90,out=90] (3.4,1) to (3.4,2);
		\draw[in=-90,out=90] (2.6,0) to (3.4,1);
		\draw[red,in=-90,out=90] (.6,0) to (0.6,1.2);
		\draw[white,line width=1mm,in=-90,out=90] (0.6,1.2) to (0,2);
		\draw[red,in=-90,out=90] (0.6,1.2) to (0,2);
		\draw[red,in=-90,out=90] (2.6,1) to (2.6,2);
		\draw[white,line width=1mm,in=-90,out=90] (3.4,0) to (2.6,1);
		\draw[red,in=-90,out=90] (3.4,0) to (2.6,1);
		\node[draw,thick,rounded corners, fill=white] at (0.6,0.8) {$\eta$};
		\node[draw,thick,rounded corners, fill=white] at (2.6,1.4) {$\kappa$};
		\node[right] at (.5,0.2) {$ \Lambda_k $};
		\node[right] at (.5,1.35) {$\Omega_k $};
		\node[left] at (0.2,1.8) {$\Omega_{k+1} $};
		\node[right] at (3.3,0.2) {$ \Lambda_k $};
		\node[left] at (2.8,.75) {$\Lambda_{k+1} $};
		\node[left] at (2.75,1.85) {$\Omega_{k+1} $};
		\node at (1.5,1) {$ = $};
	\end{tikzpicture}
}
holds.
\end{defn}
\begin{rem}\label{xrelunique}
The exchange relation pair is unique separately in each variable, that is, if $ (\eta, \kappa_1) $ and $ (\eta , \kappa_2) $ (resp., $ (\eta_1, \kappa) $ and $ (\eta_2 , \kappa) $) both satisfy exchange relation, then $ \kappa_1 = \kappa_2 $ (resp., $ \eta_1 = \eta_2 $); this is because the connections are unitary and the functors $ \Gamma_k $ and $ \Delta_k $ are bi-faithful.
\end{rem}
We only require that the 2-cells satisfy this exchange relation \textit{eventually}. To make this precise, we let $$ \text{Ex}(\Lambda_{\bullet}, \Omega_{\bullet})$$ denote the space of sequences  $\{\eta^{(k)} \in \t{NT} \left(\Lambda_k , \Omega_k\right)\}_{k\geq 0}$ such that there exists an $N$ such that $(\eta_{k},\eta_{k+1})$ satifies the exchange relation for all $k\ge N$.
Consider the subspace
$$
\text{Ex}_{0}(\Lambda_{\bullet}, \Omega_{\bullet}):=\left\{\{\eta_{k}\}_{k\geq 0}  \in \text{Ex}(\Lambda_{\bullet}, \Omega_{\bullet}) :\ \eta_{k}=0 \ \text{ for all }\ k\ge N \  \t{  for some }  \ N \in \N  \right\}
$$
\begin{defn}\label{2cellpreC*}
Let $ \Lambda_\bullet  , \Omega_\bullet \in \textbf{UC}_1  \left(\Gamma_\bullet , \Delta_\bullet\right)$.
We define the space of $ 2 $-cells $$\textbf{UC}_2 \left(\Lambda_\bullet , \Omega_\bullet \right):=  \frac {\text{Ex}(\Lambda_{\bullet}, \Omega_{\bullet})}  {\text{Ex}_{0}(\Lambda_{\bullet}, \Omega_{\bullet})} $$
\end{defn}
For notational convenience, instead of denoting a $ 2 $-cell by an equivalence class of sequences, we simply use a sequence in the class and truncate upto a level after which the exchange relation holds for every consecutive pair, namely, $ \left\{ \eta^{(k)}\right\}_{k \geq N} \in \textbf{UC}_2 \left(\Lambda_\bullet , \Omega_\bullet \right) $ where $(\eta_{k},\eta_{k+1})$ satifies the exchange relation for all $k\ge N$.

If $ \ul \eta = \left\{ \eta^{(k)}\right\}_{k \geq K} \in \textbf{UC}_2 \left(\Lambda_\bullet , \Omega_\bullet \right) $ and $ \ul \kappa = \left\{ \kappa^{(k)}\right\}_{k \geq L} \in \textbf{UC}_2 \left( \Omega_\bullet , \Xi_\bullet \right) $, then define the `\textit{vertical}' composition of $ 2 $-cells by $\ul \kappa \cdot  \ul \eta  \coloneqq \left\{ \left(\kappa^{(k)} \circ \eta^{(k)} \right) \right\}_{k\geq \max\{K,L\}} $.
It is easy to check that  $\ul \kappa \cdot  \ul \eta \in\textbf{UC}_2 \left( \Lambda_\bullet , \Xi_\bullet \right) $ is well defined and the composition is associative.

Given two $ 0 $-cells $ \Gamma_\bullet  $ and $ \Delta_\bullet $, we have obtained a category whose object space consists of $ 1 $-cells $ \Lambda_\bullet $, and morphisms are given by $ 2 $-cells.
We call this the \textit{category of unitary connections from $ \Gamma_\bullet  $ to $ \Delta_\bullet $} and denote by $\textbf{UC}\,_{ \Gamma_\bullet, \Delta_\bullet   } $.

Following with the structure in the previous subsections, we will see that 2-cells uniquely define bimodule intertwiners between the bimodules associated to the 1-cells.
We will borrow the notations $ H_k$, $ H_\infty $ $\mscr S$, etc.  (arising out of $ \Lambda_\bullet  $) from previous subsections, and for those arising out of $ \Omega_\bullet $, we will use the notation $ G_k$, $ G_\infty$, $\mscr T$, etc. and we will also work with the pictures as before.
For each $ k \geq 0 $, we define $ \mcal N_k \left(\Lambda_k \Gamma_k \cdots \Gamma_1 m_0 \ , \ \Omega_k \Gamma_k \cdots \Gamma_1 m_0\right) \ni \gamma \os{\displaystyle \Phi}{\longmapsto} \Phi_\gamma \in \mcal L \left( H_\infty , G_\infty \right)$ (the space of adjointable operators with respect to the $ B_\infty $-valued inner product) in the following way
\begin{equation}\label{Phidef}
	H_\infty \supset H_{k+l} \ni \alpha \os{\displaystyle \Phi_\gamma}{\longmapsto} \ 
	\raisebox{-12mm}{
		\begin{tikzpicture}
			\node[right] at (-.1,0.1) {$ n_0 $};
			\node[right] at (-.1,1.1) {$ m_0 $};
			\node[right] at (-.1,2.1) {$ m_0 $};
			\node[left] at (-2.2,2.3) {$ \Omega_{k+l} $};
			\node[left] at (-2.1,1.1) {$ \Lambda_{k+l} $};
			\node[draw,thick,rounded corners, fill=white,minimum width=80] at (-1.1,.65) {$\alpha$};
			\node[draw,thick,rounded corners, fill=white,minimum width=40] at (-.5,1.6) {$\gamma$};
			\node[left] at (-.1,2.1) {$ \cdots $};
			\node[left] at (-1.3,1.6) {$ \cdots $};
			\node[left] at (-.1,1.1) {$ \cdots $};
			\node[left] at (-1,.2) {$ \cdots $};
			\node[left] at (-.1,.2) {$ \cdots $};
			\begin{scope}[on background layer]
			\draw[thick,dashed] (0,0) to (0,2.3);
				\draw[in=-90,out=90] (-.2,0) to (-.2,2.3);
				\draw[in=-90,out=90] (-.9,0) to (-.9,2.3);
				\draw[in=-90,out=90] (-1.1,0) to (-1.1,.9);
				\draw[in=-90,out=90] (-1.1,.9) to (-1.4,1.3);
				\draw[in=-90,out=90] (-1.4,1.3) to (-1.4,1.9);
				\draw[in=-90,out=90] (-1.4,1.9) to (-1.1,2.3);
				\draw[in=-90,out=90] (-1.8,0) to (-1.8,.9);
				\draw[in=-90,out=90] (-1.8,.9) to (-2.2,1.3);
				\draw[in=-90,out=90] (-2.2,1.3) to (-2.2,1.9);
				\draw[in=-90,out=90] (-2.2,1.9) to (-1.8,2.3);
				\draw[white,line width=1mm,in=-90,out=90] (-1.1,1.9) to (-2.2,2.3);
				\draw[red,in=-90,out=90] (-1.1,1.9) to (-2.2,2.3);
				\draw[red,in=-90,out=90] (-1.1,1.3) to (-1.1,1.9);
				\draw[white,line width=1mm,in=-90,out=90] (-2.2,.9) to (-1.1,1.3);
				\draw[red,in=-90,out=90] (-2.2,.9) to (-1.1,1.3);
			\end{scope}
		\end{tikzpicture}
	} \in G_{k+l} \subset G_\infty 
\end{equation}
for $ l\geq 0 $.
It is easy to check that $ \Phi_\gamma $ is well-defined and adjointable.
We list a few basic properties of $ \Phi $ in the following lemma.
\begin{lem}\label{PhiproppreC*} For all $ k \geq 0 $ and $ \gamma \in \mcal N_k \left(\Lambda_k \Gamma_k \cdots \Gamma_1 m_0 \ , \ \Omega_k \Gamma_k \cdots \Gamma_1 m_0\right)  $, the following conditions hold
	
	(i) $ \Phi_{\gamma^*} = (\Phi_\gamma)^* $,
	
	(ii) $ \Phi_\gamma (H_l) \subset G_l$ for all $ l \geq k $,
	
	(iii) the map $ \gamma \longmapsto \left. \Phi_\gamma \right|_{H_k}  $ is one-to-one, and
	
	(iv) $ \Phi_\gamma \in \mcal L_{B_\infty} (H_\infty , G_\infty) $.
\end{lem}
\begin{proof}
The only nontrivial part is to prove (iii).
This easily follows from the equality $  \gamma = \displaystyle \sum_{\sigma \in \mscr S}
\raisebox{-14mm}{
	\begin{tikzpicture}
		\node[draw,thick,rounded corners, fill=white,minimum width=60] at (-.8,2.1) {$\Phi_\gamma \, \sigma$};
		\node[draw,thick,rounded corners, fill=white] at (-.15,.9) {$\sigma^*$};
		\node[left] at (-.6,.9) {$ \cdots $};
		\node[right] at (-.1,2.7) {$ m_0 $};
		\node[right] at (-.1,.2) {$ m_0 $};
		\node[right] at (-.1,1.4) {$ n_0 $};
		\node[left] at (-1.5,2.7) {$ \Omega_k $};
		\node[left] at (-1.5,.2) {$ \Lambda_k $};
		\node[left] at (-.3,2.7) {$ \cdots $};
		\draw[white,line width=1mm,in=-90,out=90] (-1.6,0) to (-.3,.6);
		\draw[red,in=-90,out=90] (-1.6,0) to (-.3,.6);
		\begin{scope}[on background layer]
			\draw[in=-90,out=90] (-.5,0) to (-.8,1.75);
			\draw[in=-90,out=90] (-1.3,0) to (-1.6,1.75);
			\draw[thick,dashed] (0,0) to (0,3);
			\draw[red,in=-90,out=90] (-1.6,2.4) to (-1.6,3);
			\draw (-1.3,2.4) to (-1.3,3);
			\draw (-.3,2.4) to (-.3,3);
		\end{scope}
	\end{tikzpicture}
}$.
In fact, we have deduced a stronger statement, namely, $ \gamma $ is nonzero if and only if $ \left. \Phi_\gamma \right|_{H_0} $ is nonzero.
\end{proof}
\begin{lem}\label{AkcommNTpreC*}
	For each $ k \geq 0 $, the space $ \mcal N_k \left(\Lambda_k \Gamma_k \cdots \Gamma_1 m_0 \ , \ \Omega_k \Gamma_k \cdots \Gamma_1 m_0\right) $ gets an $ A_k $-$ A_k $-bimodule structure via
\[
\begin{split}
(\alpha_1 , \gamma &, \alpha_2) \in A_k \times  \mcal N_k \left(\Lambda_k \Gamma_k \cdots \Gamma_1 m_0 \ , \ \Omega_k \Gamma_k \cdots \Gamma_1 m_0\right) \times A_k \\
\rotatebox{270}{$ \longmapsto $}& \\
\Omega_k \alpha_1 \circ \gamma &\circ \Lambda_k \alpha_2 \in \mcal N_k \left(\Lambda_k \Gamma_k \cdots \Gamma_1 m_0 \ , \ \Omega_k \Gamma_k \cdots \Gamma_1 m_0\right)
\end{split}\]
	and the space $\normalfont \t{NT}(\Lambda_k , \Omega_k) $ of natural linear transformations is isomorphic to the space of $ A_k $-central vectors in $ \mcal N_k \left(\Lambda_k \Gamma_k \cdots \Gamma_1 m_0 \ , \ \Omega_k \Gamma_k \cdots \Gamma_1 m_0\right) $ via $ \eta \longmapsto \eta_{\Gamma_k\cdots \Gamma_1 m_0} $.
\end{lem}
\begin{proof}
The map
\begin{align*}
&\hspace*{90mm}\gamma \in \mcal N_k \left(\Lambda_k \Gamma_k \cdots \Gamma_1 m_0 \ , \ \Omega_k \Gamma_k \cdots \Gamma_1 m_0\right) \\
&\hspace*{90mm}\rotatebox{270}{$ \longmapsto  $}  \\
&\left(\hspace*{-1mm} \sum_{v \in V_{\mcal M_k}} \hspace*{-2mm} \left[\dim_\C \left( \mcal M_k (v , \Gamma_k \cdots \Gamma_1 m_0) \right)\right]^{-1} \hspace*{-6mm} \sum_{\substack{\sigma \in \t{ONB} \left(v, x \right)\\ \tau \in \t{ONB} \left(v, \Gamma_k \cdots \Gamma_1 m_0 \right) } } \hspace*{-6mm} \Omega_k (\sigma \tau^*) \circ \gamma  \circ \Lambda_k (\tau \sigma^*) \right)_{\hspace*{-2mm}x \in \t{Ob} (\mcal M_k)} \hspace*{-.5cm} \in \t{NT}(\Lambda_k , \Omega_k)
\end{align*}
when restricted to the $ A_k $-central vectors, turns out to be the inverse of the map in the statement of the lemma (since $ m_0 $ contains every simple as a subobject and $ \Gamma_j $'s are bi-faithful).
\end{proof}

\begin{lem}\label{xrelphi}
The pair $ (\eta , \kappa) \in \normalfont \t{NT}(\Lambda_k , \Omega_k) \times \normalfont \t{NT}(\Lambda_{k+1} , \Omega_{k+1}) $ satisfies exchange relation if and only if $ \Phi_{\eta_{\Gamma_k \cdots \Gamma_1 m_0}} = \Phi_{\kappa_{\Gamma_{k+1} \cdots \Gamma_1 m_0}} $
\end{lem}
\begin{proof}
The `only if' part direct follows from the definitions.\\
For the `if' part, let $ \vphantom{\eta}^\times \eta $ and $ \kappa_\times $ denote the left and the right sides of the exchange relation equation.
Applying \Cref{PhiproppreC*} (iii) on the equation in our hypothesis, we deduce $ \left(\vphantom{\eta}^\times \eta\right)_{\Gamma_k \cdots \Gamma_1 m_0} = \left(\kappa_\times\right)_{\Gamma_k \cdots \Gamma_1 m_0} $.
Now, by bi-faithfulness, $ {\Gamma_k \cdots \Gamma_1 m_0} $ contains all the simples of $ \mcal M_{k} $ as subobjects, and thereby $ \vphantom{\eta}^\times \eta = \kappa_\times $.
\end{proof}
\begin{thm}\label{flatthm}
Starting from a $ 2 $-cell $ \left\{\eta^{(k)} \in \normalfont \t{NT}\left(\Lambda_k , \Omega_k \right)\right\}_{k\geq K} $, we have an intertwiner $ \Phi_{\eta^{(k)}_{\Gamma_k \cdots \Gamma_1 m_0}} \in \vphantom{\mcal L}_{A_\infty} \mcal L_{B_\infty} (H_\infty, G_\infty) $ which is independent of $ k \geq K $.

Conversely, for every $ T \in \vphantom{\mcal L}_{A_\infty} \mcal L_{B_\infty} (H_\infty, G_\infty) $ ($ = $ the space of $ A_\infty $-$ B_\infty $-linear adjointable operator) and for all $ k \geq K_T\coloneqq \min \left\{l : T H_0 \subset G_l \right\} $, there exists unique $ \eta^{(k)} \in \normalfont \t{NT} (\Lambda_k , \Omega_k) $ such that $ T  =  \Phi_{\eta^{(k)}_{\Gamma_k \cdots \Gamma_1 m_0}} $.
Further, $ \left(\eta^{(k)} , \eta^{(k+1)}\right) $ satisfies exchange relation for all $ k \geq K_T $.
\end{thm}
\begin{proof}
The forward direction trivially follows from \Cref{xrelphi} and the $ A_\infty $-$ B_\infty $-linearity is obvious.
For the converse, set
\[
\zeta_k \coloneqq  \displaystyle \sum_{\sigma \in \mscr S}
\raisebox{-14mm}{
	\begin{tikzpicture}
		\node[draw,thick,rounded corners, fill=white,minimum width=60,minimum height=20] at (-.8,2.1) {$T\sigma$};
		\node[draw,thick,rounded corners, fill=white] at (-.15,.9) {$\sigma^*$};
		\node[left] at (-.6,.9) {$ \cdots $};
		\node[right] at (-.1,2.7) {$ m_0 $};
		\node[right] at (-.1,.2) {$ m_0 $};
		\node[right] at (-.1,1.4) {$ n_0 $};
		\node[left] at (-1.5,2.7) {$ \Omega_k $};
		\node[left] at (-1.5,.2) {$ \Lambda_k $};
		\node[left] at (-.3,2.7) {$ \cdots $};
		\draw[white,line width=1mm,in=-90,out=90] (-1.6,0) to (-.3,.6);
		\draw[red,in=-90,out=90] (-1.6,0) to (-.3,.6);
		\begin{scope}[on background layer]
			\draw[in=-90,out=90] (-.5,0) to (-.8,1.75);
			\draw[in=-90,out=90] (-1.3,0) to (-1.6,1.75);
			\draw[thick,dashed] (0,0) to (0,3);
			\draw[red,in=-90,out=90] (-1.6,2.4) to (-1.6,3);
			\draw (-1.3,2.4) to (-1.3,3);
			\draw (-.3,2.4) to (-.3,3);
		\end{scope}
	\end{tikzpicture}
}	
\in \mcal N \left(\Lambda_k \Gamma_k \cdots \Gamma_1 m_0 , \Omega_k \Gamma_k \cdots \Gamma_1 m_0\right) \t{ for } k\geq K_T 
\] where $ T \sigma $ is treated as an element of $ G_k $ and $ \mscr S $ ($ \subset H_0 $) is a PP-basis for the right $ B_\infty  $-module $ H_\infty $.
Using the right $ B_\infty $-valued inner product, the PP-basis and right $ B_k $-linearity of $ T $, one can easily conclude $ T \xi = \Phi_{\zeta_k} \xi$ for all $ \xi \in H_k $; moreover, this equation uniquely determines $ \zeta_k $ by \Cref{PhiproppreC*} (iii).
Further, the left side of the equation is $ A_k $-linear; then so is the right side.
Again by \Cref{PhiproppreC*} (iii), $ \zeta_k $ becomes $ A_k $-central.
Applying \Cref{AkcommNTpreC*}, we get a unique $ \eta^{(k)} \in \t{NT}(\Lambda_k , \Omega_k)$ satisfying $ \zeta_k = \eta^{(k)}_{\Gamma_k \cdots \Gamma_1 m_0} $.
The rest of the proof is straight forward.
\end{proof}
\begin{rem}\label{UCRcorrfunctor}
If $ \mcal C_{B_\infty,A_\infty} $ denotes the category of $ A_\infty $-$ B_\infty $ right correspondences where $ A_\infty $ and $ B_\infty $ are the unital filtration of finite dimesnional C*-algebras associated to the $ 0 $-cells $ \Gamma_\bullet$ and $ \Delta_\bullet $ respectively, then combining \Cref{flatthm}, \Cref{2cellpreC*} and  the definition of the vertical composition of $ 2 $-cells, we have a fully faithful $ * $-linear functor from
\[
\Psi_{\Gamma_\bullet , \Delta_\bullet}	: \textbf{UC}_{ \Gamma_\bullet , \Delta_\bullet} \lra \mcal C_{B_\infty,A_\infty} \; .
\]
\end{rem}
\vspace*{4mm}

\subsection{The horizontal structure}$ \ $

This is the final step of constructing a $ * $-linear $ 2 $-category of unitary connections, denoted by $\textbf{UC}$ whose $ 0 $-, $ 1 $- and $ 2 $-cells are already defined in Sections \ref{0cellspreC*}, \ref{1cellspreC*} and \ref{2cellspreC*} respectively.
For $ 0 $-cells $\Gamma_\bullet , \Delta_\bullet, \Sigma_\bullet $, we will define a bifunctor
\[ \btimes\, : \, \textbf{UC}_{\Delta_\bullet ,  \Sigma_\bullet } \times \textbf{UC}_{\Gamma_\bullet ,\Delta_\bullet } \lra \textbf{UC}_{ \Gamma_\bullet, \Sigma_\bullet }
\]
in such a way that it corresponds to the  reverse relative tensor product of the associated right correspondences.
For $\Omega_\bullet\in \textbf{UC}_1 \left(\Delta_\bullet ,\Sigma_\bullet\right) $ and $\Lambda_\bullet  \in  \textbf{UC}_1 \left(\Gamma_\bullet ,\Delta_\bullet \right) $, define
\begin{equation}\label{tensor1cell}
\Omega_\bullet \btimes \Lambda_\bullet \coloneqq \left( \left\{\Omega_k \, \Lambda_k\right\}_{k\geq 0} \, , \,  \left\{\raisebox{-1.35cm}{
\begin{tikzpicture}
\draw[in=-90,out=90] (0,0) to (1,2);
\draw[white,line width=1mm,in=-90,out=90] (0.5,0) to (0,2);
\draw[red,in=-90,out=90] (0.5,0) to (0,2);
\draw[white,line width=1mm,in=-90,out=90] (1,0) to (0.5,2);
\draw[red,in=-90,out=90] (1,0) to (0.5,2);
\node[left] at (0.1,0.1) {$ \Sigma_k $};
\node[left] at (1.2,-0.1) {$\Omega_{k-1} $};
\node[right] at (.9,0.1) {$ \Lambda_{k-1} $};
\node[right] at (.9,1.8) {$ \Gamma_k$};
\node[right] at (.2,2.15) {$\Lambda_k$};
\node[left] at (.1,1.8) {$ \Omega_k$};
\end{tikzpicture}
}\right\}_{k\geq 1} \right) \ .
\end{equation}
\begin{prop}\label{confusionbimodpreC*}
The bimodule $\Psi_{ \Gamma_\bullet , \Sigma_\bullet }\left( \Omega_\bullet \btimes \Lambda_\bullet  \right)$ is isomorphic to the relative tensor product of the right correspondences $\Psi_{ \Gamma_\bullet , \Delta_\bullet } \left(\Lambda_\bullet\right)  $ and $\Psi_{ \Delta_\bullet , \Sigma_\bullet } \left(\Omega_\bullet \right) $.
\end{prop} 
\begin{proof}
We first consider the following notations:\\
$ \left\{\mcal M_{k-1} \os{\displaystyle \Gamma_k} \lra \mcal M_k \right\}_{k\geq 1}  \leadsto \;\;\; \cdots \subset A_k = \t{End} (\Gamma_k \cdots \Gamma_1 m_0) \subset \cdots \subset \us{k\geq 0} \cup A_k =  A_\infty  $\\
$\left\{\mcal N_{k-1} \os{\displaystyle \Delta_k} \lra \mcal N_k \right\}_{k\geq 1}  \leadsto \;\;\; \cdots \subset B_k = \t{End} (\Delta_k \cdots \Delta_1 n_0) \subset \cdots \subset \us{k\geq 0} \cup B_k = B_\infty $\\
$\left\{\mcal Q_{k-1} \os{\displaystyle \Sigma_k} \lra \mcal Q_k \right\}_{k\geq 1} \leadsto \;\;\; \cdots \subset C_k = \t{End} (\Sigma_k \cdots \Sigma_1 q_0) \subset \cdots \subset \us{k\geq 0} \cup C_k = C_\infty $\\
$ \Lambda_\bullet \leadsto  \cdots \subset H_k =\! \mcal N_k (\Delta_k \cdots \Delta_1 n_0 , \Lambda_k \Gamma_k \cdots \Gamma_1 m_0) \subset \cdots \subset \us{k\geq 0} \cup H_k = H_\infty $\\
$ \Omega_\bullet \leadsto  \cdots \subset G_k = \mcal Q_k (\Sigma_k \cdots \Sigma_1 q_0 , \Omega_k \Delta_k \cdots \Delta_1 n_0) \subset \cdots \subset \us{k\geq 0} \cup G_k = G_\infty $\\
$ \Omega_\bullet \btimes \Lambda_\bullet \leadsto \cdots \subset  F_k = \mcal Q (\Sigma_k \cdots \Sigma_1 q_0 , \Omega_k \Lambda_k \Gamma_k \cdots \Gamma_1 m_0) \subset \cdots \subset \us{k\geq 0} \cup F_k = F_\infty $\\
Consider the linear map
\[
H_k \otimes G_k \ni \xi \otimes \zeta \lra \Omega_k (\xi ) \circ \zeta =
\raisebox{-1.5cm}{
\begin{tikzpicture}
\draw[dashed,thick] (1,0)to (1,2.9);
\draw[in=-90,out=90] (0,0)to (0,2.9);
\draw[in=-90,out=90] (.8,0)to (.8,2.9);
\draw[red,in=-90,out=90] (-.2,2.3)to (-.2,2.9);
\draw[red,in=-90,out=90] (-.2,1.1) to (-1,2.5) to (-1,2.9);
\node[draw,thick,rounded corners, fill=white,minimum width=45] at (.4,2) {$\xi$};
\node[draw,thick,rounded corners, fill=white,minimum width=45] at (.4,.8) {$\zeta$};
\node[left] at (-.1,2.6) {$ \Lambda_k$};
\node[left] at (-.9,2.6) {$\Omega_k$};
\node[right] at (.9,.2) {$ q_0$};
\node[right] at (.9,1.35) {$ n_0$};
\node[right] at (.9,2.6) {$ m_0$};
\node[right] at (0,.2) {$ \cdots$};
\node[right] at (0,1.35) {$ \cdots$};
\node[right] at (0,2.6) {$ \cdots$};
\end{tikzpicture}
} \in F_k.
\]
It follows directly from the definition that the above map is $ A_k $-$ C_k $-linear and compatible with the inclusions $ H_k \hookrightarrow H_{k+1}$,  $G_k \hookrightarrow G_{k+1} $ and $F_k \hookrightarrow F_{k+1}  $; as a result, it extends to a $ A_\infty  $-$ C_\infty  $-linear map $ f: H_\infty \otimes G_\infty \lra F_\infty $.
Consider the $ C_\infty $-valued sesquilinear form on $ H_\infty \otimes G_\infty$ defined by
\[
\left(H_\infty \otimes G_\infty\right) \times \left(H_\infty \otimes G_\infty \right) \ni \left(\xi_1\otimes \zeta_1 \ , \ \xi_2 \otimes \zeta_2 \right) \longmapsto \left\lab \left\lab \xi_1 , \xi_2 \right\rab_{B_\infty} \zeta_1 ,\zeta_2 \right\rab_{C_\infty}  \in C_\infty
\]
which after applying $ f $, clearly goes to the desired one on the right correspondence $ G_\infty $.
Moreover, kernel of $ f $ matches exactly with the null space of the form (using non-degeneracy of the $ C_\infty $-valued inner product on $ F_\infty $).
Thus $ f $ factors through the relative tensor product and induces an injective $ A_\infty $-linear map.

Finally, we need to show that it is surjective as well.
For this, consider the PP-basis $ \mscr S $ (resp., $ {\mscr T} $)
sitting inside $ H_0 $ (resp., $ G_0 $)
for the right $ B_\infty $- (resp., $ C_\infty $- ) module $ H_\infty $ (resp., $ {G_\infty}$) considered in \Cref{PPbasispreC*}.
Note that $ \us{\sigma \in \mscr S} \sum \; \us{\tau \in \mscr {T}} \sum \Omega_0 (\sigma ) \circ \tau \circ \tau^* \circ \Omega_0 (\sigma^* ) = 1_{\Omega_0 \Lambda_0 m_0}$.
Thus by \Cref{PPbasispreC*}, $  \left\{\Omega_0 (\sigma ) \circ \tau \right\}_{(\sigma , \tau )\in \mscr S\times \mscr {T}} $ turns out to be a PP-basis for the right $ C_\infty $-module $ F_\infty $.
\end{proof}
We next proceed towards defining $ \btimes $ at the level of $ 2 $-cells.
\begin{defn}
For $ \Omega^i_\bullet \in \textbf{UC}_1 \left( \Delta_\bullet , \Sigma_\bullet \right) $ and $ \Lambda^i_\bullet \in \textbf{UC}_1 \left( \Gamma_\bullet , \Delta_\bullet \right) $
where $ i=1,2 $, and $ 2 $-cells $ \ul \eta = \left\{\eta^{(k)}\right\}_{k\geq K} \in \textbf{UC}_2 \left(\Lambda_\bullet^1  , \Lambda_\bullet^2 \right)$ and $ \ul \kappa = \left\{\kappa^{(k)}\right\}_{k\geq L} \in \textbf{UC}_2 \left(\Omega_\bullet^1  , \Omega_\bullet^2   \right)$, define $ \ul \kappa \btimes \ul \eta \in\textbf{UC}_2 \left( \Omega^1_\bullet  \btimes \Lambda^1_\bullet   \ , \ \Omega^2_\bullet  \btimes \Lambda^2_\bullet \right) $ by
\[
\left(\ul \kappa \btimes \ul \eta  \right)_k \coloneqq \  \Omega_{k}^2 \left( \eta^{(k)} \right) \circ \kappa^{(k)}_{\Lambda_{k}^1} \ = \ \kappa^{(k)}_{\Lambda_{k}^2} \circ  \Omega_{k}^1 \left( \eta^{(k)} \right) \ =
\raisebox{-10mm}{
\begin{tikzpicture}
\draw[red,in=-90,out=90] (0,-.35)to (0,1);
\draw[red,in=-90,out=90] (0,-1)to (0,-.35);
\draw[red,in=-90,out=90] (1.1,-.4)to (1.1,1);
\draw[red,in=-90,out=90] (1.1,-1)to (1.1,-.4);
\node[draw,thick,rounded corners, fill=white] at (0,0) {$\kappa^{(k)}$};
\node[draw,thick,rounded corners, fill=white] at (1.1,0) {$\eta^{(k)}$};
\node[left] at (0.1,-.7) {$ \Omega^1_k $};
\node[left] at (0.1,.65) {$ \Omega^2_k $};
\node[right] at (1.05,-.7) {$ \Lambda^1_k $};
\node[right] at (1.05,.65) {$ \Lambda^2_k $};
\end{tikzpicture}
}
\]
for $ k \geq \max\{K,L\} $.
(It is easy to check that every pair of consecutive terms in $ \ul \kappa \btimes \ul \eta $ satisfies the exchange relation, which is a requirement for being a $ 2 $-cell.)
\end{defn}
The compatibility of the vertical and horizontal compositions $ \cdot $ and $ \btimes $ between $ 2 $-cells follows easily from the pictures.
\begin{prop}
Continuing with the above set up,  $ \Psi_{\Gamma_\bullet,\Sigma_\bullet} \left(\ul \kappa\btimes  \ul \eta  \right) $ corresponds to the operator $ \Psi_{\Gamma_\bullet,\Delta_\bullet} \left(\ul \eta \right) \us {B_\infty} \otimes \Psi_{\Delta_\bullet , \Sigma_\bullet} \left(\ul \kappa \right)$ via the isomorphism of bimodules in \Cref{confusionbimodpreC*}.
\end{prop}
\begin{proof}
Let ${H^i_\infty}, G^i_\infty $ and $ F_\infty^i$ denote the $ A_\infty $-$ B_\infty $-,  $ B_\infty $-$ C_\infty $- and $ A_\infty $-$ C_\infty $-right correspondences $ \Psi_{ \Gamma_\bullet ,   \Delta_\bullet } (\Lambda_\bullet^i)$, $  \Psi_{ \Delta_\bullet  ,  \Sigma_\bullet } \left(\Omega_\bullet^i \right)$ and $ \Psi_{\Gamma_\bullet , \Sigma_\bullet } \left(\Omega_\bullet^i \btimes \Lambda_\bullet^i \right)$ for $ i = 1,2 $ respectively.
	
\noindent Set $ T \coloneqq \Psi_{ \Gamma_\bullet , \Delta_\bullet } (\ul \eta) \in \vphantom{\mcal L}_{A_\infty} \mcal L_{B_\infty} (H^1_\infty , H^2_\infty) $ and $ S \coloneqq \Psi_{ \Delta_\bullet  ,  \Sigma_\bullet } (\ul \kappa) \in \vphantom{\mcal L}_{B_\infty} \mcal L_{C_\infty} (G^1_\infty , G^2_\infty)$.
Suppose $ X $ denote the intertwiner in $ \vphantom{\mcal L}_{A_\infty} \mcal L_{C_\infty} (F^1_\infty,F^2_\infty) $ induced by $ T\us {B_\infty} \otimes S $ under the isomorphism in \Cref{confusionbimodpreC*}.
For $ k \geq \max \{K,L \}$ and $ \xi \in H^1_k,\ \zeta \in G^1_k $, applying $ X $ on the element corresponding to the basic tensor $ \xi \us {B_\infty} \otimes \zeta $ (via \Cref{confusionbimodpreC*}), we get
\[
X \left(\Omega_k^1 (\xi) \circ \zeta\right)
= \Omega^2_k \left( \Phi_{\eta^{(k)}_{\Gamma_k \cdots \Gamma_1 m_0} } (\xi)\right) \circ \left[ \Phi_{\kappa^{(k)}_{\Delta_k \cdots \Delta_1 q_0} } (\zeta) \right]
=\raisebox{-20mm}{
\begin{tikzpicture}
\draw[dashed,thick] (1,0)to (1,4);
\draw[in=-90,out=90] (0,0)to (0,4);
\draw[in=-90,out=90] (.8,0)to (.8,4);
\draw[red,in=-90,out=90] (-.2,2.3) to (-.6,2.8) ;
\draw[red,in=-90,out=90]  (-.6,2.8) to (-.6,4);
\draw[red,in=-90,out=90] (-.2,1.1) to (-1.7,2.85) ;
\draw[red,in=-90,out=90] (-1.7,2.85) to (-1.7,4);
\node[draw,thick,rounded corners, fill=white,minimum width=45] at (.4,2) {$\xi$};
\node[draw,thick,rounded corners, fill=white,minimum width=45] at (.4,.8) {$\zeta$};
\node[draw,thick,rounded corners, fill=white] at (-.6,3.2) {$\eta^{(k)}$};
\node[draw,thick,rounded corners, fill=white] at (-1.7,3.2) {$\kappa^{(k)}$};
\node[left] at (-.4,2.4) {$ \Lambda^1_k$};
\node[left] at (-.55,3.85) {$ \Lambda^2_k$};
\node[left] at (-1.2,2) {$\Omega^1_k$};
\node[left] at (-1.65,3.8) {$\Omega^2_k$};
		\node[right] at (.9,.2) {$ q_0$};
		\node[right] at (.9,1.35) {$ n_0$};
		\node[right] at (.9,3.2) {$ m_0$};
		\node[right] at (0,.2) {$ \cdots$};
		\node[right] at (0,1.35) {$ \cdots$};
		\node[right] at (0,3.2) {$ \cdots$};
	\end{tikzpicture}
}
\]
$  = \Psi_{\Gamma_\bullet , \Sigma_\bullet} \left(\ul \kappa \btimes \ul \eta \right) \left(\Omega_k^1 (\xi) \circ \zeta\right)$.
\end{proof}
We summarize the above finding in the following theroem.
\begin{thm}
$ \Psi $ is a $ * $-linear, fully faithful, tensor-reversing $ 2 $-functor from the $ 2 $-category of unitary connections $ \normalfont \textbf{UC} = \{\normalfont \textbf{UC}_{\Gamma_\bullet  , \Delta_\bullet} : \Gamma_\bullet, \Delta_\bullet \t{ are } 0\t{-cells} \} $ to the $ 2 $-category of right correspondence over pairs of AFD pre-C*-algebras.
\end{thm}
\begin{rem}\label{completion}
The $ * $-algebras $ A_\infty $, $ B_\infty $ associated to $ 0 $-cells $ \Gamma_\bullet $, $ \Delta_\bullet $ can be completed using their unique C*-norm, and obtain the C*-algebras $ A $, $ B $ respectively.
Then, the $ A $-$ B $ right correspondence associated to the $ 1 $-cell $ \Lambda_\bullet  $ will be the completion $ H $ of the space $ H_\infty $ with respect to the norm $ \norm{\xi}_{\t{C*}} \coloneqq \displaystyle \sqrt{\norm{\left\lab \xi, \xi \right \rab_B }} $.
The PP-basis $ \mscr S $ for the right $ B_\infty  $-module $ H_\infty$ continue to be so for the right $ B $-module $ H $.
As a result, $ H $ as a $ B $-module becomes isomorphic to $ q \left[ \C^{\mscr S} \otimes B \right] $ where the right $ B $-action on the latter module is the diagonal one and $ q  $ is the projection $\displaystyle \sum_{\sigma_1 , \sigma_2 \in \mscr S} E_{\sigma_1 , \sigma_2}  \otimes \left\lab \sigma_2 , \sigma_1\right\rab_B $ in the C*-algebra $ M_{\mscr S \times \mscr S} \otimes B$.
The left $ A $-action on $ H $ will translate into a $ * $-homomorphism $\Pi: A \lra q \left[ M_{\mscr S \times \mscr S} \otimes B \right]q $ giving rise to an $ A $-action on $ q \left[\C^{\mscr S} \otimes B\right] $.
Now, at the level of $ 2 $-cells from the $ 1 $-cell $ \Lambda_\bullet $ to $ \Omega_\bullet $, the obvious candidates that come up are the adjointable $ A $-$ B $ intertwiners; these are in one-to-one correspondence with elements in $ s \left[ M_{\mscr T \times \mscr S} \otimes B \right] q $ which intertwines $ \Pi (a) $ and $ \widetilde \Pi (a) $ for $ a \in A $ (where $ \mscr T, s, \widetilde \Pi $ are related to the $ 1 $-cell $ \Omega_\bullet$ in the same way as $ \mscr S , q ,\Pi $ are related to $ \Lambda_\bullet$ respectively).
However, there is no apparent interpretation of such $ 2 $-cells in terms of natural transformations between $ \Lambda_k $'s and $ \Omega_k $'s compatible with the $ W_k $'s as shown in \Cref{flatthm}.
\end{rem}

%% file: tracialcase.tex
\section{The tracial case}\label{unicontr}
Locally semisimple algebras equipped with a tracial state, extend to finite von Neumann algebras. Hyperfinite subfactor reconstruction works by passing from the algebraic category described in the previous section to von Neumann algebras, and showing that for special cases arising in finite index subfactor theory, it is fully faithful.
A natural question is to figure out what happens in our more general setting.

To make this question precise, we consider modifications of the 2-category $\textbf{UC}$ at every level.
At the level of $ 0$-cells, we will be considering Bratteli diagrams with extra tracial data, and make necessary adjustments to the $ 1 $- and $ 2 $-cells in order to ``preserve" this extra structure.
Our particular choice of adjustments admittedly appears ad hoc, but it is the condition that was required to make our functor work in the forthcoming proofs.
We define the $2$-category $\textbf{UC}^{\t{tr}}$ as follows:
\begin{itemize}
\item
\textbf{0-cells.} The $ 0 $-cells are given by pairs $ \left( \Gamma_\bullet , \ul \mu^\bullet \right) $ where $ \Gamma_\bullet $ is a $ 0 $-cell $ \left\{ \mcal M_{k-1} \os{\displaystyle \Gamma_k}{\lra} \mcal M_{k} \right\}_{k\geq 0} $
in \textbf{UC}, and $ \ul \mu^\bullet $ denotes weight vectors $ \ul \mu^k = (\mu^k_v)_{v \in V_{\mcal M_k} }  $ with positive entries satisfying $ \us {v\in V_{\mcal M_0} } \sum  \mu^0_v = 1$ and $ \left(\Gamma_k\right)' \ul \mu^k =\ul \mu^{k-1}$ all $ k\geq 1 $ (where we use the same symbol for the functor and its adjacency matrix).
In other words (recasting in terms of semisimple categories and functors),
the data of a $ 0 $-cell in $\textbf{UC}^{\t{tr}}$ is a sequence of weighted finitely semisimple C*-categories $ \left\{\left(\mcal M_k , \ul \mu^k \right) \right\}_{k\geq 0} $
along with a sequence of $ * $-linear, bi-faithful functors $ \Gamma_k : \mcal M_{k-1} \ra \mcal M_k $ such that the tracial solution (say $ \left(\rho_k : \t{id}_{\mcal M_k}\ra\Gamma_k\ \Gamma'_k \ , \ \rho'_k : \t{id}_{\mcal M_{k-1}}\ra\Gamma'_k\ \Gamma_k  \right)$ where $ \Gamma'_k $ is adjoint to $ \Gamma_k $) to the conjugate equations commensurate with the weight functions $ \left(\ul \mu^{k-1} , \ul \mu^k \right) $ satisfies
\begin{equation}\label{trivial-loop-anticlock}
	\left(\rho'_k\right)^*_\bullet \circ  \left(\rho'_k\right)_\bullet = 1_\bullet
\end{equation}
which is equivalent to the matrix equation $ \left(\Gamma_k\right)' \ul \mu^k =\ul \mu^{k-1}$ (via \Cref{loop-anticlock}). The purpose of the equation $ \us {v\in V_{\mcal M_0}} \sum  \mu^0_v = 1$ is to normalize scaling.
For simplicity, we will denote the $ 0 $-cell of $\textbf{UC}^{\t{tr}}$ by $ \Gamma_\bullet $.
\vspace*{4mm}
\item
\textbf{1-cells.} A $ 1 $-cell in $\textbf{UC}^{\t{tr}}$ from the $ 0 $-cell $ \left\{ \left(\mcal M_{k-1}, \ul{\mu}^{k-1}\right) \os{\displaystyle \Gamma_k}{\lra} \left(\mcal M_{k}, \ul{\mu}^{k}\right)  \right\}_{k\geq 0} $ to $ \left\{ \left(\mcal N_{k-1}, \ul{\nu}^{k-1}\right) \os{\displaystyle \Delta_k}{\lra} \left(\mcal N_{k}, \ul{\nu}^{k}\right)  \right\}$ is given by a $ 1 $-cell $ \Lambda_\bullet $ in $ \textbf{UC}_1 \left(\Gamma_\bullet , \Delta_\bullet \right) $ such that there exists $ \epsilon , M >0 $ satisfying the boundedness condition:
\begin{equation}\label{bndonwt}
\epsilon \  {\mu^{k}_v} \ \leq \ {\left[\Lambda'_k\ \ul \nu^{k}\right]_v} \ \leq \ M \ {\mu^{k}_v} \ \t{ for all } \ {k\geq 0 }  ,  \ {v \in V_{\mcal M_k} } \ .
\end{equation}
\vspace*{4mm}
\item \textbf{Tensor of $ 1 $-cells.}
For $ 0 $-cells $\Gamma_\bullet , \Delta_\bullet, \Sigma_\bullet$, we will define a map
\[ \btimes\, : \, \textbf{UC}^\t{tr}_1 \left(\Delta_\bullet, \Sigma_\bullet\right) \times \textbf{UC}^\t{tr}_1 \left(\Gamma_\bullet, \Delta_\bullet\right) \lra \textbf{UC}^\t{tr}_1 \left(\Gamma_\bullet, \Sigma_\bullet\right)
\]
exactly the same as that for $ \textbf{UC}_1 $ given by \Cref{tensor1cell}; however, we have to check whether \Cref{bndonwt} is satisfied by $\Omega_\bullet \boxtimes \Lambda_\bullet$, where $\Omega_\bullet \in \textbf{UC}^\t{tr}_1(\Delta_\bullet, \Sigma_\bullet)$ and $\Lambda_\bullet \in \textbf{UC}^\t{tr}_1(\Gamma_\bullet,\Delta_\bullet)$.
Suppose we have,
\[\epsilon \mu^k_u \leq [\Lambda'_{k} \, \ul{\nu}^k]_u \leq M \mu^k_u  \ \t{ and } \ \delta \nu^k_v \leq [\Omega'_{k} \, \ul{\pi}^k]_v \leq N \nu^k_v \  \t{ for each } \ k \geq 0, u \in V_{\mcal M_k} ,v \in V_{\mcal N_k} . \]
Applying $ \Lambda'_k $ on the second set of inequalities, we get $\epsilon \delta \mu^k_u \leq [\Lambda'_{k} \, \Omega'_{k} \, \ul{\pi}^k]_u \leq MN \mu^k_u$ for each $k \geq 0, u \in V_{\mcal M_k}$.
Thus, \Cref{bndonwt} is satisfied for $\Omega_\bullet \boxtimes \Lambda_\bullet$.
\vspace*{4mm}
\item
\textbf{2-cells.}
Consider two $ 1 $-cells $ \Lambda_\bullet , \Omega_\bullet \in \textbf{UC}^{\t{tr}}_1 \left( \left(\Gamma_\bullet , \ul \mu^\bullet  \right) , \left(\Delta_\bullet , \ul \nu^\bullet  \right) \right)$.
The $ 2 $-cells in $ \textbf{UC}_2 \left( \Lambda_\bullet , \Omega_\bullet \right) $, that is, the sequences eventually satisfying the exchange relation at every level, do not use the extra data of $ \ul \mu^\bullet $ and $ \ul \nu^\bullet $.
We introduce the following tool which will generalize the exchange relation.
\begin{defn}
	The \textit{loop operator from} $ \Lambda_\bullet $ \textit{to} $\Omega_\bullet $ is the sequence of linear maps $ \left\{S_k : \t{NT}  (\Lambda_k , \Omega_k) \ra \t{NT}  (\Lambda_{k-1} , \Omega_{k-1}) \right\}_{k\geq 1} $ defined by
	\[
	\t{NT}  (\Lambda_k , \Omega_k) \ni \eta \ \os{\displaystyle S_k}{\longmapsto}
	\raisebox{-1.5cm}{
		\begin{tikzpicture}
			\draw (-.3,0) ellipse (7mm and 12mm);
			\draw[white,line width=1mm,in=-90,out=90] (.6,-1.5) to (0,-.29);
			\draw[red,in=-90,out=90] (.6,-1.5) to (0,-.29);
			\draw[white,line width=1mm,in=-90,out=90] (0,.29) to (.6,1.5);
			\draw[red,in=-90,out=90] (0,.29) to (.6,1.5);
			\node[left] at (.1,.55) {$ \Omega_k $};
			\node[right] at (.4,1.1) {$ \Omega_{k-1} $};
			\node[left] at (.2,-.55) {$\Lambda_k $};
			\node[right] at (.45,-1.2) {$ \Lambda_{k-1} $};
			\node[draw,thick,rounded corners, fill=white] at (0,0) {$\eta$};
			\node[right] at (.5,0) {$ \Gamma_k $};
			\node[right] at (-.3,-1.3) {$ \Delta_k $};
			\node[right] at (-.3,1.3) {$ \Delta_k $};
		\end{tikzpicture}
	}
	\in \t{NT}  (\Lambda_{k-1} , \Omega_{k-1}) \ .
	\]
where the cap and the cup come from tracial solution to the conjugate equation for the duality of $ \Delta_k : \mcal N_{k-1} \ra \mcal N_k $ commensurate with $ \left(\ul \nu^{k-1} , \ul \nu^k\right) $.
\end{defn}
We will encounter equations and inequalities involving multiple loop operators all of which might not have the same source $ 1 $-cell or the same target $ 1 $-cell in $ \textbf{UC}^{\textbf{tr}}_1 $; for notational convenience, we will simply use $ S_\bullet $, and from the context, it will be clear what the source and the targets are.
\begin{rem}\label{Sprop}
	The loop operator satisfies the following properties which are easy to derive:
	
	(i) $ S_k $ is unital when $ \Lambda_\bullet  = \Omega_\bullet $, (which follows from \Cref{trivial-loop-anticlock}),
	
	(ii) $ S_k \eta^* = (S_k \eta)^* $,
	
	(iii) $ S_k \eta^* \, \circ S_k\eta \leq S_k (\eta^* \eta) $ in the C*-algebra NT$ (\Lambda_{k-1},\Lambda_{k-1}) $ and the loop operator is a contraction.
\end{rem}
\begin{defn}\label{BQFS}
A sequence $\ul \eta =  \left\{\eta^{(k)} \in \t{NT} \left(\Lambda_k , \Omega_k\right) \right\}_{k\geq 0} $ will be referred as:
	
	(a) \textit{quasi-flat sequence from $ \Lambda_\bullet $ to $ \Omega_\bullet $} if it satisfies $ S_{k+1} \eta^{(k+1)} = \eta^{(k)} $ for all $ k\geq 0 $,
	
	(b) \textit{flat sequence from $ \Lambda_\bullet $ to $ \Omega_\bullet $} if it is quasi-flat and there exists $ K\in \N $ such that $ \left(\eta^{(k)} , \eta^{(k+1)}\right) $ satisfies the exchange relation (\Cref{XreldefpreC*}) for every $ k \geq K $.
\end{defn}
\begin{rem}\label{BQFSflat}
There is a one-to-one correspondence between flat sequences from $ \Lambda_\bullet $ to $ \Omega_\bullet $ , and the $ 2 $-cells in $ \textbf{UC}_2  \left(\Lambda_\bullet , \Omega_\bullet\right)$.
	Note that the exchange relation	
	\raisebox{-1.1cm}{
		\begin{tikzpicture}
			\draw[in=-90,out=90] (0,0) to (0,1.2);
			\draw[in=-90,out=90] (0,1.2) to (0.6,2);
			\draw[in=-90,out=90] (3.4,1) to (3.4,2);
			\draw[in=-90,out=90] (2.6,0) to (3.4,1);
			\draw[red,in=-90,out=90] (.6,0) to (0.6,1.2);
			\draw[white,line width=1mm,in=-90,out=90] (0.6,1.2) to (0,2);
			\draw[red,in=-90,out=90] (0.6,1.2) to (0,2);
			\draw[red,in=-90,out=90] (2.6,1) to (2.6,2);
			\draw[white,line width=1mm,in=-90,out=90] (3.4,0) to (2.6,1);
			\draw[red,in=-90,out=90] (3.4,0) to (2.6,1);
			\node[draw,thick,rounded corners, fill=white] at (0.6,0.8) {$\eta^{(k)}$};
			\node[draw,thick,rounded corners, fill=white] at (2.6,1.4) {$\eta^{(k+1)}$};
			\node[right] at (.5,0.1) {$ \Lambda_k $};
			\node[right] at (.5,1.45) {$\Omega_k $};
			\node[left] at (0.2,1.8) {$\Omega_{k+1} $};
			\node[right] at (3.3,0.1) {$ \Lambda_k $};
			\node[left] at (2.8,.75) {$\Lambda_{k+1} $};
			\node[left] at (2.65,2.05) {$\Omega_{k+1} $};
			\node at (1.5,1) {$ = $};
		\end{tikzpicture}
	}, unitarity of the connection and \Cref{trivial-loop-anticlock} yield the equation $ S_{k+1} \eta^{(k+1)} = \eta^{(k)} $.
	Thus every $ 2 $-cell $ \left\{ \eta^{(k)} \right\}_{k\geq K} \in \textbf{UC}_2 \left( \Lambda_\bullet  ,  \Omega_\bullet \right)$ extends to a unique quasi-flat sequence from $ \Lambda_\bullet $ to $ \Omega_\bullet $ by setting $ \eta^{(k)} \coloneqq S_{k+1} \cdots S_K\ \eta^{(K)} $ for $ k< K $.
Further, a flat sequence is bounded in C*-norm.
\end{rem}
\begin{defn}
A $ 2 $-cell in $ \textbf{UC}^\t{tr}_2 \left( \Lambda_\bullet , \Omega_\bullet\right) $ is given by a bounded (in C*-norm) quasi-flat sequence (abbreviated as `BQFS') from $ \Lambda_\bullet $ to $ \Omega_\bullet $.
\end{defn}
\vspace*{4mm}
\item
\textbf{Horizontal and vertical compositions of $ 2 $-cells.}
\begin{defn}\label{2cellcompdef}
(a) The \textit{vertical composition of the $ 2 $-cells} $ \ul \kappa \in \textbf{UC}^\t{tr}_2 \left( \Omega_\bullet , \Xi_\bullet \right)$ and $ \ul \eta \in \textbf{UC}^\t{tr}_2 \left( \Lambda_\bullet  , \Omega_\bullet \right)$ is defined as
\[
\left(\ul \kappa \cdot \ul \eta \right) \coloneqq \left\{\left(\ul \kappa \cdot \ul \eta \right)^{(k)} \coloneqq \displaystyle \lim_{l \to \infty} S_{k+1} \cdots S_{k+l} \left(\kappa^{(k+l)}  \circ\, \eta^{(k+l)}\right)\right\}_{ k\geq 0} \in \textbf{UC}^\t{tr}_2 \left( \Lambda_\bullet  , \Xi_\bullet \right) \ .
\]
(b) Let $ \Lambda_\bullet^i \in \textbf{UC}^\t{tr}_1 {\left(\Gamma_\bullet, \Delta_\bullet \right)}  $ and $ \Omega_\bullet^i \in \textbf{UC}^\t{tr}_1 {\left(\Delta_\bullet, \Sigma_\bullet\right)}$) for $ i =1,2 $.
Then, the \textit{horizontal composition} (or the \textit{tensor product}) \textit{of the $ 2 $-cells} $ \ul \eta = \left\{\eta^{(k)}\right\}_{k\geq 0} \in \textbf{UC}^\t{tr}_2 \left(\Lambda_\bullet^1 , \Lambda_\bullet^2  \right)$ and $ \ul \kappa = \left\{\kappa^{(k)}\right\}_{k\geq 0} \in \textbf{UC}^\t{tr}_2 \left(\Omega_\bullet^1 , \Omega_\bullet^2  \right)$ is given by
\[
\ul \kappa \btimes \ul \eta  \coloneqq 
\left\{ \left(\ul \kappa \btimes \ul \eta \right)^{(k)} \coloneqq \us{l \to \infty} \lim S_{k+1} \cdots S_{k+l} \left( \Omega_{k+l}^2 \left( \eta^{(k+l)} \right) \circ \kappa^{(k+l)}_{\Lambda_{k+l}^1} \right) \right\}_{k \geq 0} \!\!\!\! \in \textbf{UC}^\t{tr}_2 \left(\Omega_\bullet^1  \btimes \Lambda_\bullet^1    ,  \Omega_\bullet^2 \btimes \Lambda_\bullet^2   \right) \ .
\]
\end{defn}
\end{itemize}
\begin{rem}
A natural question to ask is why the limits in \Cref{2cellcompdef} exists and even if they all exist, why the sequences built by these limit will yield a $ 2 $-cell in $ \textbf{UC}^\t{tr} $.
One way to settle this issue is by viewing the loop operator $ S $ as a UCP operator and express the compositions as a certain Chois-Effros product (along the lines of Izumi's Poisson boundary approach in \cite{Izm}).
However, we will not take this route.
Instead we will make $ \textbf{UC}^\t{tr} $ sit inside the $ 2 $-category of von Neumann algebras, bimodules and intertwiners in a fully faithful way.
The benefit of this approach is that the details which are left out in defining $ \textbf{UC}^\t{tr} $ as a W*-2-category, namely, unit object, associativity of the two types of compositions, etc. will be automatically verified.
\end{rem}
\begin{rem}
In the passage from $ \textbf{UC} $ to $ \textbf{UC}^\t{tr} $, we are imposing restriction at the level of $ 1 $-cells but the $ 2 $-cell spaces have been generalized.
So, on the nose, neither we have a forgetful functor nor one turns out as a subcategory of the other.
However, we do have a subcategory of $ \textbf{UC}^\t{tr} $ which we call its \textit{flat part} and denote by $ \textbf{UC}^\t{flat} $ where everything is the same as that of $ \textbf{UC}^\t{tr} $ at the level of $ 0 $- and $ 1 $-cells but the $ 2 $-cells in $ \textbf{UC}^\t{flat} $ are only flat sequences (and not all BQFS).
Indeed compositions of the $ 2 $-cells in $ \textbf{UC}^\t{flat} $ correspond to exactly to those in $ \textbf{UC} $; this easily follows from \Cref{BQFSflat}.
\end{rem}


\section{A concrete realization of $\normalfont \textbf{UC}^\t{tr} $ }
The goal of this section is to build a fully faithful 2-functor $\mathcal{PB}:\textbf{UC}^\t{tr}\rightarrow \textbf{vNAlg}$ where \textbf{vNAlg} is the $ 2 $-category of von Neumann algebras, bimodules and intertwiners (as stated in \Cref{mainthm}).
Our starting point will be the pre-C* algebras and right correspondences produced from $ 0 $- and $ 1 $-cells in $ \textbf{UC}^\t{tr} $ viewed as those in $ \textbf{UC} $ as descibed in \Cref{unicon}, and then take their appropriate completions.
At this point, it might seem it is enough to build the functor starting from the flat part $ \textbf{UC}^\t{flat} $; however, in that case, the functor may not be fully faithful at the level of 2-cells (which are only flat sequences).
In this section, we will be analyzing ``the kernel" of the 2-functor from $ \textbf{UC}^\t{flat} $; as a consequence, we justify the need of generalizing the $ 2 $-cells in $ \textbf{UC}^\t{flat} $ to those in $ \textbf{UC}^\t{tr} $.
\vspace*{4mm}
\subsection{$\mathcal{PB}$ on 0-cells}\label{0cells}$ \ $

Given a $0$-cell $\left(\Gamma_{\bullet} , \ul \mu^\bullet \right)$ in $\textbf{UC}^\t{tr}$, we consider $ m_0 $, $ A_k $'s and their inclusions as in the non-tracial case \Cref{unicon}.
Using the categorical trace $\t{Tr}= (\t{Tr}_x)_{x \in \t{ob}(\mcal M_k)} $ associated to the weight vector $ \ul \mu^k $, we define $ \t{Tr}_{A_k} \coloneqq \t{Tr}_{\Gamma_k \cdots \Gamma_1 m_0} \ : A_k \ra \C $ which turns out to be a faithful tracial state which by \Cref{TrCondExp}, turns out to be compatible with the inclusion.
Thus, we have a faithful tracial state $ \t{Tr}_{A_\infty} $ on the $ * $-algebra $ A_\infty
$.
Note that the action of an element of $ A_\infty $ on the GNS Hilbert $ L^2 (A_\infty , \t{Tr}_{A_\infty}) $ is bounded.
Let $ A $ denote the type $ II_1 $ von Neumann algebra obtained by taking the WOT closure of $ A_\infty $ acting on $ L^2 (A_\infty , \t{Tr}_{A_\infty}) $.

\begin{defn}
We define $\mathcal{PB}\left(\Gamma_{\bullet} , \ul \mu^\bullet \right):=A=A^{\prime \prime}_{\infty}\subseteq \mcal L (L^{2}(A,\t{Tr}(A_{\infty}))$.
\end{defn}
\vspace*{4mm}
\subsection{$ \mcal {PB} $ on $ 1 $-cells}\label{1cells} $ \ $

Let $ \Lambda_\bullet \in \textbf{UC}^\t{tr}_1 \left(\left(\Gamma_\bullet , \ul \mu^\bullet\right) , \left(\Delta_\bullet , \ul \nu^\bullet\right)\right) $.
Consider the $ A_\infty $-$ B_\infty $ right correspondence $ H_\infty $ associated to the $ 1 $-cell $ \Lambda_\bullet $ treated as a $ 1 $-cell in \textbf{UC}.
Let $ H $ be the completion of $ H_\infty $ with respect to the scalar inner product $ \left\lab \xi , \zeta  \right \rab \coloneqq \t{Tr}_{B_\infty} \left(\left\lab \xi , \zeta  \right \rab_{B_\infty}\right) $ for $ \xi , \zeta \in H_\infty $.
$ A_\infty , B_\infty $ being locally semisimple $ * $-algebras, must have the action of their elements on $ H_\infty $ bounded, and hence extend to action on $ H $.

To obtain a right $ B $-action on $ H $, we work with the Pimsner-Popa basis $ \mscr S $ for the right-$ B_\infty $-module $ {H_\infty }$ with respect to the $ B_\infty $-valued inner product obtained in \Cref{PPbasispreC*}.
Observe that the map
\[
H \supset H_\infty \ni \xi \longmapsto \displaystyle \sum_{\sigma \in \mscr S} \sigma \otimes \left \lab \xi , \sigma \right \rab_{B_\infty} \in q \left[  \ell^2 \left(\mscr S\right) \otimes L^2 \left( B_\infty , \t{Tr}_{B_\infty} \right) \right] \eqqcolon K \ \t{ (say)}
\]
extends to an isometric isomorphism preserving the right $ B_\infty  $-action where $ q $ is the projection $ \displaystyle \sum_{\sigma , \tau \in \mscr S} E_{\sigma, \tau} \otimes \left \lab \tau , \sigma \right \rab_{B_\infty} $.
Clearly, the $ B_\infty $ action on $ K $ extends to a normal action of $ B $ and hence, the same holds for the Hilbert space $ H $.

In order to extend the $ A_\infty $-action on $ H $ (which is clearly bounded) to a normal action of $ A $, we first analyse the commutant of $ B $ in $ \mcal L (H) $.
For $ k\geq 0 $, define $ C_k \coloneqq \t{End} (\Lambda_k \Gamma_k \cdots \Gamma_1 m_0) $.
The $ * $-homomorphism $C_k \ni \gamma \longmapsto \left.  \Phi_\gamma \right|_{H_k}  = \gamma \circ \bullet \in \mcal L (H_k)$ is faithful by \Cref{PhiproppreC*}, and hence an isometry.
Thus, $ \Phi_\gamma $ extends to the whole of $ H$ as a bounded operator commuting with the right action of $ B_\infty $ (and thereby $ B $).
Consider the unital $ * $-algebra inclusion
\begin{align*}
	C_k \ni \gamma \longmapsto \left(W_k\right)_{\Gamma_k \cdots \Gamma_1 m_0}\ \circ \ \Delta_{k+1} \gamma\ \circ\ \left(W_k\right)^*_{\Gamma_k \cdots \Gamma_1 m_0} =
	\raisebox{-8mm}{
		\begin{tikzpicture}
			\node[right] at (-.1,1.1) {$ m_0 $};
			\node[right] at (-.1,2.1) {$ m_0 $};
			\node[left] at (-1.3,2.2) {$ \Lambda_{k+1} $};
			\node[left] at (-1.3,1.1) {$ \Lambda_{k+1} $};
			\node[draw,thick,rounded corners, fill=white,minimum width=40] at (-.5,1.6) {$\gamma$};
			\node[left] at (-.1,2.1) {$ \cdots $};
			\node[left] at (-.1,1.1) {$ \cdots $};
			\begin{scope}[on background layer]
				\draw[thick,dashed] (0,0.9) to (0,2.3);
				\draw[in=-90,out=90] (-.2,0.9) to (-.2,2.3);
				\draw[in=-90,out=90] (-.9,0.9) to (-.9,2.3);
				\draw[in=-90,out=90] (-1.1,.9) to (-1.4,1.3);
				\draw[in=-90,out=90] (-1.4,1.3) to (-1.4,1.9);
				\draw[in=-90,out=90] (-1.4,1.9) to (-1.1,2.3);
				\draw[white,line width=1mm,in=-90,out=90] (-1.1,1.9) to (-1.4,2.3);
				\draw[red,in=-90,out=90] (-1.1,1.9) to (-1.4,2.3);
				\draw[red,in=-90,out=90] (-1.1,1.3) to (-1.1,1.9);
				\draw[white,line width=1mm,in=-90,out=90] (-1.4,.9) to (-1.1,1.3);
				\draw[red,in=-90,out=90] (-1.4,.9) to (-1.1,1.3);
			\end{scope}
		\end{tikzpicture}
	} \in C_{k+1} \ .
\end{align*}
Note that $ \Phi_\gamma $ is compatible with the above inclusion.
Indeed, $ C_\infty \ni \gamma  \os {\displaystyle \Phi} \longmapsto \Phi_\gamma \in \mcal L_B (H) $ becomes a unital faithful $ * $-algebra homomorphism where $ C_\infty \coloneqq \us{k\geq 0} \cup C_k $.
\begin{prop}\label{C=B'}
	$ \mcal L_B (H) = \left\{ \Phi_\gamma : \gamma \in C_\infty\right\}''$ .
\end{prop}
\begin{proof}
	Consider the projection $ p_k \in \mcal L (H) $ such that $ \t{Range} (p_k) = H_k $.
	Since $ A_k\cdot H_k \cdot B_k = H_k $, therefore $ p_k $ must be $ A_k $-$ B_k $-linear.
	
	Let $ T \in \mcal L_B (H) $.
	Set  $ \zeta_k \coloneqq \us {\sigma \in \mscr S}  \sum
	\raisebox{-14mm}{
		\begin{tikzpicture}
			\node[draw,thick,rounded corners, fill=white,minimum width=60] at (-.8,2.1) {$p_k (T\sigma)$};
			\node[draw,thick,rounded corners, fill=white] at (-.15,.9) {$\sigma^*$};
			\node[left] at (-.6,.9) {$ \cdots $};
			\node[right] at (-.1,2.7) {$ m_0 $};
			\node[right] at (-.1,.2) {$ m_0 $};
			\node[right] at (-.1,1.4) {$ n_0 $};
			\node[left] at (-1.5,2.7) {$ \Lambda_k $};
			\node[left] at (-1.5,.2) {$ \Lambda_k $};
			\node[left] at (-.3,2.7) {$ \cdots $};
			\draw[white,line width=1mm,in=-90,out=90] (-1.6,0) to (-.3,.6);
			\draw[red,in=-90,out=90] (-1.6,0) to (-.3,.6);
			\begin{scope}[on background layer]
				\draw[in=-90,out=90] (-.5,0) to (-.8,1.75);
				\draw[in=-90,out=90] (-1.3,0) to (-1.6,1.75);
				\draw[thick,dashed] (0,0) to (0,3);
				\draw[red,in=-90,out=90] (-1.6,2.4) to (-1.6,3);
				\draw (-1.3,2.4) to (-1.3,3);
				\draw (-.3,2.4) to (-.3,3);
			\end{scope}
		\end{tikzpicture}
	}
	\in C_k$ where $ \mscr S $ ($ \subset H_0 $) is a PP-basis of the right $ B $-module $ H $ as constructed in \Cref{PPbasispreC*}.
	Since $ T $ (resp. $ p_k $) is right $ B $- (resp. $ B_k $-) linear, one may deduce the relation $ p_k T p_k = \Phi_{\zeta_k}\, p_k$.
	Clearly, $ p_k $ converges to $ \t{id}_H $ in SOT as $ k $ goes to $ \infty $, and $ \{\Phi_{\zeta_k}\}_{k\geq 0} $ is a norm bounded subset of $ \mcal L (H) $.
	Hence, $ T \in \ol{\left\{ \Phi_\gamma : \gamma \in C_\infty\right\}}^{\t{SOT}}  = \left\{ \Phi_\gamma : \gamma \in C_\infty\right\}''$.
\end{proof}
\begin{rem}\label{pk-picture}
Using \Cref{trivial-loop-anticlock}, we may represent the projection $ p_k $ in the following way
\[
H \supset H_{k+l} \ni \xi \os{\displaystyle p_k}{\longmapsto} 
\raisebox{-18mm}{
	\begin{tikzpicture}
		\node[left] at (-.1,1.475) {$ \cdots $};
		\node[left] at (-.1,-.6) {$ \cdots $};
		\node[left] at (-3.2,.65) {$ \cdots $};
		\node[left] at (-1.1,.15) {$ \cdots $};
		\node[left] at (-1.1,1.1) {$ \cdots $};
		\draw[thick,dashed] (0,-1.15) to (0,2.3);
		\draw[in=-90,out=90] (-.2,-1.15) to (-.2,2.3);
		\draw[in=-90,out=90] (-.9,-1.15) to (-.9,2.3);
		\draw (-2.55,.65) ellipse (7mm and 12mm);
		\draw (-2.55,.65) ellipse (15mm and 16mm);
		\draw[white,line width=1mm,in=-90,out=90] (-2.2,.9) to (-1.1,2.3);
		\draw[red,in=-90,out=90] (-2.2,.9) to (-1.1,2.3);
		\node[left] at (-1.05,2.2) {$ \Lambda_{k} $};
		\node[left] at (-2.05,1.2) {$ \Lambda_{k+l} $};
		\node[left] at (-1.85,.15) {$ \Delta_{k+l} $};
		\node[left] at (-.9,-.95) {$ \Delta_{k+1} $};
		\node[draw,thick,rounded corners, fill=white,minimum width=80] at (-1.1,.65) {$\xi$};
		\node[right] at (-.1,1.475) {$ m_0 $};
		\node[right] at (-.1,-.6) {$ n_0 $};
	\end{tikzpicture}
} \in H_k \subset H
\]
where the local maxima and minima are given by the natural transformation appearing in the tracial solution to conjugate equation for the duality of the functors $ \Delta_i
$'s associated to the positive weights $ \ul \nu^{i-1} $ and $ \ul \nu ^i $ on the vertices (that is, the simple objects of $ \mcal N_{i-1} $ and $ \mcal N_i  $).
Clearly, $ p_k $ is $ A_k $-$ B_k $-linear.
\end{rem}

Consider the unital $ * $-algebra inclusion 
$ 
A_k \ni \alpha \os {\displaystyle \Lambda_k} \longmapsto \Lambda_k \alpha \in C_k \ .
 $ Again, this inclusion is compatible with $ C_k \hookrightarrow C_{k+1} $ and $ A_k \hookrightarrow A_{k+1} $; thus, $ A_\infty $ sits as a unital $ * $-subalgebra inside $ C_\infty $.
Observe that if $ \gamma \in C_k$ comes from $ A_k $, that is, $ \gamma = \Lambda_k \alpha $ for some $ \alpha \in A_k $, then $ \Phi_\gamma $ matches exactly with the action of $ \alpha $ on $ H_\infty $.
Now, the functional
\[
\t{Tr}' \coloneqq \left[ d_B (H)\right]^{-1} \us{\sigma \in \mscr S} \sum \left\lab \bullet \, \sigma , \sigma \right\rab :  \mcal L_B (H) \ra \C
\]
is a faithful normal tracial state where $ \mscr S $ is a PP-basis for the module $ H_B $ and $ d_B(H) \coloneqq \us {\sigma \in \mscr S} \sum \norm{\sigma}^2 $; however, its restriction on $ A_\infty $ may not match with that of $ \t{Tr}_{A_\infty} $.
\begin{prop}\label{vNext}
	The above inclusion of $ A_\infty $ inside $ C_\infty $ extends to a normal inclusion of $ A $ inside $ \mcal L_B (H) $, and thereby $ H $ becomes a `von Neumann' $ A $-$ B $-bimodule.
\end{prop}
\begin{proof}
	By construction, $ A $ is the von Neumann algebra obtained from the GNS of $ A_\infty $ with respect to $ \t{Tr}_{A_\infty} $.
Let $ A''_\infty $ denote the double commutant of $ A_\infty $ sitting inside $ \mcal L (H) $ via the inclusions $ A_\infty \ni \alpha \os {\displaystyle \Lambda_\bullet} \longmapsto \Lambda_\bullet \alpha \in C_\infty $ and $ C_\infty \os{\displaystyle \Phi} \hookrightarrow  \mcal L (H)$.
It is enough to produce a central positive invertible element $ T $ in $ A''_\infty $ satisfying $ \t{Tr}' \left( \Phi_{\Lambda_\bullet \alpha} \ T \right)  = \t{Tr}_{A_\infty} (\alpha)  $ for $ \alpha \in A_\infty $ (that is, $ \t{Tr}_{A_\infty} $ extends to a faithful normal trace on $ A''_\infty $).
	
Consider the natural transformation $ \theta^{k} \coloneqq \left( \displaystyle \frac {\mu^{k}_v} {\left[\Lambda'_k\ \ul \nu^{k}\right]_v}  \, 1_v \right)_{v\in V_{\mcal M_k}} \in \t{End} (\t{id}_{\mcal M_k}) $.
Set $ T_k \coloneqq \Phi_{\Lambda_k \left( \theta^{k}_{\Gamma_k \cdots \Gamma_1 m_0} \right)} \in A''_\infty $ and $ \psi \coloneqq \us{\sigma \in \mscr S} \sum \left\lab \bullet \, \sigma , \sigma \right\rab = d_B (H) \ \t{Tr}'$.
	\vspace*{4mm} 
	
	\noindent \textit{Assertion:} $ \psi \left( \Phi_{\Lambda_k (\bullet)} \, T_k \right) = \t{Tr}_{A_k} $ for all $ k\geq 0 $.
	
	\noindent\textit{Proof of the assertion.} Let $ \alpha \in A_k $.
	Then, $ \psi \left( \Phi_{\Lambda_k (\alpha)} \, T_k \right) = \us{\sigma \in \mscr S} \sum \left\lab \Phi_{\Lambda_k \left(\alpha  \ \theta^{k}_{\Gamma_k \cdots \Gamma_1 m_0}  \right)} \, \sigma , \sigma \right\rab $
	\[
	= \us{\sigma \in \mscr S} \sum \t{Tr}_{B_k} \left(\left\lab \Phi_{\Lambda_k \left(\alpha  \ \theta^{k}_{\Gamma_k \cdots \Gamma_1 m_0}  \right)} \, \sigma , \sigma \right\rab_{B_k}\right)
	= \us{\sigma \in \mscr S} \sum \t{Tr}_{\Delta_k \cdots \Delta_1 m_0} \left(
	\raisebox{-18mm}{
		\begin{tikzpicture}
			\draw[thick,dashed] (1.8,1.8) to (1.8,-1.8);
			\draw[in=-90,out=90] (.7,1) to (.7,1.8);
			\draw[in=-90,out=90] (1.6,.3) to (.7,1);
			\draw[in=-90,out=90] (1.6,-.3) to (1.6,.3);
			\draw[in=-90,out=90] (.7,-1) to (1.6,-.3);
			\draw[in=-90,out=90] (.7,-1.8) to (.7,-1);
			\draw[in=-90,out=90] (-.1,1) to (-.1,1.8);
			\draw[in=-90,out=90] (.8,.3) to (-.1,1);
			\draw[in=-90,out=90] (.8,-.3) to (.8,.3);
			\draw[in=-90,out=90] (-.1,-1) to (.8,-.3);
			\draw[in=-90,out=90] (-.1,-1.8) to (-.1,-1);
			\draw[white,line width=1mm,in=-90,out=90] (1,-1) to (-.6,-.3);
			\draw[red,in=-90,out=90] (1,-1) to (-.6,-.3);
			\draw[white,line width=1mm,in=-90,out=90] (-.6,.3) to (1,.9);
			\draw[red,in=-90,out=90] (-.6,.3) to (1,.9);
			\draw[red,in=-90,out=90] (-.6,-.3) to (-.6,.3);
			\node[draw,thick,rounded corners, fill=white] at (0.1,0) {$\theta^{k}$};
			\node[draw,thick,rounded corners, fill=white,minimum width=30] at (1.4,-1.2) {$\sigma$};
			\node[draw,thick,rounded corners, fill=white,minimum width=30] at (1.4,1.2) {$\sigma^*$};
			\node[draw,thick,rounded corners, fill=white,minimum width = 35] at (1.3,0) {$\alpha$};
			\node[right] at (.8,-.4) {$ \cdots $};
			\node[right] at (.8,.4) {$ \cdots $};
			\node[right] at (-.1,1.2) {$ \cdots $};
			\node[right] at (-.1,-1.2) {$ \cdots $};
			\node[left] at (-.5,0) {$ \Lambda_k $};
			\node[right] at (1.7,-1.65) {$ n_0 $};
			\node[right] at (1.7,1.6) {$ n_0 $};
			\node[right] at (1.7,-.7) {$ m_0 $};
			\node[right] at (1.7,.6) {$ m_0 $};
		\end{tikzpicture}
	}\right) \ .
	\]
	Using the property of the categorical trace, the equation in \Cref{PPbasispreC*} satisfied by the set $ \mscr S $ and the natural unitaries (namely, the crossings), we may rewrite the last expression as
	\[
	\t{Tr}_{\Lambda_k \Gamma_k \cdots \Gamma_1 m_0} \left(
	\raisebox{-6mm}{
		\begin{tikzpicture}
			\draw[thick,dashed] (1.8,.7) to (1.8,-.7);
			\draw[in=-90,out=90] (1.6,-.7) to (1.6,.7);
			\draw[in=-90,out=90] (.8,-.7) to (.8,.7);
			\draw[red,in=-90,out=90] (-.6,-.7) to (-.6,.7);
			\node[draw,thick,rounded corners, fill=white] at (0.1,0) {$\theta^{k}$};
			\node[draw,thick,rounded corners, fill=white,minimum width = 35] at (1.3,0) {$\alpha$};
			\node[right] at (.8,-.4) {$ \cdots $};
			\node[right] at (.8,.4) {$ \cdots $};
			\node[left] at (-.5,0) {$ \Lambda_k $};
			\node[right] at (1.7,-.5) {$ m_0 $};
			\node[right] at (1.7,.4) {$ m_0 $};
		\end{tikzpicture}
	}\right)
	= \t{Tr}_{\Gamma_k \cdots \Gamma_1 m_0} \left(
	\raisebox{-6mm}{
		\begin{tikzpicture}
			\draw[thick,dashed] (1.8,.7) to (1.8,-.7);
			\draw[in=-90,out=90] (1.6,-.7) to (1.6,.7);
			\draw[in=-90,out=90] (.8,-.7) to (.8,.7);
			\draw[red] (-1,0) ellipse (4mm and 6mm);
			\draw[red,in=-90,out=90] (-.6,0) to (-.6,0);
			\node[draw,thick,rounded corners, fill=white] at (0.1,0) {$\theta^{k}$};
			\node[draw,thick,rounded corners, fill=white,minimum width = 35] at (1.3,0) {$\alpha$};
			\node[right] at (.8,-.4) {$ \cdots $};
			\node[right] at (.8,.4) {$ \cdots $};
			\node[left] at (-.45,0) {$ \Lambda_k $};
			\node[right] at (1.7,-.5) {$ m_0 $};
			\node[right] at (1.7,.4) {$ m_0 $};
		\end{tikzpicture}
	}\right)
	\] 
	by \Cref{TrCondExp} where the red cap and cup correspond to tracial solution to conjugate equation for the duality of the functor $ \Lambda_k $ with respect to weight vectors $ \ul \mu^{k} $ and $ \ul \nu^k $ on the vertex sets $ V_{\mcal M_k} $ and $ V_{\mcal N_k} $ respectively.
	Now, it is a matter of routine verification that the red loop appearing above is indeed the inverse of $ \theta^k $ in the algebra $ \t{End} (\t{id}_{\mcal M_k}) $.
	Cancelling the two, we get $ \t{Tr}_{A_k} (\alpha) $.
	\vspace*{4mm}
	
	\Cref{bndonwt} implies that C*-norm of $ \theta^k $ is uniformly bounded by $ \epsilon^{-1} $ for $ k\geq 0 $, and thereby $ \left\{T_k\right\}_{k \geq 0} $ is norm-bounded sequence in $A''_\infty \subset \mcal L (H) $.
	By compactness, there exists a subsequence $ \left\{T_{k_l}\right\}_{l} $ which converges in WOT to $ T_0 \in A''_\infty $ (say).
	Clearly, $ \psi (\Phi_{\Lambda_\bullet} (\alpha) \ T_0) = \t{Tr}_{A_\infty} (\alpha)$ for all $ \alpha \in A_\infty $.
Observe that $ T_k $ commutes with $ \Phi_{\Lambda_k \alpha} $ for all $ \alpha \in A_k $, $ k\geq 0 $; this implies $ T_0 $ must be central in $ A''_\infty $.
Again, $ T_k $ is a positive element in $ A''_\infty $ satisfying $ T_k \geq M^{-1} $ (using \Cref{bndonwt}); thus, the subsequential WOT-limit $ T_0 $ also satisfies the same.
\end{proof}
\begin{defn}
Define
\[
\textbf{UC}^\t{tr}_1 \left(\left(\Gamma_\bullet , \ul \mu^\bullet\right) , \left(\Delta_\bullet , \ul \nu^\bullet\right)\right)  \ni \Lambda_\bullet \os{\displaystyle \mcal{PB}}{\longmapsto} \mcal{PB} \left( \Lambda_\bullet \right) \coloneqq \vphantom{H}_A H_B \in \textbf{vNalg}_1 \left(\mcal{PB} \left(\Delta_\bullet , \ul \nu^\bullet\right) , \mcal{PB} \left(\Gamma_\bullet , \ul \mu^\bullet\right) \right) \ .
\]
\end{defn}
\comments{
In \Cref{vNext}, we saw that the boundedness condition \Cref{bndonwt} serves as a sufficient condition for the von Neumann extension of the right correspondence $ \vphantom{H_\infty}_{A_\infty} {H_\infty}_{B_\infty}  $.
In fact, it is necessary as well.
To see this, first observe that the existence of PP-basis ensures that $ H $ is a right von Neumann module over $ B $ (for which we do not need \Cref{bndonwt}).
Now, suppose the left $ A_\infty $ action on $ H $ extends to a normal action of $ A $.
}
\vspace*{4mm}
\subsection{$ \mcal{PB} $ on $ 2 $-cells}\label{2cells}$ \ $

Let $ \Lambda_\bullet, \Omega_\bullet \in \textbf{UC}^\t{tr}_1 \left(\left(\Gamma_\bullet , \ul \mu^\bullet\right) , \left(\Delta_\bullet , \ul \nu^\bullet\right)\right) $ and $ \mcal {PB} \left(\Gamma_\bullet, \ul \mu^\bullet\right) = A$, $ \mcal {PB} \left(\Delta_\bullet, \ul \nu^\bullet\right) = B$.
We will borrow the notations $ H_k$,  $H$, $p_k$, $\mscr S$, etc.  (arising out of $ \Lambda_\bullet  $) from previous subsections, and for those arising out of $ \Omega_\bullet $, we will use
$ G_k$,  $G$, $q_k$, $\mscr T$, etc. respectively, and we will also work with the pictures as before.
For $ \gamma \in \mcal N_k \left(\Lambda_k \Gamma_k \cdots \Gamma_1 m_0 \ , \ \Omega_k \Gamma_k \cdots \Gamma_1 m_0\right) $, we will consider the unique bounded extension of $ \Phi_\gamma \in \mcal L_{B_\infty} \left( H_\infty ,  G_\infty \right) $ (defined in \Cref{Phidef}) and denote it with the same symbol $ \Phi_\gamma \in \mcal L_B \left(H , G\right) $.
\begin{prop}\label{Spropiv}
(a) $ q_{k-1} \ \Phi_{\eta_{\Gamma_{k}\cdots \Gamma_1 m_0}} \ p_{k-1} = \Phi_{(S_k \eta)_{\Gamma_{k-1}\cdots \Gamma_1 m_0}} \ p_{k-1}  $ for all $ \eta \in \normalfont \t{NT} (\Lambda_k, \Omega_k) $ and $ k\geq 0 $.

(b) If $ \ul \eta = \left\{ \eta^{(k)}\right\}_{k\geq 0} \in \normalfont \textbf{UC}^\t{tr}_2 \left(\Lambda_\bullet , \Omega_\bullet \right) $ (that is, a BQFS from $ \Lambda_\bullet $ to $ \Omega_\bullet $), then it gives rise to a unique bounded operator $ T \in \vphantom{\mcal L}_A \mcal L_B \left(H,G\right)$ satisfying
\begin{equation}\label{TPhieqn}
q_k\ T \ p_k = \Phi_{\eta^{(k)}_{\Gamma_k \cdots \Gamma_1 m_0}} \, p_k \ \t{ for all }\ k \geq 0 \ .
\end{equation}
\end{prop}
\begin{proof}
Part (a) directly follows from \Cref{pk-picture}.

For (b), set $ T_k \coloneqq \Phi_{\eta^{(k)}_{\Gamma_k \cdots \Gamma_1 m_0}} $ for each $ k \geq 0 $; it is easy to see that $ T_k \in \vphantom{\mcal L}_{A_k} \mcal L_B (H, G)  $ and $ \norm{T_k} = \norm{\eta^{(k)}} $.
For all $ \gamma \in H_k $, using part (a) followed by quasi-flat condition of $ \ul \eta $, we have
\[
q_k  T_{k+1}\, \gamma
= \Phi_{\left(S_{k+1} \ \eta^{(k+1)} \right)_{\Gamma_k \cdots \Gamma_1 m_0}}  \gamma
=\Phi_{\eta^{(k)}_{\Gamma_k \cdots \Gamma_1 m_0}}  \gamma
= T_k\, \gamma \ .
\]
In other words, $ q_k\, T_{k+1}\, p_k = T_k\, p_k$.
Applying this iteratively, we get $  q_k\, T_{k+l}\, p_k = T_k\, p_k = q_k\, T_k\, p_k $ for all $ k,l \geq 0 $. 
This again implies 
\begin{equation}\label{Tk+l-Tk}
	\norm{T_{l+m}\, \gamma - T_l\, \gamma}^2 = \norm{T_{l+m}\, \gamma }^2 - \norm{T_l\, \gamma}^2 \ \t{ for all } k\leq l, \, \gamma \in H_k \ .
\end{equation}

Now, fix $ \gamma \in H_k $.
\Cref{Tk+l-Tk} tell us that the sequence $ \left\{ \norm{T_l\, \gamma} \right\}_{l\geq k} $ must be increasing; also, it is bounded by $ \left[\sup_{m\geq 0} \norm{\eta^{(m)}}\right] \norm \gamma$ and hence convergent.
Letting $ l $ tend towards $ \infty $ in \Cref{Tk+l-Tk}, we find that $ \left\{ T_k \right\}_{k\geq 0} $ converges pointwise on $ H_\infty $ (because of completeness of $ G $).
Since $ H_\infty $ is dense in $ H $ and $ \left\{ T_k\right\}_{k\geq 0} $ is norm-bounded by $ \sup_{k\geq 0} \norm{\eta^{(k)}} $, we may conclude that $ \left\{ T_k\right\}_{k\geq 0} $ converges in SOT to some $ T \in \mcal L (H , G) $.

To prove \Cref{TPhieqn}, consider $ q_k\, T\, p_k = \t{SOT-}{\displaystyle\lim_{l\to \infty}}\, q_k\, T_{k+l}\, p_k = q_k\, T_k\, p_k = \Phi_{\eta^{(k)}_{\Gamma_k \cdots \Gamma_1 m_0}}\, p_k$.
Since the right side of condition (b) is $ A_k \t{-} B_k $-linear, so is the other side namely, $ q_k T p_k $.
Since $ \left\{ q_k T p_k\right\}_{k\geq 0} $ converges in SOT to $ T $, therefore, $ T $ must be $ A_k $-$ B_k $-linear, and therby $ A_\infty $-$ B_\infty $-linear, and finally $ A $-$ B $-linear.

If $ T_1 $ is any other operator satisfying \Cref{TPhieqn}, then $ q_k\, (T - T_1)\, p_k = 0$.
Now, $ p_k $ and $ q_k $ increase to $ \t{id}_H $ and $ \t{id}_{G} $ respectively.
This forces $ T $ and $ T_1 $ to be identical.
\end{proof}
\begin{defn}
For $ \ul \eta = \left\{ \eta^{(k)}\right\}_{k\geq 0} \in \textbf{UC}^\t{tr}_2 \left(\Lambda_\bullet , \Omega_\bullet \right) $, define
\[
\mcal {PB} \left( \ul \eta \right) \coloneqq \t{SOT-}\displaystyle \lim_{k \to \infty} \Phi_{\eta^{(k)}_{\Gamma_k \cdots \Gamma_1 m_0}} \in \vphantom{\mcal L}_A \mcal L_B \left(H,G\right) = \textbf{vNAlg}_{2} \left( \mathcal{PB}\left (\Lambda_{\bullet}\right ), \mathcal{PB}\left (\Omega_{\bullet} \right ) \right ) \ .
\]
\end{defn}


\begin{prop}\label{BQFSprop}
	For every $ T \in \vphantom{\mcal L}_A \mcal L_B (H, G) $ and $ k \geq 0 $, there exists unique $ \eta^{(k)} \in \t{NT} (\Lambda_k , \Omega_k) $ such that $ q_k\, T\, p_k = \Phi_{\eta^{(k)}_{\Gamma_k \cdots \Gamma_1 m_0}} \, p_k$ (which is the same as \Cref{TPhieqn}).
\end{prop}
\begin{proof}
We will use a modified version of a trick which we have already seen twice before, namely, in the proofs of \Cref{PhiproppreC*} and \Cref{flatthm}.
Set $  \zeta_k \coloneqq  \displaystyle \sum_{\sigma \in \mscr S}
\raisebox{-14mm}{
	\begin{tikzpicture}
		\node[draw,thick,rounded corners, fill=white,minimum width=60] at (-.8,2.1) {$q_k (T\sigma)$};
		\node[draw,thick,rounded corners, fill=white] at (-.15,.9) {$\sigma^*$};
		\node[left] at (-.6,.9) {$ \cdots $};
		\node[right] at (-.1,2.7) {$ m_0 $};
		\node[right] at (-.1,.2) {$ m_0 $};
		\node[right] at (-.1,1.4) {$ n_0 $};
		\node[left] at (-1.5,2.7) {$ \Omega_k $};
		\node[left] at (-1.5,.2) {$ \Lambda_k $};
		\node[left] at (-.3,2.7) {$ \cdots $};
		\draw[white,line width=1mm,in=-90,out=90] (-1.6,0) to (-.3,.6);
		\draw[red,in=-90,out=90] (-1.6,0) to (-.3,.6);
		\begin{scope}[on background layer]
			\draw[in=-90,out=90] (-.5,0) to (-.8,1.75);
			\draw[in=-90,out=90] (-1.3,0) to (-1.6,1.75);
			\draw[thick,dashed] (0,0) to (0,3);
			\draw[red,in=-90,out=90] (-1.6,2.4) to (-1.6,3);
			\draw (-1.3,2.4) to (-1.3,3);
			\draw (-.3,2.4) to (-.3,3);
		\end{scope}
	\end{tikzpicture}
}	
\in \mcal N_k \left(\Lambda_k \Gamma_k \cdots \Gamma_1 m_0 , \Omega_k \Gamma_k \cdots \Gamma_1 m_0\right) $.
	With similar reasoning as before, one can easily conclude $ q_k\, T\, p_k = \Phi_{\zeta_k} \, p_k $; moreover, this equation uniquely determines $ \zeta_k $ by \Cref{PhiproppreC*} (iii).
	Further, the left side of the equation is $ A_k $-linear; then so is the right side.
	Again by \Cref{PhiproppreC*} (iii), $ \zeta_k $ becomes an $ A_k $-central vector of $ \mcal N_k \left(\Lambda_k \Gamma_k \cdots \Gamma_1 m_0 , \Omega_k \Gamma_k \cdots \Gamma_1 m_0\right)  $.
	Applying \Cref{AkcommNTpreC*}, we get a unique $ \eta^{(k)} \in \t{NT}(\Lambda_k , \Omega_k)$ satisfying $ \zeta_k = \eta^{(k)}_{\Gamma_k \cdots \Gamma_1 m_0} $.
	This completes the proof.
\end{proof}
\begin{thm}\label{BQFSthm}
The following is an isomorphism 
\[
\normalfont\textbf{UC}^\t{tr}_2 \left(\Lambda_\bullet , \Omega_\bullet \right) \ni  \ul \eta  \os{\displaystyle \mcal {PB}}{\longmapsto} \mcal {PB} \left( \ul \eta \right) \in  \textbf{vNAlg}_{2} \left( \mathcal{PB}\left (\Lambda_{\bullet}\right ), \mathcal{PB}\left (\Omega_{\bullet} \right ) \right ) \ .
\]
(This will eventually imply that the $ 2 $-functor $ \mcal{PB} $ is fully faithful.)
\end{thm}
\begin{proof}
Suppose $ \mcal {PB} \left( \ul \eta \right) = 0 $.
Then, by \Cref{TPhieqn}, we have $ \left. \Phi_{\eta^{(k)}_{\Gamma_k\cdots \Gamma_1 m_0}} \right|_{H_k} = 0 $ which (by \Cref{PhiproppreC*} (iii)) implies $ \eta^{(k)}_{\Gamma_k\cdots \Gamma_1 m_0} = 0 $.
Now, \Cref{AkcommNTpreC*} ensures that $ \eta^{(k)} $ must be zero for all $ k $.
	
For surjectivity, pick $ T \in \vphantom{\mcal L}_A \mcal L_B (H, G) $.
We only need to show that the unique sequence $ \left\{ \eta^{(k)} \in \t{NT} (\Lambda_k , \Omega_k)  \right\}_{k\geq 0} $ associated to $ T $ obtained in \Cref{BQFSprop}, is quasi-flat and bounded in C*-norm.
Note that for all $ \gamma \in H_k = \mcal N_k (\Delta_k \cdots \Delta_1 n_0, \Lambda_k \Gamma_k \cdots \Gamma_1 m_0)$, we apply \Cref{TPhieqn} twice and obtain
	\[
	\Phi_{\eta^{(k)}_{\Gamma_k\cdots \Gamma_1 m_0}} \, \gamma 
	= q_k T p_k\ \gamma
	=  q_k q_{k+1} T p_{k+1}\ \gamma
	=   q_k  \Phi_{\eta^{(k+1)}_{\Gamma_{k+1}\cdots \Gamma_1 m_0}} p_{k+1} \, \gamma 
	= \Phi_{\left(S_{k+1} \ \eta^{(k+1)} \right)_{\Gamma_k\cdots \Gamma_1 m_0}} \, \gamma
	\]
where the last equality follows from \Cref{Spropiv} (a).
By \Cref{PhiproppreC*} (iii), we must have $ \eta^{(k)}_{\Gamma_k \cdots \Gamma_1 m_0} = \left[S \ \eta^{(k+1)} \right]_{\Gamma_k \cdots \Gamma_1 m_0} $ which via the isomorphism in \Cref{AkcommNTpreC*},
	implies $ \eta^{(k)} = S_{k+1} \ \eta^{(k+1)} $.
For boundedness, we apply the norm on both sides of \Cref{TPhieqn};
note that the map in \Cref{PhiproppreC*} (iii) is actually an isometry (with respect to the C*-norms) which yeilds the inequality $ \norm T \geq \norm{\eta^{(k)}_{\Gamma_k \cdots \Gamma_1 m_0}} = \norm{\eta^{(k)}}$ where the last equality holds because $ \Gamma_k \cdots \Gamma_1 m_0 $ contains every simple of $ \mcal M_k $ as a subobject.
\end{proof}
\vspace*{4mm}
\subsection{$ \mcal {PB} $ preserves tensor product of $ 1 $-cells and compostions of $ 2 $-cells}$ \ $

Our goal here is clear from the title of this section.
As a by product of achieving this goal, we will prove the existence of the limits appearing in

\begin{prop}\label{2cellcomplim}
For	$ \ul \eta = \left\{ \eta^{(k)}\right\}_{k \geq 0} \in \normalfont\textbf{UC}^\t{tr}_2 \left(\Lambda_\bullet , \Omega_\bullet\right) $ and $ \ul \kappa = \left\{ \kappa^{(k)}\right\}_{k \geq 0} \in \normalfont\textbf{UC}^\t{tr}_2 \left( \Omega_\bullet, \Xi_\bullet \right) $,

(a)  the sequence $ \left\{ S_{k+1} \cdots S_{k+l} \left(\kappa^{(k+l)} \circ \eta^{(k+l)} \right) \right\}_{l\geq 0} $ converges (in NT$ (\Lambda_k, \Xi_k) $) for every $ k\geq 0 $, and

(b) $ \mcal{PB} \left( \ul \kappa \cdot \ul \eta\right) = \mcal{PB} \left( \ul \kappa \right) \circ \mcal{PB} \left( \ul \eta\right)$.
\end{prop}
\begin{proof}
We continue using the previous notations and let us denote the Hilbert spaces and the projection corresponding to $ \Xi_\bullet $ by $ F_k $, $ F $ and $ s_k $; the intertwiners corresponding to $\ul \eta $ and $\ul \kappa $ will be denoted by $ X \in \vphantom{\mcal L}_A \mcal L_B (H,G) $ and $ Y \in \vphantom{\mcal L}_A \mcal L_B (G , F) $ respectively.
	Set $ Z \coloneqq YX \in \vphantom{\mcal L}_A \mcal L_B (H , F) $ whose corresponding BQFS will be $ \left\{ \psi^{(k)} \right\}_{k \geq 0} $.
	
	For fixed $ k \geq 0 $, using \Cref{BQFSprop} and \Cref{Spropiv} (a), we obtain
	\begin{align*}
& \Phi_{\psi^{(k)}_{\Gamma_k \cdots \Gamma_1 m_0}} p_k = s_k YX p_k = \t{SOT-}\lim_{l \to \infty} \, s_k Y q_{k+l} X p_k\\
& = \t{SOT-}\lim_{l \to \infty} \, s_k \Phi_{\kappa^{(k+l)}_{\Gamma_{k+l} \cdots \Gamma_1 m_0}} \Phi_{\eta^{(k+l)}_{\Gamma_{k+l}\cdots \Gamma_1 m_0}} p_k
= \t{SOT-}\lim_{l \to \infty} \Phi_{\left[S_{k+1} \cdots S_{k+l} \left(\kappa^{(k+l)} \circ\, \eta^{(k+l)}\right) \right]_{\Gamma_{k} \cdots \Gamma_1 m_0}} p_k
	\end{align*}
	Since $ \mcal N_k \left( \Lambda_k \Gamma_k \cdots \Gamma_1 m_0 , \Xi_k \Gamma_k \cdots \Gamma_1 m_0  \right) $ is finite dimensional, by the isometry in \Cref{PhiproppreC*}(iii), we may conclude that $ \left[S_{k+1} \cdots S_{k+l} \left(\kappa^{(k+l)}  \circ\, \eta^{(k+l)}\right)\right]_{\Gamma_k \cdots \Gamma_1 m_0} $ converges as $ l $ approaches $ \infty $ which again implies convergence of $ \left\{S_{k+1} \cdots S_{k+l} \left(\kappa^{(k+l)}  \circ\, \eta^{(k+l)}\right)\right\}_{l\geq 0} $ via \Cref{AkcommNTpreC*}.
\end{proof}

\vspace*{4mm}

Next, we deal with tensor product of $ 1 $-cells.
We will show that $ \mcal {PB} $ preserves it in the reverse order.
\begin{prop}\label{confusionbimod}
For $ 0 $-cells $\Gamma_\bullet , \Delta_\bullet, \Sigma_\bullet$ in $ \normalfont\textbf{UC}^\t{tr}_0 $, and $\Omega_\bullet \in \normalfont\textbf{UC}^\t{tr}_1(\Delta_\bullet, \Sigma_\bullet)$ and $\Lambda_\bullet \in \normalfont\textbf{UC}^\t{tr}_1(\Gamma_\bullet,\Delta_\bullet)$, the bimodule $\mcal {PB} \left( \Omega_\bullet \btimes \Lambda_\bullet  \right)$ is isomorphic to the Connes fusion $\mcal {PB}(\Lambda_\bullet)  \otimes \mcal {PB} \left(\Omega_\bullet \right) $.
\end{prop} 
\begin{proof}
	We borrow the notations in \Cref{confusionbimodpreC*} where we have already obtained an $ A_\infty $-$ C_\infty $-linear isomorphism from the relative tensor product of the dense subspace of  $\mcal {PB} (\Lambda_\bullet)  $ and $\mcal {PB} \left(\Omega_\bullet \right) $ to that of $\mcal {PB} \left( \Omega_\bullet \btimes \Lambda_\bullet  \right)$.
	Moreover, this isomorphism preserves the right $ C_\infty $-valued inner product; composing with $ \t{Tr}_{C_\infty} $, it preserves the scalar inner product as well, and hence extends to a $ A $-$ C $-linear unitary.
\end{proof}

We next proceed towards the horizontal composition of $ 2 $-cells.

\begin{prop}
Let $ \Lambda_\bullet^i \in \normalfont\textbf{UC}^\t{tr}_1 {\left(\Gamma_\bullet, \Delta_\bullet \right)}  $ and $\normalfont \Omega_\bullet^i \in \textbf{UC}^\t{tr}_1 {\left(\Delta_\bullet, \Sigma_\bullet\right)}$) for $ i =1,2 $, and $\normalfont \ul \eta = \left\{\eta^{(k)}\right\}_{k\geq 0} \in \textbf{UC}^\t{tr}_2 \left(\Lambda_\bullet^1 , \Lambda_\bullet^2  \right)$ and $\normalfont \ul \kappa = \left\{\kappa^{(k)}\right\}_{k\geq 0} \in \textbf{UC}^\t{tr}_2 \left(\Omega_\bullet^1 , \Omega_\bullet^2  \right)$.
	Then,
	
	(a) for each $ k\geq 0 $, the sequence $ \left\{ S_{k+1} \cdots S_{k+l} \left( \Omega_{k+l}^2 \left( \eta^{(k+l)} \right) \circ \kappa^{(k+l)}_{\Lambda_{k+l}^1} \right) \right\}_{l\geq 0} $ converges in $ NT \left(\Omega_k^1 \Lambda_k^1 , \Omega_k^2 \Lambda_k^2\right) $ where $ S $ is the loop operator from $ \Omega_\bullet^1 \btimes \Lambda_\bullet^1 $ to $ \Omega_\bullet^2 \btimes \Lambda_\bullet^2 $,
and indeed the sequence $\ul \kappa \btimes \ul \eta  \coloneqq 
	\left\{ \left(\ul \kappa \btimes \ul \eta \right)^{(k)} \coloneqq \us{l \to \infty} \lim S_{k+1} \cdots S_{k+l} \left( \Omega_{k+l}^2 \left( \eta^{(k+l)} \right) \circ \kappa^{(k+l)}_{\Lambda_{k+l}^1} \right) \right\}_{k \geq 0}  
	$ is a BQFS from $\Omega_\bullet^1  \btimes \Lambda_\bullet^1  $ to $\Omega_\bullet^2 \btimes \Lambda_\bullet^2 $,
	
	(b) $ \mcal {PB} \left(\ul \kappa \btimes \ul \eta \right) $ corresponds to the operator $ \mcal {PB} \left(\ul \eta\right)  \otimes \mcal {PB} \left(\ul \kappa \right)$ via the isomorphism of bimodules in \Cref{confusionbimod}.
\end{prop}
\begin{proof}
Continuing with the notations used in \Cref{confusionbimodpreC*}, we set $\vphantom{H^i}_A H^i_B \coloneqq \mcal {PB} (\Lambda_\bullet^i)$, $\vphantom{G^i}_B G^i_C \coloneqq \mcal {PB} \left(\Omega_\bullet^i  \right)$, $\vphantom{F^i}_A F^i_C \coloneqq \mcal {PB} \left(\Omega_\bullet^i \btimes \Lambda_\bullet^i  \right)$ for $ i = 1,2 $.

\noindent Set $ T \coloneqq \mcal {PB} (\ul \eta) \in \vphantom{\mcal L}_A \mcal L_B (H^1 , H^2) $ and $ T' \coloneqq \mcal {PB} (\ul \kappa) \in \vphantom{\mcal L}_B \mcal L_C (G^1 , G^2)$.
	Suppose $ X $ denote the intertwiner in $ \vphantom{\mcal L}_A \mcal L_C (F^1,F^2) $ induced by $ T\us B \otimes T' $ under the isomorphism in \Cref{confusionbimod}.
We need to prove that $ \ul \kappa \btimes \ul \eta  $ is the unique BQFS which gets mapped to $ X $ under the functor $ \mcal {PB} $.
	For $ \xi_i \in H^i_k $, $ \zeta_i \in G^i_k $, $ i = 1,2 $, $ k\geq 0 $, applying \Cref{BQFSprop} and using the isomorphism of bimodules in \Cref{confusionbimod} (in fact, \Cref{confusionbimodpreC*}), we get
	\begin{align*}
		I \coloneqq \left\lab \Phi_{\left(\ul \kappa \btimes \ul \eta \right)^{(k)}_{\Gamma_k \cdots \Gamma_1 m_0}} \left(\Omega_k^1 (\xi_1) \circ \zeta_1\right) \; , \; \Omega_k^2 (\xi_2) \circ \zeta_2 \right\rab_{F^2_k}
		&= \left\lab X \left(\Omega_k^1 (\xi_1) \circ \zeta_1\right) \; , \; \Omega_k^2 (\xi_2) \circ \zeta_2 \right\rab_{F^2} \\
		& = \left\lab  \left\lab T \xi_1 , \xi_2 \right \rab_B \; T' \zeta_1\; ,\; \zeta_2  \right \rab_{G^2}\; .
	\end{align*}
	We will now express $\left\lab T \xi_1 , \xi_2 \right \rab_B $ as a limit.
	Consider the sequence 
	\[
	\left\{b_l \coloneqq \raisebox{-2.1cm}{
		\begin{tikzpicture}
			\draw[in=-90,out=90] (.2,1.2) to (.2,2.2);
			\draw[in=-90,out=90] (.8,.4) to (.2,1.2);
			\draw[in=-90,out=90] (.8,-.4) to (.8,.4);
			\draw[in=-90,out=90] (.2,-1.2) to (.8,-.4);
			\draw[in=-90,out=90] (.2,-2.2) to (.2,-1.2);
			\draw[in=-90,out=90] (.8,1.2) to (.8,2.2);
			\draw[in=-90,out=90] (1.4,.4) to (.8,1.2);
			\draw[in=-90,out=90] (1.4,-.4) to (1.4,.4);
			\draw[in=-90,out=90] (.8,-1.2) to (1.4,-.4);
			\draw[in=-90,out=90] (.8,-2.2) to (.8,-1.2);
			\draw[in=-90,out=90] (1.6,-2.2) to (1.6,2.2);
			\draw[in=-90,out=90] (2.2,-2.2) to (2.2,2.2);
			\draw[dashed,thick] (2.4,-2.2) to (2.4,2.2);
			\draw[white,line width=1mm,in=-90,out=90]  (.2,.4)to (1.4,1.0);
			\draw[red,in=-90,out=90,] (.2,.4)to (1.4,1.0);
			\draw[white,line width=1mm,in=-90,out=90]  (1.4,-1.05) to (.2,-.4);
			\draw[red,in=-90,out=90] (1.4,-1.05) to (.2,-.4);
			\node[right] at (.1,1.4) {$ \cdots $};
			\node[right] at (.1,-1.4) {$ \cdots $};
			\node[right] at (.7,0) {$ \cdots $};
			\node[right] at (1.5,1.9) {$ \cdots $};
			\node[right] at (1.5,0) {$ \cdots $};
			\node[right] at (1.5,-1.9) {$ \cdots $};
			\node[draw,thick,rounded corners, fill=white,minimum height=20,minimum width=45] at (1.8,-1.4) {$\xi_1$};
			\node[draw,thick,rounded corners, fill=white,minimum height=20,minimum width=45] at (1.8,1.4) {$\xi^*_2$};
			\node[draw,thick,rounded corners, fill=white] at (0,0) {$\eta^{(k+l)}$};
			\node[left] at (0.3,.6) {$ \Lambda^2_{k+l} $};
			\node[left] at (0.35,-.65) {$ \Lambda^1_{k+l} $};
			\node[right] at (2.3,0) {$ m_0 $};
			\node[right] at (2.3,1.9) {$ n_0 $};
			\node[right] at (2.3,-2) {$ n_0 $};
		\end{tikzpicture}
	} \in B_{k+l} \subset B \right\}_{l\geq 0} \; .
	\]
	Observe that $ \left\lab b_l\, b'\; , \; b''\right\rab_{L^2 (B)} $ eventually becomes $ \left\lab \left\lab T \xi_1 , \xi_2 \right \rab_B \, b'\; , \; b''\right\rab_{L^2 (B)} $ as $ l $ grows bigger for $ b' , b'' \in B_\infty $.
	Since $ \{b_l\}_{l\geq 0} $ is bounded, it converges ultraweakly to $ \left\lab T \xi_1 , \xi_2 \right \rab_B $.
	Thus,
	\[
	I = \lim_{l\to \infty}  \left\lab  T' (b_l\, \zeta_1)\; ,\; \zeta_2  \right \rab_{G^2}
	= \lim_{l\to \infty}  \left\lab  \tilde \Phi_{\kappa^{(k+l)}_{\Sigma_{k+l} \cdots \Sigma_1 q_0}} (b_l\, \zeta_1)\; ,\; \zeta_2  \right \rab_{G^2_{k+l}}
	\]
	\[
=\lim_{l\to \infty}  \t{Tr}_{C_{k+l}}
\raisebox{-3.35cm}{
\begin{tikzpicture}
\draw[dashed,thick] (3.9,-3.5) to (3.9,3.5);
\draw (1,-3.5) to (1,-2.8) to[in=-90,out=90] (2.2,-.4) to (2.2,.4) to[in=-90,out=90] (1,2.8) to (1,3.5);
\draw (1.6,-3.5) to (1.6,-2.8) to[in=-90,out=90] (2.8,-.4) to (2.8,.4) to[in=-90,out=90] (1.6,2.8) to (1.6,3.5);
\draw (3.1,-3.5) to (3.1,3.5);
\draw (3.7,-3.5) to (3.7,3.5);
\draw[white,line width=1mm,in=-90,out=90] (2.9,-1.05) to[in=-90,out=90] (1.4,-.4) to (1.4,.4) to[in=-90,out=90] (2.9,1);
\draw[red] (2.9,-1.05) to[in=-90,out=90] (1.4,-.4) to (1.4,.4) to[in=-90,out=90] (2.9,1);
\draw[white,line width=1mm,in=-90,out=90] (2.9,-2.45) to[in=-90,out=90] (0,-.4) to (0,.4) to[in=-90,out=90] (2.9,2.4);
\draw[red] (2.9,-2.45) to[in=-90,out=90] (0,-.4) to (0,.4) to[in=-90,out=90] (2.9,2.4);
\draw[red] (2.9,1) to (2.9,1);
\draw[red] (1.4,-.4) to (1.4,-.4);
\draw[red] (0,-.4) to (0,-.4);
\draw[red] (2.9,2.4);
\node[right] at (2.1,0) {$ \cdots $};
\node[right] at (3,0) {$ \cdots $};
\node[right] at (3,2.1) {$ \cdots $};
\node[right] at (3,-2.1) {$ \cdots $};
\node[right] at (3,3.3) {$ \cdots $};
\node[right] at (3,-3.3) {$ \cdots $};
\node[right] at (.9,3) {$ \cdots $};
\node[right] at (.9,-3) {$ \cdots $};
\node[draw,thick,rounded corners, fill=white,minimum height=20,minimum width=0] at (0,0) {$\kappa^{(k+l)}$};
\node[draw,thick,rounded corners, fill=white,minimum height=20,minimum width=0] at (1.4,0) {$\eta^{(k+l)}$};
\node[draw,thick,rounded corners, fill=white,minimum height=20,minimum width=40] at (3.4,-1.4) {$\xi_1$};
\node[draw,thick,rounded corners, fill=white,minimum height=20,minimum width=40] at (3.4,1.4) {$ \xi^*_2 $};
\node[draw,thick,rounded corners, fill=white,minimum height=20,minimum width=40] at (3.4,-2.8) {$\zeta_1$};
\node[draw,thick,rounded corners, fill=white,minimum height=20,minimum width=40] at (3.4,2.8) {$ \zeta^*_2 $};
\end{tikzpicture}
}
= \lim_{l\to \infty} \;  \t{Tr}_{C_{k}}
\raisebox{-3.35cm}{
	\begin{tikzpicture}
		\draw[dashed,thick] (3.9,-3.5) to (3.9,3.5);
\draw (-.7,0) to[in=180,out=-90] (1,-2) to[in=-90,out=0] (2.2,-.4) to (2.2,.4) to[in=0,out=90] (1,2) to[in=90,out=180] (-.7,0);
\draw (-1.3,0) to[in=180,out=-90] (1,-2.6) to[in=-90,out=0] (2.8,0) to[in=0,out=90] (1,2.6) to[in=90,out=180] (-1.3,0);
		\draw (3.1,-3.5) to (3.1,3.5);
		\draw (3.7,-3.5) to (3.7,3.5);
		\draw[white,line width=1mm,in=-90,out=90] (2.9,-1.05) to[in=-90,out=90] (1.4,-.4) to (1.4,.4) to[in=-90,out=90] (2.9,1);
		\draw[red] (2.9,-1.05) to[in=-90,out=90] (1.4,-.4) to (1.4,.4) to[in=-90,out=90] (2.9,1);
		\draw[white,line width=1mm,in=-90,out=90] (2.9,-2.45) to[in=-90,out=90] (0,-.4) to (0,.4) to[in=-90,out=90] (2.9,2.4);
		\draw[red] (2.9,-2.45) to[in=-90,out=90] (0,-.4) to (0,.4) to[in=-90,out=90] (2.9,2.4);
		\draw[red] (2.9,1) to (2.9,1);
		\draw[red] (1.4,-.4) to (1.4,-.4);
		\draw[red] (0,-.4) to (0,-.4);
		\draw[red] (2.9,2.4);
		\node[right] at (2.1,0) {$ \cdots $};
		\node[right] at (3,0) {$ \cdots $};
		\node[right] at (3,2.1) {$ \cdots $};
		\node[right] at (3,-2.1) {$ \cdots $};
		\node[right] at (3,3.3) {$ \cdots $};
		\node[right] at (3,-3.3) {$ \cdots $};
		\node[left] at (-.55,0) {$ \cdots $};
		\node[draw,thick,rounded corners, fill=white,minimum height=20,minimum width=0] at (0,0) {$\kappa^{(k+l)}$};
		\node[draw,thick,rounded corners, fill=white,minimum height=20,minimum width=0] at (1.4,0) {$\eta^{(k+l)}$};
		\node[draw,thick,rounded corners, fill=white,minimum height=20,minimum width=40] at (3.4,-1.4) {$\xi_1$};
		\node[draw,thick,rounded corners, fill=white,minimum height=20,minimum width=40] at (3.4,1.4) {$ \xi^*_2 $};
		\node[draw,thick,rounded corners, fill=white,minimum height=20,minimum width=40] at (3.4,-2.8) {$\zeta_1$};
		\node[draw,thick,rounded corners, fill=white,minimum height=20,minimum width=40] at (3.4,2.8) {$ \zeta^*_2 $};
	\end{tikzpicture}
}
\]
	\[
	= \lim_{l\to \infty} \;  \left\lab \Phi_{\left[S_{k+1} \cdots S_{k+l} \left( \Omega_{k+l}^2 \left( \eta^{(k+l)} \right) \circ \kappa^{(k+l)}_{\Lambda^1_{k+l}} \right)\right]_{\Gamma_k \cdots \Gamma_1 m_0}} \left(\Omega_k^1 (\xi_1) \circ \zeta_1\right) \; , \; \Omega_k^2 (\xi_2) \circ \zeta_2 \right\rab_{F^2_k} \; .
	\]
	Since $ NT \left(\Omega_k^1 \Lambda_k^1  , \Omega_k^2 \Lambda_k^2 \right) $ has finite dimension and sits injectively in $ \mcal L (F^1_k , F^2_k) $ via $ \left. \Phi_\bullet \right|_{F^1_k}  $, the limit in part (a) indeed converges.
The rest is already taken care by the construction.
\end{proof}

%% file: flat2intwin.tex
\section{Flatness}\ 

We have seen that a BQFS depends solely on the loop operator.
In order to understand when a BQFS turns out to be flat,
analyzing the loop operator becomes crucial.
We take on this job next.

Let $ \Lambda_\bullet \ \t{and} \ \Omega_\bullet  $ be two $ 1 $-cells from the $ 0 $-cell $ \Gamma_\bullet $ to $ \Delta_\bullet $ in ${\textbf{UC}^\t{tr}}$.
We will work with the adjoints of the functors $ \Gamma_k $'s, $ \Delta_k $'s, $ \Lambda_k $'s, and solution to conjugate equations commensurate with the given weight functions associated to the objects in WSSC*Cat (defined in \Cref{NaTrace}).
\begin{prop}\label{Xrelandloop}
(a) If the spaces $\normalfont \t{NT}\left(\Lambda_k , \Omega_k \right) $ and $\normalfont \t{NT}\left(\Lambda_{k-1} , \Omega_{k-1} \right) $ are equipped with the inner product induced by the categorical traces $\normalfont \t{Tr}^{\Lambda_k} $ and $ \normalfont\t{Tr}^{\Lambda_{k-1}}$ (as defined in \Cref{NTTr}(a) ) respectively, then the adjoint of the loop operator $ S_k : \normalfont\t{NT}\left(\Lambda_k , \Omega_k \right) \lra \t{NT}\left(\Lambda_{k-1} , \Omega_{k-1} \right)$ is given by
\[
\normalfont\t{NT}  (\Lambda_{k-1} , \Omega_{k-1}) \ni \kappa \ \os{\displaystyle S^*_k}{\longmapsto}
\raisebox{-1.6cm}{
	\begin{tikzpicture}
		\draw (-.3,0) ellipse (7mm and 12mm);
		\draw[white,line width=1mm,in=-90,out=90] (-1.2,-1.5) to (-0.6,-.25);
		\draw[red,in=-90,out=90] (-1.2,-1.5) to (-0.6,-.25);
		\draw[white,line width=1mm,in=-90,out=90] (-.6,.25) to (-1.2,1.5);
		\draw[red,in=-90,out=90] (-.6,.25) to (-1.2,1.5);
		\node[right] at (-.75,.5) {$ \Omega_{k-1} $};
		\node[right] at (-1.25,1.5) {$ \Omega_{k} $};
		\node[right] at (-.85,-.65) {$\Lambda_{k-1} $};
		\node[right] at (-1.3,-1.4) {$ \Lambda_{k} $};
		\node[draw,thick,rounded corners, fill=white] at (-.6,0) {$\kappa$};
		\node[left] at (-.9,0) {$ \Delta_k $};
	\end{tikzpicture}
}
\in \t{NT}  (\Lambda_{k} , \Omega_{k}) \ .
\]
(b) For $ \eta \in \normalfont\t{NT}\left( \Lambda_k , \Omega_k\right) $, the pair $ \left(S_k \eta , \eta\right) $ satisfy the exchange relation (as in \Cref{XreldefpreC*}) if and only if $ S^*_k S_k \eta = 
\raisebox{-.75cm}{
\begin{tikzpicture}
\draw[red,in=-90,out=90] (-.6,-.75) to (-.6,.75);
		\node[right] at (-.7,.45) {$ \Omega_{k} $};
		\node[right] at (-.7,-.55) {$\Lambda_{k} $};
\node[draw,thick,rounded corners, fill=white] at (-.6,0) {$\eta$};
\draw (0.7,0) circle (4mm);
\node[left] at (.5,-.2) {$ \Gamma_k$};
\node[right] at (1,-.2) {$ \Gamma'_k$};
\end{tikzpicture}
} $
where $ \Gamma'_k $ is an adjoint of $\Gamma_k$ and the loop is the natural transformation from $\normalfont \t{id}_{\mcal M_k} $ to $ \normalfont\t{id}_{\mcal M_k} $ coming from the solution to the conjugate equation commensurate with $ \left(\ul \mu^{k-1} , \ul \mu ^k\right) $. [cf. \cite{Jo1} Theorem 2.11.8]
\end{prop}
\begin{proof}
(a) Using \Cref{NTTr} multiple times, the inner product $ \left\lab S_k \eta , \kappa \right \rab $ turns out to be
\[
= \t{Tr}^{\Delta_{k-1} \Lambda_{k-1}}
\raisebox{-1.1cm}
{\begin{tikzpicture}
\draw[in=-90,out=90] (0,-.8) to (.6,-.3) to (.6, .3) to (0,.8) to (0,1.6);
\draw[white,line width=1mm,in=-90,out=90] (.6,-.8) to (0,-.3);
\draw[red,in=-90,out=90] (.6,-.8) to (0,-.3);
\draw[red,in=-90,out=90] (0,-.3) to (0,.3);
\draw[white,line width=1mm,in=-90,out=90] (0,.3) to (0.6,.75);
\draw[red,in=-90,out=90] (0,.3) to (0.6,.75);
\draw[red,in=-90,out=90] (0.6,.75) to (.6,1.6);
\node[draw,thick,rounded corners, fill=white] at (0,0) {$\eta$};
\node[draw,thick,rounded corners, fill=white] at (0.6,1.05) {$\kappa^*$};
\end{tikzpicture}
}= \t{Tr}^{\Lambda_k \Gamma_k}
\raisebox{-1.1cm}
{\begin{tikzpicture}
	\draw[in=-90,out=90] (0.6,-.7) to (.6,-.3) to (.6, .3) to (0,.8) to (0,1.3) to (.6,1.8);
	\draw[red,in=-90,out=90] (0,-.7) to (0,-.3);
	\draw[red,in=-90,out=90] (0,-.3) to (0,.3);
	\draw[white,line width=1mm,in=-90,out=90] (0,.3) to (0.6,.75);
	\draw[red,in=-90,out=90] (0,.3) to (0.6,.75);
	\draw[red,in=-90,out=90] (0.6,.75) to (.6,1.3);
\draw[white,line width=1mm,in=-90,out=90] (.6,1.3) to (0,1.8);
\draw[red,in=-90,out=90] (.6,1.3) to (0,1.8);
	\node[draw,thick,rounded corners, fill=white] at (0,0) {$\eta$};
	\node[draw,thick,rounded corners, fill=white] at (0.6,1.05) {$\kappa^*$};
\end{tikzpicture}
}= \t{Tr}^{\Lambda_k }
\raisebox{-1.1cm}
{\begin{tikzpicture}
	\draw (1.2, .8) to[in=-90,out=-90] (0,.8) to (0,1.3) to[in=90,out=90] (1.2,1.3) to (1.2,.8);
	\draw[red,in=-90,out=90] (0,-.7) to (0,-.3);
	\draw[red,in=-90,out=90] (0,-.3) to (0,.3);
	\draw[white,line width=1mm,in=-90,out=90] (0,.3) to (0.6,.75);
	\draw[red,in=-90,out=90] (0,.3) to (0.6,.75);
	\draw[red,in=-90,out=90] (0.6,.75) to (.6,1.3);
	\draw[white,line width=1mm,in=-90,out=90] (.6,1.3) to (0,1.8);
	\draw[red,in=-90,out=90] (.6,1.3) to (0,1.8);
	\node[draw,thick,rounded corners, fill=white] at (0,0) {$\eta$};
	\node[draw,thick,rounded corners, fill=white] at (0.6,1.05) {$\kappa^*$};
\end{tikzpicture}
} = \left\lab \eta , S^*_k \kappa \right\rab
\]
where $ \eta\in \t{NT}\left(\Lambda_k ,\Omega_k\right) $ and $ \kappa\in \t{NT}\left(\Lambda_{k-1} ,\Omega_{k-1}\right) $.

(b) The `only if' part easily follows from the pictorial relations.

\noindent\textit{if part:} Consider the maps
\begin{align*}
\t{NT} \left(\Lambda_k , \Omega_k\right) \ni \sigma \os{\displaystyle f} \longmapsto \raisebox{-.7cm}
{\begin{tikzpicture}
		\draw[in=-90,out=90] (0,-.75) to (.6,-.3) to (.6, .3) to (0,.75);
		\draw[white,line width=1mm,in=-90,out=90] (.6,-.75) to (0,-.25);
		\draw[red,in=-90,out=90] (.6,-.75) to (0,-.25);
		\draw[red,in=-90,out=90] (0,-.25) to (0,.3);
		\draw[white,line width=1mm,in=-90,out=90] (0,.3) to (0.6,.75);
		\draw[red,in=-90,out=90] (0,.3) to (0.6,.75);
		\node[draw,thick,rounded corners, fill=white] at (0,0) {$\sigma$};
	\end{tikzpicture}
} \in \t{NT} \left( \Delta_k \Lambda_{k-1} ,\Delta_k \Omega_{k-1} \right)\\
\t{NT} \left(\Lambda_{k-1} , \Omega_{k-1}\right) \ni \tau \os{\displaystyle g} \longmapsto \raisebox{-.75cm}
{\begin{tikzpicture}
		\draw[in=-90,out=90] (0.2,-.75) to (0.2,.75);
		\draw[red,in=-90,out=90] (.6,-.75) to (0.6,-.25);
		\draw[red,in=-90,out=90] (0.6,-.25) to (0.6,.75);
		\node[draw,thick,rounded corners, fill=white] at (0.6,0) {$\tau$};
	\end{tikzpicture}
} \in  \t{NT} \left( \Delta_k \Lambda_{k-1} ,\Delta_k \Omega_{k-1} \right)
\end{align*}
and the subspace $ Q $ of $ \t{NT} \left( \Delta_k \Lambda_{k-1} ,\Delta_k \Omega_{k-1} \right) $ generated by the ranges of $ f $ and $ g $.
Let $ \eta $ satisfy the hypothesis.
We need to establish the equation $ f(\eta) = g \left( S_k \eta \right) $.
It is enough to show that $ \left\lab f(\eta) , \chi\right \rab  = \left\lab g\left(S_k \eta\right) , \chi\right \rab  $ for all $ \chi \in Q $ where the inner product is induced by $ \t{Tr}^{\Delta_k \Lambda_{k-1} }$.

For $ \sigma \in \t{NT} \left(\Lambda_k , \Omega_k\right) $, we get (from the `categorical trace' property) $ \left\lab f(\eta), f(\sigma)\right\rab = \t{Tr}^{\Lambda_k \Gamma_k} \left(  [\sigma^* \eta  ]_{\Gamma_k} \right) $ which by \Cref{NTTr}(b) becomes
\[
\t{Tr}^{\Lambda_k} \left( \hspace*{-2.5mm}
\raisebox{-.75cm}{
	\begin{tikzpicture}
		\draw[red,in=-90,out=90] (-.6,-.75) to (-.6,.75);
		\node[right] at (-.7,.5) {$ \Lambda_{k} $};
		\node[right] at (-.7,-.6) {$\Lambda_{k} $};
		\node[draw,thick,rounded corners, fill=white] at (-.6,0) {$\sigma^* \eta$};
		\draw (0.9,0) circle (4mm);
		\node[left] at (.7,-.2) {$ \Gamma_k$};
		\node[right] at (1.2,-.2) {$ \Gamma'_k$};
	\end{tikzpicture}
} \hspace*{-3.5mm} \right) = \t{Tr}^{\Lambda_k} \left(\sigma^* \ S^*_k S_k (\eta) \right) = \t{Tr}^{\Lambda_{k-1}} \left(S_k \left(\sigma^*\right) \  S_k (\eta) \right)
\]
where the last equality follows from part (a).
Applying \Cref{NTTr}(b) and categorical trace property again on the the last expression, we get $ \t{Tr}^{\Delta_k \Lambda_{k-1}} \left( \left[f(\sigma)\right]^* \ g \left(S_k (\eta)\right) \right) = \left\lab g \left(S_k (\eta)\right) , f(\sigma) \right \rab$.

For $ \tau \in \t{NT}\left( \Lambda_{k-1} , \Omega_{k-1} \right) $, we use \Cref{NTTr}(b) and deduce
\[
\left\lab f(\eta), g(\tau) \right \rab = \t{Tr}^{\Delta_k \Lambda_k} \left( \left[g(\tau)\right]^* \ f(\eta)\right) = \t{Tr}^{\Lambda_k} \left( \tau^* \ S_k (\eta)  \right)\ .
\]
which by \Cref{trivial-loop-anticlock} along with \Cref{NTTr}(b) turns out to be $  \left\lab g\left(S_k(\eta)\right) , g(\tau) \right \rab  $.
\end{proof}
\begin{rem}
Similar to \Cref{Xrelandloop}, one can prove that for $ \kappa \in \t{NT} \left(\Lambda_{k-1}, \Omega_{k-1} \right) $, $ \left( \kappa \ , \ S^*_k \kappa \odot \left[
\raisebox{-.4cm}{
	\begin{tikzpicture}
		\draw (0,0) circle (4mm);
		\node[left] at (-.2,-.2) {$ \Gamma_k$};
		\node[right] at (.3,-.2) {$ \Gamma'_k$};
	\end{tikzpicture}
}\right]^{-1} \right) $ satisfy exchange relation if and only if 
\[
S_k \left( S^*_k \kappa \  \odot \left[
\raisebox{-.4cm}{
	\begin{tikzpicture}
		\draw (0,0) circle (4mm);
		\node[left] at (-.2,-.2) {$ \Gamma_k$};
		\node[right] at (.3,-.2) {$ \Gamma'_k$};
	\end{tikzpicture}
}\right]^{-1}\right) = \kappa
\]
where $ \odot $ stands for tensor product of natural transformations.
\end{rem}
\begin{rem}
If the $ 0 $-cell $ \left( \Gamma_\bullet , \ul \mu^\bullet\right) $ satisfy an extra condition that $ \raisebox{-.4cm}{
	\begin{tikzpicture}
		\draw (0,0) circle (4mm);
		\node[left] at (-.2,-.2) {$ \Gamma_k$};
		\node[right] at (.3,-.2) {$ \Gamma'_k$};
	\end{tikzpicture}
} $ is trivial (that is, $ \Gamma_k \ul \mu^{k-1} = d_k \ \ul \mu^{k} $ for some $ d_k > 0 $) eventually for all $ k $, then a BQFS $ \left\{\eta^{(k)}\right\}_{k\geq 0} $ is flat if and only if $ \eta^{(k)} $ is an eigenvector of $ S^*_k S_k $ with respect to the eigenvalue $ d_k $ eventually for all $ k $.
\end{rem}
\subsection{Periodic case}$  $

In this section, we focus on a particular case where the $ 0 $- and the $ 1 $-cells are periodic in nature, and investigate whether the $ 2 $-cells are flat.
To do that, we need a fact in linear algebra which could very well be standard; we give a proof nevertheless.
\begin{defn}
For $ X \in M_n (\C) $, a sequence $ \left\{ \ul x^{(k)} \right\}_{k\geq 0} $ in $ \C^n $ is called \textit{$ X $-harmonic} if $ X \ul x^{{(k+1)}} = \ul x^{(k)}$ for all $ k \geq 0 $.
\end{defn}
\begin{prop}\label{harmonicprop}
	If spectral radius of $ X \in M_n (\C) $ is at most $ 1 $, then the space of bounded $ X $-harmonic sequences are spanned by elements of the form $ \left\{ \lambda^{-k} \ul a \right\}_{k\geq 0} $ where $ \lambda $ is an eigenvalue of $ X $ and $ \ul a $ is a corresponding eigenvector such that $ \abs \lambda = 1 $.
\end{prop}
\begin{proof}
	Let $X = T Y T^{-1}$ where $ Y $ is in Jordan canonical form.
	Note that $ \left\{\ul y^{(k)}\right\}_{k\geq 0} $ is bounded $ Y $-harmonic if and only if $ \left\{T \ul y^{(k)}\right\}_{k\geq 0} $ is bounded $ X $-harmonic.
	Suppose $ \left\{p_s : 1\leq s \leq t   \right\} $ be projections in $ M_n $ such that $I_n =  \us{1\leq s \leq t}\sum p_s$, and for each $ s $,  $ p_s Y = Y p_s $ has exactly one nonzero Jordan block corresponding an eigen value, say, $ \lambda_s $.
	As a result, any bounded $ Y $-harmonic sequence $ \left\{\ul y^{(k)} \right\}_{k\geq 0} $ splits into the sum of $ \left\{p_s\, \ul y^{(k)} \right\}_{k\geq 0} $
	which is bounded $ Y $-harmonic as well as $ (p_s \, Y) $-harmonic.
	So, it becomes essential to find bounded harmonic sequences for a Jordan block.
	
	Suppose $ J = 
	\begin{bmatrix}
		\lambda & 1&0&\cdots &0&0\\
		0 &\lambda &1&\cdots&0&0\\
		0&0&\lambda&\cdots&0&0\\
		\cdot&\cdot&\cdot&\cdots&\cdot&\cdot\\
		0&0&0&\cdots&\lambda&1\\
		0&0&0&\cdots&0&\lambda\\
	\end{bmatrix}_{m\times m}
	= \lambda \, I_m +N $ where $ \abs \lambda \leq 1 $.
	If $ \lambda = 0 $, then the only $ J $-harmonic sequence would be the zero sequence.
	So, let us assume $ \lambda \neq 0 $.
	Any nonzero $ J $-harmonic sequence is of the form $ \left\{ J^{-k}\, \ul x\right\}_{k\geq 0} $ for some nonzero vector $ \ul x \in \C^m $; however, it may not always be bounded as $ k  $ varies.
	Now, $ J^{-k} = \lambda^{-k} \ous{m-1}{\sum}{l=0}  {k+l-1 \choose l}  \left[-\lambda^{-1} N\right]^l$ for $ k\geq 1 $.
	Fix nonzero $ \ul{x} \in \C^m$.
	Set $ t \coloneqq \max\{l:x_l \neq 0\} $ and $ C = \max\{\abs{x_l} : 1\leq l \leq m\} $.
	Hence
	\[
	\left[ J^{-k} \ul x\right]_1 = \lambda^{-k} \ous{t-1}{\sum}{l=0}  {k+l-1 \choose l}  \left(-\lambda^{-1} \right)^l x_{l+1} \; .
	\]
	Since $0 < \abs \lambda \leq 1 $ and $ {k+l-1 \choose l} $ increase as $ l $ increases, we have the following inequality if $ t > 1 $
	\[
	\abs{ \left[ J^{-k} \ul x\right]_1 } \; \geq \;  \abs{\lambda}^{-k} \left[{k+t-2 \choose t-1}\abs{\lambda}^{-(t-1)} \abs{x_t} \; - \; {k+t-3 \choose t-2}\abs{\lambda}^{-(t-2)} C t \right] \; .
	\]
	If $ t > 1 $, then
	\begin{align*}
		\abs{ \left[ J^{-k} \ul x\right]_1 } \; \geq \; 
		\abs{\lambda}^{-(k+t-1)} {k+t-3 \choose t-2} \left[\frac{k+t-2}{t-1}\abs{x_t} \; - \; \abs{\lambda} C t \right] \to \infty \t{ as } k \to \infty \; .
	\end{align*}
	Thus, in order to have a \textbf{bounded} $ J $-harmonic sequence, $\ul x$ must be in the unique one-dimensional eigen space of $ J $, that is, $ \C\, \ul e_1 $.
	On the other hand, if $ \abs \lambda < 1 $ and $ t=1 $, then
	\begin{align*}
		\abs{ \left[ J^{-k} \ul x\right]_1 } \; = \; 
		\abs{\lambda}^{-k} \abs{x_1}  \to \infty \t{ as } k \to \infty \; .
	\end{align*}
	So, a nonzero bounded $ J $-harmonic sequence exists only when $ \abs \lambda = 1 $, and then it is a scalar multiple of $ \left\{ \lambda^{-k} \ul e_1 \right\}_{k\geq 0} $.
	
	Getting back to the matrix $ Y $ (which is in Jordan canonical form) and applying the above result, we may conclude that all bounded $ Y $-harmonic sequences are linear combination of sequences of the form
	\[
	\left\{ \lambda^{-k} \ul y \right\}_{k\geq 0}
	\]
	where $ \lambda $ is an eigenvalue with absolute value $ 1 $ and $ \ul y $ is a corresponding eigenvector.
	By similarity, the same result holds for $ X $ too.
\end{proof}
\begin{prop}\label{periodicOcneanucompact}
Let $ \left(\Lambda_\bullet, W^\Lambda_\bullet\right) $ and $ \left(\Omega_\bullet  , W^{\Omega}_\bullet \right) $ be two $ 1 $-cells from the $ 0 $-cell $ \left(\Gamma_\bullet, \ul \mu^\bullet\right) $ to $ \left(\Delta_\bullet, \ul \nu^\bullet\right) $ in $\normalfont{\textbf{UC}^\t{tr}}$ such that there exists a `period' $ K\in \N $ and a `Perron-Frobenius (PF) value' $ d >0 $ satisfying:
\begin{itemize}
\item[(i)] the periodic condition: $ \Gamma_k $'s, $ \Delta_k $'s, $ \Lambda_k $'s, $ \Omega_k $'s, $ W^\Lambda_k $'s and $W^\Omega_k$'s repeat with a periodicity $ K $ eventually for all $ k $,\\
\item[(ii)] the PF condition:
\begin{align*}
d^{-1} \ \Gamma_{k+K} \cdots \Gamma_{k+1}\ \ul \mu^{k} = \ul \mu^k = d\ \ul \mu^{k+K}\\
d^{-1} \ \Delta_{k+K} \cdots \Delta_{k+1}\ \ul \nu^{k} = \ul \nu^k = d\ \ul \nu^{k+K}
\end{align*}
eventually for all $ k $.
\end{itemize}
Then, all {\normalfont BQFS} from $ \left(\Lambda_\bullet, W^\Lambda_\bullet\right) $ to $ \left(\Omega_\bullet  , W^\Omega_\bullet \right) $ are flat.
\end{prop}
\begin{proof}
Choose a level $ L\in \N $ large enough after which the ingredients in (i) keep repeating with periodicity $ K $, and the equations in (ii) hold.
Set\\
$ \mcal M \coloneqq \mcal M_L = \mcal M_{L+nK}$\\
$ \mcal N \coloneqq \mcal N_L = \mcal N_{L+nK}$\\
$ \Gamma \coloneqq \Gamma_{L+K} \cdots \Gamma_{L+1}  = \Gamma_{L+(n+1)K} \cdots \Gamma_{L+nK+1} : \mcal M \lra \mcal M$\\
$ \Delta \coloneqq \Delta_{L+K} \cdots \Delta_{L+1}  = \Delta_{L+(n+1)K} \cdots \Delta_{L+nK+1} : \mcal N \lra \mcal N$\\
$ \Lambda \coloneqq \Lambda_{L}  = \Lambda_{L+nK} : \mcal M \lra \mcal N$\\
$\Omega \coloneqq \Omega_{L}  = \Omega_{L+nK}  : \mcal M \lra \mcal N$\\
$W^\Lambda \coloneqq \raisebox{-1.1cm}{
\begin{tikzpicture}
\draw[in=-90,out=90] (0,0) to (1,1.5);
\draw[in=-90,out=90] (.75,0) to (1.75,1.5);
\draw[white,line width=1mm,in=-90,out=90] (1.75,0) to (0,1.5);
\draw[red,in=-90,out=90] (1.75,0) to (0,1.5);
\node[right] at (0,.1) {$ \cdots$};
\node[right] at (1,1.4) {$ \cdots$};
\node[right] at (1.65,.1) {$ \Lambda_{L+nK}$};
\node[left] at (.1,1.3) {$ \Lambda_{L+(n+1)K}$};
\node[right] at (1.65,1.3) {$ \Gamma_{L+nK+1}$};
\node[right] at (.25,1.9) {\rotatebox{70} {$ \Gamma_{L+(n+1)K}$}};
\node[right] at (.55,-.2) {$ \Delta_{L+nK+1}$};
\node[left] at (.1,.1) {$ \Delta_{L+(n+1)K}$};
\end{tikzpicture}
}$ and also denoted by \raisebox{-.8cm}{
\begin{tikzpicture}
	\draw[in=-90,out=90] (0,0) to (1,1.5);
	\draw[white,line width=1mm,in=-90,out=90] (1,0) to (0,1.5);
	\draw[red,in=-90,out=90] (1,0) to (0,1.5);
	\node[left] at (1.1,.2) {$ \Lambda$};
	\node[right] at (0,1.3) {$ \Lambda$};
	\node[left] at (1.1,1.3) {$ \Gamma$};
	\node[right] at (-.1,.2) {$ \Delta$};
\end{tikzpicture}
}\\
$W^\Omega \coloneqq \raisebox{-1.1cm}{
\begin{tikzpicture}
	\draw[in=-90,out=90] (0,0) to (1,1.5);
	\draw[in=-90,out=90] (.75,0) to (1.75,1.5);
	\draw[white,line width=1mm,in=-90,out=90] (1.75,0) to (0,1.5);
	\draw[red,in=-90,out=90] (1.75,0) to (0,1.5);
	\node[right] at (0,.1) {$ \cdots$};
	\node[right] at (1,1.4) {$ \cdots$};
	\node[right] at (1.65,.1) {$ \Omega_{L+nK}$};
	\node[left] at (.1,1.3) {$\Omega_{L+(n+1)K}$};
	\node[right] at (1.65,1.3) {$ \Gamma_{L+nK+1}$};
	\node[right] at (.25,1.9) {\rotatebox{70} {$ \Gamma_{L+(n+1)K}$}};
	\node[right] at (.55,-.2) {$ \Delta_{L+nK+1}$};
	\node[left] at (.1,.1) {$ \Delta_{L+(n+1)K}$};
\end{tikzpicture}
}$ and also denoted by \raisebox{-.8cm}{
\begin{tikzpicture}
	\draw[in=-90,out=90] (0,0) to (1,1.5);
	\draw[white,line width=1mm,in=-90,out=90] (1,0) to (0,1.5);
	\draw[red,in=-90,out=90] (1,0) to (0,1.5);
	\node[left] at (1.1,.2) {$ \Omega$};
	\node[right] at (0,1.3) {$\Omega$};
	\node[left] at (1.1,1.3) {$ \Gamma$};
	\node[right] at (-.1,.2) {$ \Delta$};
\end{tikzpicture}
}\\
$ \ul \mu \coloneqq \ul \mu^L  = d^n \ \ul \mu^{L+nK}  $\\
$ \ul \nu \coloneqq \ul \nu^L  = d^n \ \ul \nu^{L+nK}  $\\
for any $ n\geq 0 $.
Now condition (ii) and \Cref{trivial-loop-anticlock} imply the following relations:
\[
\Gamma \ul \mu = d\ \ul \mu = \Gamma' \ul \mu \ \t{ and } \ \Delta \ul \nu = d\ \ul \nu = \Delta' \ul \nu \ .
\]
Consider the loop operators given by
\[
S \coloneqq d^{-1}
\raisebox{-1.4cm}{
	\begin{tikzpicture}
		\draw (-.3,0) ellipse (7mm and 12mm);
		\draw[white,line width=1mm,in=-90,out=90] (.6,-1.5) to (0,-.29);
		\draw[red,in=-90,out=90] (.6,-1.5) to (0,-.29);
		\draw[white,line width=1mm,in=-90,out=90] (0,.29) to (.6,1.5);
		\draw[red,in=-90,out=90] (0,.29) to (.6,1.5);
		\node[left] at (.1,.5) {$ \Omega $};
		\node[right] at (.4,1.1) {$ \Omega$};
		\node[left] at (.15,-.5) {$\Lambda$};
		\node[right] at (.45,-1.2) {$ \Lambda $};
		\node[draw,thick,rounded corners, fill=white,minimum width=15] at (0,0) {$\vphantom{\eta}$};
		\node[right] at (.3,0) {$ \Gamma $};
		\node[right] at (-.45,-.9) {$ \Delta$};
		\node[right] at (-.45,.95) {$ \Delta $};
\node[right] at (-1.1,0) {$ \Delta' $};
	\end{tikzpicture}
}
\ \t{ and } \
S^* \coloneqq d^{-1}
\raisebox{-1.45cm}{
	\begin{tikzpicture}
		\draw (-.3,0) ellipse (7mm and 12mm);
		\draw[white,line width=1mm,in=-90,out=90] (-1.2,-1.5) to (-0.6,-.25);
		\draw[red,in=-90,out=90] (-1.2,-1.5) to (-0.6,-.25);
		\draw[white,line width=1mm,in=-90,out=90] (-.6,.25) to (-1.2,1.5);
		\draw[red,in=-90,out=90] (-.6,.25) to (-1.2,1.5);
		\node[right] at (-.7,.45) {$ \Omega $};
		\node[right] at (-1.25,1.4) {$ \Omega $};
		\node[right] at (-.7,-.45) {$\Lambda $};
		\node[right] at (-1.3,-1.3) {$ \Lambda $};
		\node[draw,thick,rounded corners, fill=white,minimum width=15] at (-.6,0) {$\vphantom{\kappa}$};
		\node[left] at (-.9,0) {$ \Delta$};
\node[left] at (.5,0) {$ \Gamma'$};
\node[right] at (-.75,.9) {$ \Gamma$};
\node[right] at (-.75,-.9) {$ \Gamma$};
\end{tikzpicture}
}
: \t{NT}  (\Lambda , \Omega) \lra \t{NT}  (\Lambda , \Omega)
\]
where we use tracial solution to the conjugate equation for $ \Gamma\in\t{End} \left(\mcal M\right) $ (resp., $ \Delta \in\t{End} \left(\mcal N\right)$) commensurate with the weight function $ \ul \mu $ on $ \mcal M $ (resp., $ \ul \nu $ on $ \mcal N $) for both source and target, and the crossings are given by $ W^\Lambda, W^{\Lambda^*}, W^\Omega ,  W^{\Omega^*} $.

Observe that $ S =  S_{L+1} \cdots S_{L+K} = S_{L+nK +1}  \cdots S_{L+(n+1)K} $ for all $ n \geq 0$.
To see this, note that an $ S_k $ in composition $ \left[S_{L+1} \cdots S_{L+K}\right] $ is defined using tracial solution to conjugate equation for $ \Delta_k $ commensurate with $ \ul \nu^{k-1} $ and $ \ul \nu^{k} $.
So, for the composition $ \left[S_{L+1} \cdots S_{L+K}\right] $, we are effectively using tracial solution to the conjugate equation for $  \Delta_{L+K} \cdots \Delta_{L+1}  = \Delta $ commensurate with the $ \ul \nu^L = \ul \nu$ and $ \ul \nu^{L+K} = d^{-1} \ \ul \nu$; let us denote this solution by $ \left(\t{id}_{\mcal N} \os{\displaystyle \rho} \lra \Delta \Delta' \  , \  \t{id}_{\mcal N} \os{\displaystyle \rho'} \lra \Delta' \Delta \right) $.
The solution to the conjugate equation for $ \Delta $ commensurate with $ \ul \nu $ for both source and target, is given by $ \left(d^{-\frac 1 2 } \rho , d^{\frac 1 2 } \rho' \right) $.
Only $ d^{\frac 1 2 } \rho' $ is used while defining $ S $.
Replacing the cap and the cup by $ [d^{\frac 1 2 } \rho']^* $ and $ d^{\frac 1 2 } \rho' $, we get the desired equation.

We next prove a one-to-one correspondence between bounded $ S $-harmonic sequnces and BQFS's.
Let $ \left\{\eta^{(k)}\right\}_{k\geq 0} $ be a BQFS.
Clearly, $ \left\{ \eta^{(L+nK)}\right\}_{n\in \N} $ becomes and $ S $-harmonic sequence.
Equip $ \t{NT} \left( \Lambda , \Omega \right) $ with the inner product induced by the trace $ \t{Tr}^\Lambda $ commensurate with $ \left(\ul \mu , \ul \nu\right) $.
Finite dimensionality of $ \t{NT} \left(\Lambda , \Omega \right) $ implies that boundedness of a subset in C*-norm is equivalent to that of the $ 2 $-norm.

Conversely, let $ \left\{\kappa_n\right\}_{n \in \N} $ be a bounded $ S $-harmonic sequence.
Set $ \eta^{(k)} \coloneqq S_{k+1} \cdots  S_{L+nK} (\kappa_n)$ for any $ n $ such that $ L+nK > k $.
Indeed $ \eta^{(k)} $ is well-defined and by construction $ \left\{\eta^{(k)}\right\}_{k\geq 0} $ is quasi-flat.
Again by finite dimensionality of $ \t{NT} \left(\Lambda , \Omega \right) $, $ \left\{\kappa_n\right\}_{n\in \N} $ is bounded in C*-norm, and by \Cref{Sprop}(iii), $ \eta^{(k)} $'s become uniformly bounded as well.

In order to apply \Cref{harmonicprop} on $ S $, it is enough to show that operator norm of $ S  $ (acting on the finite dimensional Hilbert space $ \t{NT} \left(\Lambda , \Omega \right) $) is at most $ 1 $.
Note that $ d^{-1} \raisebox{-.5cm}{
\begin{tikzpicture}
\draw[in=90,out=90,looseness=2] (0,0) to  (.5,0);
\draw[in=-90,out=-90,looseness=2] (0,1) to  (.5,1);
\node[right] at (.4,.1) {$ \Delta$};
\node[right] at (.4,.9) {$ \Delta$};
\node[left] at (.1,.1) {$ \Delta'$};
\node[left] at (.1,.9) {$ \Delta'$};
\end{tikzpicture}
}$ is a projection in $ \t{End} \left(\Delta' \Delta\right) $ and hence less than $ 1_{\Delta' \Delta} $.
Using this and applying \Cref{trsol2conj} multiple times, we have
\[
\norm{S \kappa}^2 = \t{Tr}^\Lambda \left(\left(S \kappa\right)^* S \kappa\right) \leq d^{-1}\t{Tr}^\Lambda \left(
\raisebox{-1.4cm}{
	\begin{tikzpicture}
		\draw (-.3,0) ellipse (7mm and 12mm);
		\draw[white,line width=1mm,in=-90,out=90] (.6,-1.5) to (0,-.29);
		\draw[red,in=-90,out=90] (.6,-1.5) to (0,-.29);
		\draw[white,line width=1mm,in=-90,out=90] (0,.29) to (.6,1.5);
		\draw[red,in=-90,out=90] (0,.29) to (.6,1.5);
		\node[left] at (.1,.5) {$ \Lambda $};
		\node[right] at (.4,1.1) {$\Lambda$};
		\node[left] at (.15,-.5) {$\Lambda$};
		\node[right] at (.45,-1.2) {$ \Lambda $};
		\node[draw,thick,rounded corners, fill=white] at (-.2,0) {$\kappa^* \kappa$};
		\node[right] at (.3,0) {$ \Gamma $};
		\node[right] at (-.45,-.9) {$ \Delta$};
		\node[right] at (-.45,.95) {$ \Delta $};
		\node[left] at (-.9,0) {$ \Delta' $};
	\end{tikzpicture}
}\right) = \norm{\kappa}^2 \ .
\]
Thus, by \Cref{harmonicprop}, every bounded $ S $-harmonic sequence turns out to be a linear combination of sequences of the form $ \left\{\lambda^{-n} \kappa \right\}_{n\geq 0} $ where $ \abs{\lambda}  = 1$ and $ \kappa $ is an eigenvector of $ S $ with respect to the eigenvalue $ \lambda $.
We will call such sequences \textit{elementary}.
We will also borrow the notion of flatness for sequences in $ \t{NT}\left(\Lambda , \Omega \right) $ from previous the section when every consecutive pair satisfy the exchange relation with respect to $ \Gamma , \Delta , \Lambda, \Omega, W^\Lambda $ and $ W^\Omega $.
Since flat sequences form a vector space, we may conclude every bounded $ S $-harmonic sequence will be flat if all elementary ones are so.
Consider the above elementary $ S $-harmonic sequence given by $ \lambda $ and $ \kappa $.
It is enough to show
\[
\raisebox{-.75cm}
{\begin{tikzpicture}
		\draw[in=-90,out=90] (0,-.75) to (.6,-.3) to (.6, .3) to (0,.75);
		\draw[white,line width=1mm,in=-90,out=90] (.6,-.75) to (0,-.25);
		\draw[red,in=-90,out=90] (.6,-.75) to (0,-.25);
		\draw[red,in=-90,out=90] (0,-.25) to (0,.3);
		\draw[white,line width=1mm,in=-90,out=90] (0,.3) to (0.6,.75);
		\draw[red,in=-90,out=90] (0,.3) to (0.6,.75);
		\node[draw,thick,rounded corners, fill=white] at (0,0) {$\kappa$};
\node[left] at (.1,-.65) {$ \Delta$};
\node[left] at (.1,.65) {$ \Delta$};
\node[right] at (.5,-.65) {$ \Lambda$};
\node[right] at (.5,.65) {$\Omega$};
\end{tikzpicture}
} = \lambda \ \raisebox{-.75cm}
{\begin{tikzpicture}
		\draw[in=-90,out=90] (0.2,-.75) to (0.2,.75);
		\draw[red,in=-90,out=90] (.6,-.75) to (0.6,-.25);
		\draw[red,in=-90,out=90] (0.6,-.25) to (0.6,.75);
		\node[draw,thick,rounded corners, fill=white] at (0.6,0) {$\kappa$};
\node[left] at (.3,0) {$ \Delta$};
\node[right] at (.5,-.65) {$ \Lambda$};
\node[right] at (.5,.65) {$\Omega$};
\end{tikzpicture}
} \ .
\]
Note that both sides of the above equation belongs to the space $ \t{NT} \left( \Delta\Lambda , \Delta\Omega \right) $.
We equip this space with inner product induced by $ \t{Tr}^{\Delta \Lambda} $ commensurate with $ \left(\ul \mu , \ul \nu \right) $.
Consider the subspace $ \Delta \left(\t{NT} \left(\Lambda , \Omega \right) \right)  \subset \t{NT} \left( \Delta\Lambda , \Delta\Omega \right) $.
It is routine to check that the orthogonal projection onto this subspace is given by
\[
E \coloneqq d^{-1}
\raisebox{-1cm}{
\begin{tikzpicture}
\draw (-.9,-1) to (-.9,1);
\draw[red,in=-90,out=90] (.6,-1) to (0.6,-.29);
\draw[red,in=-90,out=90] (0.6,.29) to (.6,1);
\node[right,scale=0.9] at (.48,-.7) {$ \Lambda $};
\node[right,scale=0.9] at (.48,.7) {$\Omega $};
\node[draw,thick,rounded corners, fill=white,minimum width=20] at (.5,0) {$\vphantom{\eta}$};
\draw[in=90, out=90,looseness=2] (0.4,.30) to (-0.2,.29);
\draw[in=-90, out=-90,looseness=2] (0.4,-.30) to (-0.2,-.29);
\draw (-.2,-.29) to (-.2,.29);
\node[right,scale=0.9] at (-.9,0) {$ \Delta' $};
\node[left,scale=0.9] at (-.8,0.2) {$ \Delta $};
\end{tikzpicture}

} :  \t{NT} \left( \Delta\Lambda , \Delta\Omega \right) \lra \Delta \left( \t{NT} \left( \Lambda , \Omega \right)\right)
\]
Now,
\[
E\left( \raisebox{-.75cm}
{\begin{tikzpicture}
		\draw[in=-90,out=90] (0,-.75) to (.6,-.3) to (.6, .3) to (0,.75);
		\draw[white,line width=1mm,in=-90,out=90] (.6,-.75) to (0,-.25);
		\draw[red,in=-90,out=90] (.6,-.75) to (0,-.25);
		\draw[red,in=-90,out=90] (0,-.25) to (0,.3);
		\draw[white,line width=1mm,in=-90,out=90] (0,.3) to (0.6,.75);
		\draw[red,in=-90,out=90] (0,.3) to (0.6,.75);
		\node[draw,thick,rounded corners, fill=white] at (0,0) {$\kappa$};
		\node[left] at (.1,-.65) {$ \Delta$};
		\node[left] at (.1,.65) {$ \Delta$};
		\node[right] at (.5,-.65) {$ \Lambda$};
		\node[right] at (.5,.65) {$\Omega$};
	\end{tikzpicture}
} \right) = \Delta \left(S \kappa \right) =\lambda \Delta \kappa \ \t{ and  } \ \norm{
\raisebox{-.75cm}
{\begin{tikzpicture}
		\draw[in=-90,out=90] (0,-.75) to (.6,-.3) to (.6, .3) to (0,.75);
		\draw[white,line width=1mm,in=-90,out=90] (.6,-.75) to (0,-.25);
		\draw[red,in=-90,out=90] (.6,-.75) to (0,-.25);
		\draw[red,in=-90,out=90] (0,-.25) to (0,.3);
		\draw[white,line width=1mm,in=-90,out=90] (0,.3) to (0.6,.75);
		\draw[red,in=-90,out=90] (0,.3) to (0.6,.75);
		\node[draw,thick,rounded corners, fill=white] at (0,0) {$\kappa$};
		\node[left] at (.1,-.65) {$ \Delta$};
		\node[left] at (.1,.65) {$ \Delta$};
		\node[right] at (.5,-.65) {$ \Lambda$};
		\node[right] at (.5,.65) {$\Omega$};
	\end{tikzpicture}
}}^2_2 = \norm{\lambda \Delta \kappa}^2_2 \ .
\]
So, the above equation must hold.

From correspondence between bounded $ S $-harmonic sequences and BQFS's and the flatness of the former, we may conculde that in a BQFS $ \left\{\eta^{(k)}\right\}_{k\geq 0} $, two terms which are $ K $ steps apart, must satisfy exchange relation starting from level $ L $ (namely, $ \eta^{(L)}, \eta^{(L+K)}, \eta^{(L+2K)},\ldots  $).
Now, we have the freedom of choosing higher $ L $'s; as a result, we obtain exchange relation of any two terms which are $ K $ steps apart after level $ L $.
To establish exchange relation for consecutive terms, pick a $ k > L $ and recall the maps $ f $ and $ g $ defined in the proof of \Cref{Xrelandloop}(b).
By quasi-flat property, we get
\[
f\left(\eta^{(k)}\right) = f\left(S_{k+1} \cdots S_{k+K-1} \left(\eta^{(k+K-1)}\right)\right) =\raisebox{-1.4cm}{
\begin{tikzpicture}
\draw (0,0) ellipse (1.2cm and .6cm);
\draw (0,0) ellipse (1.8cm and .9cm);
\draw (1,-1.4) to[in=-90,out=90] (2,0) to[in=-90,out=90] (1,1.4);
\draw[white,line width=1mm] (2,-1.4) to[in=-90,out=90] (0,-.4) to (0,.4) to[in=-90,out=90] (2,1.4);
\draw[red] (2,-1.4) to[in=-90,out=90] (0,-.4) to (0,.4) to[in=-90,out=90] (2,1.4);
\node at (1.5,0) {$ \cdots $};
\node at (-1.5,0) {$ \cdots $};
	\node[draw,thick,rounded corners, fill=white] at (0,0) {$\eta^{(k+K-1)}$};
	\end{tikzpicture}}\]
which (by \Cref{trivial-loop-anticlock}, unitarity of the connection and the $ K $-step exchange relation) turns out to be $ g \left( \eta^{(k-1)} \right)$.
Hence, $ \left(\eta^{(k-1)} , \eta^{(k)}\right) $ satisfies the exchange relation.
This ends the proof.
\end{proof}

%% file: eg.tex
\section{Examples}

\subsection{Subfactors}\label{subfactors} \

Subfactors, more specifically, their standard invariants constitute the initial source of examples generalizing which we arrived at our objects of interest, namely, \textbf{UC} and ${\textbf{UC}^\t{tr}}$.
Basically, we associate a $ 1 $-cell in ${\textbf{UC}^\t{tr}}$ to the subfactor which captures all the information of the associated planar algebra.

Let $ \vphantom{X}_{C} X_D  $ be an extremal bifinite bimodule over $ II_1 $ factors $ C, \ D $.
For any bifinite bimodule $ \vphantom{Y}_{A} Y_B $, let $ \left \lab \vphantom{Y}_{A} Y_B  \right \rab $ denote the category of bifinite $ A$-$ B$-bimodules which are direct sum of irreducibles appearing in $\vphantom{Y}_{A} Y_B $.
Set $ \mcal M_0 \coloneqq \mcal H ilb_{fd} $ (the category of finite dimesnional Hilbert spaces), and for $ k \geq 0 $, let

\noindent $ \mcal M_{2k+1} \coloneqq \left \lab { \vphantom{ \left(\ol X \us{C}{\otimes} X\right)^{\us{D}{\otimes} k } } }_{D} \left(\ol X \us{C}{\otimes} X\right)^{\us{D}{\otimes} k }_{D}  \right \rab $ and $ \mcal M_{2k+2} \coloneqq \left \lab { \vphantom{ \left(\ol X \us{C}{\otimes} X\right)^{\us{D}{\otimes} k }  \us{D} \otimes \ol X } }_{D} {\left(\ol X \us{C}{\otimes} X\right)^{\us{D}{\otimes} k } \us{D} \otimes \ol X }_{C}  \right \rab $

\noindent and functors $ \Gamma_k : \mcal M_{k-1} \ra \mcal M_k $ be defined by

\noindent $ \t{ob}(\mcal M_0) \ni \C \os{\displaystyle \Gamma_1}{\longmapsto} { \vphantom{ L^2 (D) } }_D L^2 (D)_D \in\t{ob} (\mcal M_1 )$, $ \Gamma_{2k+2} \coloneqq \left.\bullet \us D \otimes \ol X \right|_{\mcal M_{2k+1}} $ and $ \Gamma_{2k+3} \coloneqq \left.\bullet \us C  \otimes  X \right|_{\mcal M_{2k+2}} $.
The sequence $ \Gamma_\bullet $ will serve as the source $ 0 $-cell in \textbf{UC}.
For the target $ 0 $-cell in \textbf{UC}, define

\noindent $ \mcal N_{2k} \coloneqq \left \lab { \vphantom{ \left( X \us{D}{\otimes} \ol X\right)^{\us{D}{\otimes} k } } }_{C} \left(X \us{D}{\otimes} \ol X\right)^{\us{C}{\otimes} k }_{C}  \right \rab $ and $ \mcal N_{2k+1} \coloneqq \left \lab { \vphantom{ \left( X \us{D}{\otimes} \ol X\right)^{\us{C}{\otimes} k } \us{C} \otimes  X } }_C {\left( X \us{D}{\otimes} \ol X\right)^{\us{C}{\otimes} k } \us{C} \otimes  X }_{D}  \right \rab $

\noindent and $ \Delta_{2k+1} \coloneqq \left.\bullet \us C \otimes X \right|_{\mcal N_{2k}} $ and $ \Delta_{2k+2} \coloneqq \left.\bullet \us D  \otimes \ol X \right|_{\mcal N_{2k+1}} $.
Next, consider the $ 1 $-cell $ \Lambda_\bullet $ in \textbf{UC} given by

\noindent $ \t{ob}(\mcal M_0) \ni \C \os{\displaystyle \Lambda_0}{\longmapsto} { \vphantom{ L^2 (C) } }_C L^2 (C)_C \in\t{ob} (\mcal N_0 )$ and $ \Lambda_k \coloneqq \left. X \us D \otimes \bullet  \right|_{\mcal M_k} $ where the unitary connection for the squares
\[
\xymatrix{\mcal N_{2k-1} \ar[r]_{\Delta_{2k}}^{\bullet \us{D}{\otimes} \ol X }  & \mcal N_{2k} \ar[r]^{\bullet \us{C}{\otimes} X }_{\Delta_{2k+1}} & \mcal N_{2k+1}\\
\mcal M_{2k-1} \ar[r]^{\Gamma_{2k}}_{\bullet \us{D}{\otimes} \ol X } \ar[u]^{ X \us{D}{\otimes} \bullet }_{\Lambda_{2k-1}} & \mcal M_{2k}  \ar[u]^{\Lambda_{2k}}_{ X \us{D}{\otimes} \bullet } \ar[r]^{\Gamma_{2k+1}}_{\bullet \us{C}{\otimes} X } & \mcal M_{2k+1} \ar[u]^{\Lambda_{2k+1}}_{ X \us{D}{\otimes} \bullet }
}
\]
is induced by the associativity constraint of the bimodules.

In order to turn the \textbf{UC}-$ 0 $-cells $ \Gamma_\bullet $ and $ \Delta_\bullet $ into ${\textbf{UC}^\t{tr}}$ ones, we work with the statistical dimension (same as the square root of the index) of an extremal bimodule $ {\vphantom{Y}}_{A} Y_B  $, denoted by $ d (Y) $.
Set $ \delta \coloneqq d (X) $.
Let $ V_{\mcal M_k }$ and $ V_{\mcal N_k }$ denote maximal sets of mutually non-isomorphic family of irreducible bimodules in $ \mcal M_k $ and $ \mcal N_k $ respectively.

Now we define the weight functions $\ul{\mu}^k$ and $\ul{\nu}^k$ on $ V_{\mcal M_k }$ and $ V_{\mcal N_k }$ respectively as follows:
\[\mu^0_\C \coloneqq 1, \ \mu^k_Y \coloneqq \delta^{-(k-1)} d(Y), \ \nu^k_Z \coloneqq \delta^{-k} d(Z) \ \ \ \ \t{for} \ Y \in V_{\mcal M_k }, \ Z \in V_{\mcal N_k }   \]
Since the dimension function is a linear homomorphism with respect to direct sum and Connes fusion of bimodules, we get that $\ul{\mu}^k$ and $\ul{\nu}^k$ satisfy \Cref{trivial-loop-anticlock}.
For the same reason, the boundedness condition \Cref{bndonwt} holds with the inequalities replaced by equality where both the bounds are $ 1 $.

\subsubsection{Planar algebraic view of the associated bimodule}\ 

Let $ A \coloneqq \mcal{PB} \left( \Gamma_\bullet \right)$ and $ B \coloneqq \mcal {PB} \left(\Delta_\bullet \right) $ be the AFD's and $ H \coloneqq \mcal{PB} \left( \Lambda_\bullet \right) $ be the $ A $-$ B $-bimodule where $ \Gamma_\bullet, \Delta_\bullet $ and $ \Lambda_\bullet $ and $ 0 $- and $ 1 $-cells in $\textbf{UC}^\t{tr}$ associated to the extremal bifinite bimodule $ {\vphantom{X}}_C X_D $.
Denote the planar algebra associated to $ {\vphantom{X}}_C X_D $ by $ P = \left\{ P_{\pm k}\right\}_{k\geq 0} $ where the vector spaces are given by\\
$ P_{+k} = \t{End} \left( X \us D \otimes \ol X \us C \otimes \cdots \ k \t{ tensor components}\right)$ and\\
$ P_{-k} = \t{End} \left( \ol X \us C \otimes X \us D \otimes \cdots \ k \t{ tensor components}\right)$.
Immediately from the definitions, we get the following.

(a) $ A_{k+1} = P_{-k} $ and $ B_k = P_{+k} = H_k $.

(b) Both the inclusions $ B_k \hookrightarrow B_{k+1} $ and $ H_k \hookrightarrow H_{k+1} $ are the same as $ P_{+k} \hookrightarrow P_{+(k+1)} $, and $ A_k \hookrightarrow A_{k+1} $ is same as $ P_{-(k-1)} \hookrightarrow P_{-k} $, induced by the action of inclusion tangle by a string on the right.

(c) Action of $B_k$ on $H_k$ is given by right mutiplication of $ P_{+k} $ on itself whereas that of $A_k$ on $H_k$ is given by the left multiplation of the left inclusion $ P_{-(k-1)} \os {\displaystyle \t{LI}} \hookrightarrow P_{+k}$ induced the action of inclusion tangle by a string on the left.

(d) The trace on $ A_k $ and $ B_k $ turns out to be the normalized picture trace on $ P_{-(k-1)} $ and $ P_{+k} $ respectively.

\begin{rem}\label{PAflat}\
 Let $P_{\pm \infty}$ be the union $ \us{k \geq 0}{\cup} P_{\pm k}$ of the filtered unital algebras, and $ P_{\pm}$ be the von Neumann algebra generated by it acting on the GNS with respect to the canonical normalized picture trace $\t{tr}_{\pm}$.
Finally, the bimodule $ {\vphantom{H}}_A H_B $ turns out to be the same as $ {\vphantom{L^2 (P_{+ \infty}, \t{tr}_+)}}_{P_{-}} {L^2 (P_{+ \infty}, \t{tr}_+)}_{P_{+}} $ where the $ P_- $-action on left extends from treating $ P_{+ \infty} $ as a left-module over the subalgebra $ \t{LI} \left(P_{-\infty}\right) $.
As a result, the BQFS's from $ \Lambda_\bullet $ to $ \Lambda_\bullet $ are given by intertwiners in $ {\vphantom{\mcal L}_{P_-} {\mcal L}_{P_+}} \left(L^2 \left(P_+ , \t{tr}_{+}\right)\right) = \left[\t{LI}\left(P_-\right)\right]' \cap P_+ $ via \Cref{BQFSthm} and by \Cref{BQFSflat}, the flat sequences correspond to elements in $ \left[\t{LI}\left(P_{- \infty}\right)\right]' \cap P_{+\infty } $.
\end{rem}

\subsubsection{Loop operators and Izumi's Markov operator}\ 
                        
We provide a description of loop operators $\left\{S_{k} : \t{End}(\Lambda_{k}) \to \t{End}(\Lambda_{k-1})\right\}_{k\geq 1}$ in terms of maps between intertwiner spaces.
We continue to employ the graphical calculus pertaining to bimodules and intertwiners as well.
Let us analyze the odd ones first.
For $Y \in V_{\mcal M_{2k}}$ and $\eta \in \t{End}(\Lambda_{2k+1})$,
\[ \left[S_{2k+1}\eta \right]_Y \ = \raisebox{-.9cm}{
\begin{tikzpicture}
\draw (-.2,0) ellipse (6mm and 8mm);
\draw[dashed,thick] (.8,-.9) to (.8,.9);
\draw[white,line width=1mm,in=-90,out=90] (.6,-.9) to (0,-.29);
\draw[red,in=-90,out=90] (.6,-.9) to (0,-.29);
\draw[white,line width=1mm,in=-90,out=90] (0,.29) to (.6,.9);
\draw[red,in=-90,out=90] (0,.29) to (.6,.9);
\node at (1,0) {$ Y $};
\node[draw,thick,rounded corners, fill=white] at (0,0) {$\eta$};
\end{tikzpicture}
} = \sum_{Y_1 \in V_{\mcal M_{2k+1}} } \sum_{\alpha \in \t{ONB} \left( Y_1 , \, \Gamma_{2k+1} Y \right)  } \raisebox{-1.9cm}{
\begin{tikzpicture}
\draw[dashed,thick] (.4,1) to (.4,.7) to[in=90,out=-90] (.2,.3) to (.2,0);
\draw[dashed,thick] (.4,-1) to (.4,-.7) to[in=-90,out=90] (.2,-.3) to (.2,0);
\draw[dashed,thick] (.6,1) to (.6,1.9);
\draw[dashed,thick] (.6,-1) to (.6,-1.9);
\draw (.2,1) to (.2,1.3) to[out=90,in=0] (-.4,1.8) to[out=180,in=90] (-1,0) to[out=-90,in=180] (-.4,-1.8) to[out=0,in=-90] (.2,-1.3) to (.2,-1);
\draw[white,line width=1mm,in=-90,out=90] (.2,-1.9) to[out=90,in=-90] (-.2,-1)  to (-.2,1) to[out=90,in=-90] (.2,1.9);
\draw[red] (.2,-1.9) to[out=90,in=-90] (-.2,-1)  to (-.2,1) to[out=90,in=-90] (.2,1.9);
\node[draw,thick,rounded corners, fill=white] at (0,0) {$\eta_{Y_1}$};	
\node[draw,thick,rounded corners, fill=white,minimum width=20] at (.4,1) {$\alpha$};	
\node[draw,thick,rounded corners, fill=white] at (.4,-1) {$\alpha^*$};	
\node at (.9,1.6) {$ Y$};
\node at (.9,-1.6) {$ Y$};
\node at (.7,.5) {$ Y_1 $};
\node at (.65,-.55) {$ Y_1 $};
\end{tikzpicture}
} \ . \]
The last expression is in terms of the functors $\Gamma_n $'s, $ \Delta_n $'s and $ \Lambda_n $'s; to express it purely using bimodules and intertwiners, we prove the following relation.
\begin{lem}
\normalfont	For $Z_1,Z_2 \in \t{ob}(\mcal N_{2k}) \ , \ \gamma \in \mcal N_{2k+1}(\Delta_{2k+1}Z_1,\Delta_{2k+1}Z_2) $ we have
\[
\raisebox{-.7cm}{
\begin{tikzpicture}
\draw[dashed,thick] (.2,-.7) to (.2,.7);
\draw (-.4,-.4) to[out=0,in=-90] (-.2,-.3) to (-.2,.3) to[out=90,in=0] (-.4,.4) to[out=180,in=180] (-.4,-.4) ;
\node[draw,thick,rounded corners, fill=white, minimum width=20] at (0,0) {$\gamma$};
\node at (.5,-.5) {$ Z_1$};
\node at (.5,.5) {$ Z_2$};
\end{tikzpicture}
} =  \delta^{-1} \raisebox{-.8cm}{
\begin{tikzpicture}
\draw (-.2,-.7) to (-.2,.7);
\draw (.2,-.3) to (.2,.3) to[out=90,in=180] (.5,.6);
\draw[->] (.5,.6) to[out=0,in=90] (.8,0);
\draw (.8,0) to[out=-90,in=0] (.5,-.6) to[out=180,in=-90] (.2,-.3);
\node[draw,thick,rounded corners, fill=white, minimum width=20] at (0,0) {$\gamma$};	
\node at (-.5,-.5) {$ Z_1$};
\node at (-.5,.5) {$ Z_2$};
\node at (.1,.6) {$ X $};
\node at (.05,-.65) {$ X $};
\end{tikzpicture}
}
\]
where the cap and the cup on the right (resp. left) side come from balanced spherical solution (resp. solution) to the conjugate equations for the duality of $X$ (resp. $\Delta_{2k+1}$ commensurate with $\left(\ul{\nu}^{2k+2}, \ul{\nu}^{2k} \right)$).
\end{lem}
\begin{proof}
Without loss of generality, we may assume $Z_1 = Z_2 = Z$ (say) is irreducible and $\gamma = 
\raisebox{-1.3cm}{
\begin{tikzpicture}
\draw[dashed,thick] (.2,.6) to (.2,1.3);
\draw[dashed,thick] (0,-.6) to (0,.6);
\draw[dashed,thick] (.2,-.6) to (.2,-1.3);
\draw (-.2,.6) to (-.2,1.3);
\draw (-.2,-.6) to (-.2,-1.3);
\node[draw,thick,rounded corners, fill=white, minimum width=20] at (0,.6) {$\alpha$};	
\node[draw,thick,rounded corners, fill=white, minimum width=20] at (0,-.6) {$\beta^*$};	
\node at (.2,0) {$ Y$};
\node at (.4,1.1) {$ Z$};
\node at (.4,-1.1) {$ Z$};
\node at (-.75,1.05) {$ \Delta_{2k+1}$};
\node at (-.75,-1.05) {$ \Delta_{2k+1}$};
\end{tikzpicture}}
= \raisebox{-1.3cm}{
\begin{tikzpicture}
		\draw (.2,.6) to (.2,1.3);
\draw (0,-.6) to (0,.6);
\draw (.2,-.6) to (.2,-1.3);
		\draw (-.2,.6) to (-.2,1.3);
		\draw (-.2,-.6) to (-.2,-1.3);
		\node[draw,thick,rounded corners, fill=white, minimum width=20] at (0,.6) {$\alpha$};	
		\node[draw,thick,rounded corners, fill=white, minimum width=20] at (0,-.6) {$\beta^*$};	
		\node at (.2,0) {$ Y $};
		\node at (.4,1.1) {$ X$};
		\node at (.4,-1.1) {$ X $};
		\node at (-.4,1.1) {$ Z$};
		\node at (-.4,-1.1) {$ Z$};
\end{tikzpicture}}$ where $ Y \in V_{\mcal N_{2k+1}} $.
So, the right side of the equation in the statement becomes
\[
\delta^{-1} \left[d(Z)\right]^{-1} \raisebox{-2cm}{
\begin{tikzpicture}
\draw (0,-.6) to (0,.6);
\draw (.2,-.6) to (.2,-1);
\draw (.2,-1) to[out=-90,in=-90] (.8,0);
\draw (.2,.6) to (.2,1);
\draw (.2,1) to[out=90,in=90] (.8,0);
\draw (-.2,-.6) to (-.2,-1.4);
\draw (-.2,-1.4) to[out=-90,in=-90] (1.1,0);
\draw (-.2,.6) to (-.2,1.4);
\draw (-.2,1.4) to[out=90,in=90] (1.1,0);
		\node[draw,thick,rounded corners, fill=white, minimum width=20] at (0,.6) {$\alpha$};	
		\node[draw,thick,rounded corners, fill=white, minimum width=20] at (0,-.6) {$\beta^*$};	
		\node at (.2,0) {$ Y $};
		\node at (0.05,1.1) {$ X$};
		\node at (.05,-1.15) {$ X $};
		\node at (-.4,1.1) {$ Z$};
		\node at (-.4,-1.15) {$ Z$};
\end{tikzpicture}} \  1_Z \  = \ \displaystyle \frac {d(Y)}{\delta \ d(Z)} \ \left \lab \alpha, \beta \right \rab\  1_Z \ = \ \displaystyle \frac{\nu^{2k}_Y}{\nu^{2k+1}_Z}\ \left \lab \alpha, \beta \right \rab\  1_Z
\]
where the first equality follows from traciality of spherical solutions.
The last term by \Cref{trsol2conj}, is same as the left side.
\end{proof}
Coming back to the loop operators, we apply the lemma on the blue box below and obtain
\begin{equation}\label{loop2Markov}
\left[S_{2k+1}\eta \right]_Y \ =  \sum_{\substack{{Y_1 \in V_{\mcal M_{2k+1}}}\\ {\alpha \in \t{ONB} \left( Y_1 , \, \Gamma_{2k+1} Y \right)  } } }  \raisebox{-1.9cm}{
\begin{tikzpicture}
		\draw[dashed,thick] (.4,1) to (.4,.7) to[in=90,out=-90] (.2,.3) to (.2,0);
		\draw[dashed,thick] (.4,-1) to (.4,-.7) to[in=-90,out=90] (.2,-.3) to (.2,0);
		\draw[dashed,thick] (.6,1) to (.6,1.9);
		\draw[dashed,thick] (.6,-1) to (.6,-1.9);
		\draw (.2,1) to (.2,1.3) to[out=90,in=0] (-.4,1.8) to[out=180,in=90] (-1,0) to[out=-90,in=180] (-.4,-1.8) to[out=0,in=-90] (.2,-1.3) to (.2,-1);
		\draw[white,line width=1mm,in=-90,out=90] (.2,-1.9) to[out=90,in=-90] (-.2,-1)  to (-.2,1) to[out=90,in=-90] (.2,1.9);
		\draw[red] (.2,-1.9) to[out=90,in=-90] (-.2,-1)  to (-.2,1) to[out=90,in=-90] (.2,1.9);
		\node[draw,thick,rounded corners, fill=white] at (0,0) {$\eta_{Y_1}$};	
		\node[draw,thick,rounded corners, fill=white,minimum width=20] at (.4,1) {$\alpha$};	
		\node[draw,thick,rounded corners, fill=white] at (.4,-1) {$\alpha^*$};	
		\node at (.9,1.5) {$ Y$};
		\node at (.9,-1.5) {$ Y$};
		\node at (.7,.5) {$ Y_1 $};
		\node at (.65,-.55) {$ Y_1 $};
\draw[blue] (-.5,-1.7) rectangle (1.2,1.7);
\end{tikzpicture}
} = \ \delta^{-1} \sum_{\substack{{Y_1 \in V_{\mcal M_{2k+1}}}\\ {\alpha \in \t{ONB} \left( Y_1 , \, Y \us C \otimes X \right)  } } } \raisebox{-2.15cm}{
\begin{tikzpicture}
\draw (.6,-1) to (.6,-1.4) to[out=-90,in=-90] (1.1,-1.7) to (1.1,1.7) to[out=90,in=90] (.6,1.4) to (.6,1);
\draw (.2,1) to (.2,1.9);
\draw (.2,-1) to (.2,-1.9);
\draw (.4,1) to (.4,.7) to[in=90,out=-90] (.2,.3) to (.2,0);
\draw (.4,-1) to (.4,-.7) to[in=-90,out=90] (.2,-.3) to (.2,0);
\draw (-.2,-1.9) to (-.2,1.9);
	\node[draw,thick,rounded corners, fill=white] at (0,0) {$\eta_{Y_1}$};	
	\node[draw,thick,rounded corners, fill=white,minimum width=20] at (.4,1) {$\alpha$};	
	\node[draw,thick,rounded corners, fill=white] at (.4,-1) {$\alpha^*$};	
\node at (.75,1.9) {$ X$};
\node at (.7,-1.9) {$ X$};
\node at (-.4,1) {$ X$};
\node at (-.4,-1) {$ X$};
	\node at (.4,1.7) {$ Y$};
	\node at (.4,-1.7) {$ Y$};
	\node at (.6,.4) {$ Y_1 $};
	\node at (.6,-.5) {$ Y_1 $};
\end{tikzpicture}
} \ .
\end{equation}
\begin{prop}
For all $\normalfont \eta \in \t{End} \left(\Lambda_k \right) $ and $ Y\in V_{\mcal M_{k-1}} $, the following equation holds
\[
\left[ S_k \eta \right]_Y = \displaystyle \sum_{\substack{{Y_1 \in V_{\mcal M_{k}}}\\ {\beta \in \t{ONB} \left( Y , \, Y_1  \otimes X_k \right)  } } } \displaystyle \frac{d(Y_1)}{\delta \ d(Y)} \raisebox{-1.95cm}{
	\begin{tikzpicture}
\draw (.2,-1) to (.2,1);
\draw (.8,-1) to (.8,1);
\draw (.5,1.1) to (.5,1.9);
\draw (.5,-1.1) to (.5,-1.9);
		\draw (-.2,-1.9) to (-.2,1.9);
		\node[draw,thick,rounded corners, fill=white] at (0,0) {$\eta_{Y_1}$};	
\node[draw,thick,rounded corners, fill=white,minimum width=25] at (.5,1.1) {$\beta^*$};	
\node[draw,thick,rounded corners, fill=white,minimum width=25] at (.5,-1.1) {$\beta$};	
\node at (1.1,0) {$ X_k$};
		\node at (-.4,1) {$ X$};
		\node at (-.4,-1) {$ X$};
		\node at (.7,1.7) {$ Y$};
		\node at (.7,-1.7) {$ Y$};
		\node at (.45,.45) {$ Y_1 $};
		\node at (.45,-.55) {$ Y_1 $};
	\end{tikzpicture}
}
\]
where $ X_k $ is $ X $ or $ \ol X $ according as $ k $ is even or odd.
\end{prop}
\begin{proof}
In \Cref{loop2Markov}, substituting $\beta \coloneqq {\left(\displaystyle \frac{d(Y)}{d(Y_1)} \right)}^{\frac 1 2}$
\raisebox{-6mm}{
\begin{tikzpicture}
\node[draw,thick,rounded corners,minimum width=20] at (0,0) {$\alpha^*$};
\draw (0,.3) to (0,.7);
\draw (-.2,-.3) to (-.2,-.7);
\draw[in=-90,out=-90,looseness=2] (.1,-.29) to (.6,-.29);
\draw (.6,-.29) to (.6,.7);
\node[left] at (0.1,.5) {$Y_1$};
\node[left] at (-.05,-.5) {$Y$};
\node[right] at (.6,0) {$\ol X$};
 \end{tikzpicture}}
(which yeilds an orthonormal basis of $ \mcal M_{2k+1} \left( Y \, , \, Y_1 \us D \otimes \ol X \right) $ as $\alpha$ varies over $ \t{ONB} \left( Y_1, Y\us C \otimes X \right) $), we get the desired equation for the odd case.
The proof of the even case is exactly similar.
\end{proof}
\vspace{4mm}
We now recall the Markov operator (that is, a UCP map) associated to an extremal finite index subactor / bifinite bimodule defined by Izumi in \cite{Izm}.
Consider the finite dimensional C*-algebra $ D_k \coloneqq \t{End} \left( \Lambda_k \right) \cong \us{Y \in V_{\mcal M_k} }{\oplus} {\vphantom{\mcal L}}_C \mcal L_{D} \left( X \us D \otimes Y \right) $ or $ \us{Y \in V_{\mcal M_k} }{\oplus} {\vphantom{\mcal L}}_C \mcal L_{C} \left( X \us D \otimes Y \right) $ according as $ k $ is even or odd.
Define the von Neumann algebra $ D \coloneqq \us{k\geq 0}\bigoplus {D_k} $.
Then, Izumi's Markov operator $ P : D \lra D$ is defined as
\[
D \ni \ul \eta = \left( \eta^{(k)} \right)_{k\geq 0} \os {\displaystyle P} \longmapsto P \ul \eta \coloneqq \left( S_{k+1} \eta^{(k+1)} \right)_{k\geq 0} \in D \ .
\]
By [Lemma 3.2, \cite{Izm}], the space of $ P $-harmonic elements $H^\infty \left(D,P\right)$ (that is, the fixed points of $P$) is precisely the space of bounded quasi-flat sequences corresponding to our loop operators $\left\{S_k \right\}_{k \geq 0}$.

\subsubsection{Temperley-Lieb - $ TL_\delta $ for $ \delta > 2 $}\ 

Continuing with the same set up, let us further assume $ X $ is symmetrically self-dual and tensor-generates the Temperley-Lieb category for a generic modulus $ \delta >2 $.
This example had already been investigated extensively, in particular, by Izumi in \cite{Izm} in our context.
Here, we address the question whether every UC-endormorphism of $ \Lambda_\bullet $ extends to a $\textbf{UC}^\t{tr}$-one.
\begin{prop}
The $ 1 $-cell in $\normalfont{\textbf{UC}^\t{tr}}$ corresponding to the TL-bimodule $ X $ possesses a BQFS to itself which is not flat.
\end{prop}
\begin{proof}
In \cite{Izm}, Izumi showed that $ H^\infty (D,P) $ (and hence $ \t{End}_{\normalfont{\textbf{UC}^\t{tr}}} (\Lambda_\bullet) $) is $ 2 $ dimesional.
So, it is enough to show that $ \t{End}_{\textbf{UC}} \left( \Lambda_\bullet\right) $ is the one-dimensional space generated by the identity in it.
Again, by \Cref{PAflat} and \Cref{BQFSflat}, this boils down to showing that $[\t{LI} \left( P_{ \infty} \right)]' \cap P_{\infty}$ is 1-dimensional where $ P =\left\{P_k\right\}_{k\geq 0} $ denotes the unshaded planar algebra associated to the symmetrically self-dual bimodule $ X $.

Let $x \in [\t{LI} \left( P_{- \infty} \right)]' \cap P_{+ \infty} $.
Then there exists some $k \geq 0$ such that $x \in [\t{LI} \left( P_{- \infty} \right)]' \cap P_{+k}$,
equivalently
\begin{equation}\label{xycommmutative}
\raisebox{-13mm}{
\begin{tikzpicture}
\draw (.4,-.9) to (.4,.7);
\draw (-.5,-1.7) to (-.5,.7);
\draw (-.2,-1.7) to (-.2,.7);
\draw (1.3,.7) to (1.3,-1.7);
\draw (.7,.7) to (.7,-.9);
\node[draw,fill=white,thick,rounded corners,minimum width=35] at (-.05,0) {$x$};
\node[draw,fill=white,thick,rounded corners,minimum width=50] at (.55,-.9) {$y$};
\node at (.15,.3) {$\cdots$};
\node at (.15,-.45) {$\cdots$};
\node at (1.05,-.45) {$\cdots$};
\node at (1.05,-.1) {$l$};
\node at (.1,.55) {$k$};
\node at (.55,-1.3) {$\cdots$};
\node at (.5,-1.55) {$k+l$};
\end{tikzpicture}} \ = \ 
\raisebox{-13mm}{
	\begin{tikzpicture}
		\draw (.4,0) to (.4,-1.7);
		\draw (-.5,-1.7) to (-.5,.7);
		\draw (-.2,-1.7) to (-.2,.7);
		\draw (1.3,.7) to (1.3,-1.7);
		\draw (.7,-1.7) to (.7,0);
		\node[draw,fill=white,thick,rounded corners,minimum width=35] at (-.05,-.9) {$x$};
		\node[draw,fill=white,thick,rounded corners,minimum width=50] at (.55,0) {$y$};
		\node at (.15,-1.25) {$\cdots$};
		\node at (.15,-.45) {$\cdots$};
		\node at (1.05,-.45) {$\cdots$};
		\node at (1.05,-.7) {$l$};
		\node at (.1,-1.5) {$k$};
		\node at (.55,.35) {$\cdots$};
		\node at (.5,.6) {$k+l$};
\end{tikzpicture}}
\ \t{for all} \ y \in P_{-\left(k+l \right)} \ \t{and} \ l \geq 0 \ .
\end{equation}
Using \Cref{xycommmutative} we get $ x = \delta^{-k}
\raisebox{-7mm}{
\begin{tikzpicture}
\draw (-.2,-.6) to (-.2,1);
\draw[line width=.6mm] (.2,-.6) to (.2,.3) to[out=90,in=90] (.6,.3) to (.6,-.1) to[out=-90,in=-90] (1,-.1) to (1,1);
\draw[line width=.6mm] (.2,.7) to[out=-90, in=-90] (.6,.7) to[out=90, in=90] (.2,.7);
\node[draw,fill=white,thick,rounded corners,minimum width=2] at (0,0) {$x$};
\end{tikzpicture}} = \delta^{-k} \raisebox{-7mm}{
\begin{tikzpicture}
\draw (-.2,-.6) to (-.2,1);
\draw[line width=.6mm] (1,-.6) to  (1,1);
\draw[line width=.6mm] (.2,-.1) to (.2,.5) to[out=90, in=90] (.6,.5) to (.6,-.1) to[out=-90, in=-90] (.2,-.1);
\node[draw,fill=white,thick,rounded corners,minimum width=2] at (0,.2) {$x$};
\end{tikzpicture}}
\ \in P_1$ where the thick line denotes $k$ many parallel strings.
Since $ P_1 $ is one-dimesnional, $ x $ must be a scalar multiple of identity.
\end{proof}

\subsection{Directed graphs}\ 

We will discuss 
an example arising out of directed gaphs (where we allow multiple edges from one vertex to the other).
Further, we assume the directed graphs are `strongly connected', that is, for $ v , w $ in the vertex set, there exists a path from $ v $ to $ w $. As a result, the corresponding adjacency matrices are irreducible and thereby, each possesses a Perron-Frobenius (PF) eigenvalue and PF eigenvectors.
In terms of category and functor, it is equivalent to consider a finite semisimple category $ \mcal M $ and a $ * $-linear functor $ \Gamma \in \t{End} \left(\mcal M \right) $ such that for simple $ v,w \in \t{ob} (\mcal M) $, there exists $ k\in \N $ satisfying $ \mcal M \left( v , \Gamma^k w \right) \neq \{0\} $.
From such a $ \Gamma $, we build the $ 0 $-cell $ \left\{\mcal M_{k-1} \os {\displaystyle \Gamma_k} \lra \mcal M_k \right\}_{k\geq 1} $ in $\textbf{UC}^\t{tr}$ where $ \mcal M_k = \mcal M $ and $ \Gamma_k = \Gamma $ for all $ k $, and the weight $ \ul \mu^k $ on $ \mcal M_k $ is given by $ d^{-k} \ul \mu $ where $ d $ is the PF eigenvalue of the adjacency matrix of $ \Gamma' $ and $ \ul \mu $ is PF eigenvector whose sum of the coordinates is $ 1 $.

\vspace{4mm}
Consider the $ 1 $-cell $ \Lambda_\bullet $ in ${\textbf{UC}^\t{tr}} \left( \Gamma_\bullet , \Gamma_\bullet \right) $ by setting $ \Lambda_k \coloneqq \Gamma $ for $ k \geq 0 $, with unitary connection $ W^k \coloneqq 1_{\Gamma^2}  $.
Note that the loop operator $ S_k : \t{End} \left(\Lambda_k\right) \ra \t{End} \left(\Lambda_{k-1}\right) $ is independent of $ k $ because (although the weight on $ \mcal M_k $ varies as $ k $ varies) our solution to conjugate equation for the duality of $ \Gamma_k : \mcal M_{k-1} \ra \mcal M_k $ is independent
; let us rename it as $ S: \t{End} \left(\Gamma\right) \ra \t{End} \left(\Gamma\right) $.
More explicity, $ S \eta = d^{-1} \raisebox{-4mm}{
\begin{tikzpicture}
\draw (-.4,-.3) to[out=-90,in=-90] (0,-.3) to (0,.3) to[out=90,in=90] (-.4,0.3) to (-.4,-.3);
\draw (.4,-.5) to (.4,.5);
\node[draw,fill=white,thick,rounded corners] at (0,0) {$\eta$};
\node at (.6,0) {$\Gamma$};
\end{tikzpicture}}
$ for all $ \eta \in \t{End} \left(\Gamma\right) $ where we use tracial solution to conjugate equation for $ \Gamma $ commensurate with $ \left(\ul \mu , \ul \mu\right) $.
Clearly, the range of $ S $ is contained in $ \left\{\xi \bigodot 1_\Gamma : \xi \in \t{End} \left( \t{id}_{\mcal M}\right) \right\} $.
Then an $ S $-harmonic sequence $ \left\{ \xi_k \bigodot 1_\Gamma\right\}_{k\geq 0} $ is completely captured by a sequence $ \left\{\xi_k\right\}_{k\geq 0} $ in the finite dimensional abelian C*-algebra $ \t{End} \left( \t{id}_{\mcal M} \right)$ satisfying $ d^{-1} \raisebox{-4mm}{
\begin{tikzpicture}
\draw (0,0) circle (3ex);
\node[draw,fill=white,thick,rounded corners] at (0,0) {$\xi_k$};
\node at (.7,0) {$\Gamma$};
\node at (-.75,0) {$\Gamma'$};
\end{tikzpicture}} = \xi_{k-1} $ for all $ k\geq 1 $.
The operator $ X \coloneqq d^{-1} \raisebox{-4mm}{
	\begin{tikzpicture}
		\draw (0,0) circle (3ex);
		\node[draw,fill=white,thick,rounded corners] at (0,0) {$\phantom{\xi_k}$};
		\node at (.7,0) {$\Gamma$};
		\node at (-.75,0) {$\Gamma'$};
\end{tikzpicture}} : \t{End} \left( \t{id}_{\mcal M} \right) \ra \t{End} \left( \t{id}_{\mcal M} \right) $ is UCP and $ \left\{\xi_k\right\}_k $ is $ X $-harmonic.
Using the categorical trace on natural transformations, the operator $ X $ has norm at most $ 1 $ and so is the spectral radius.
Applying \Cref{harmonicprop}, the bounded $ X $-harmonic sequences are linear span of elementary ones, namely, $ \left\{c^{-k} \  \xi \right\}_{k\geq 0} $ where $  \xi $ is an eigenvector of $ X $ for the eigenvalue $ c $ such that $ \abs{c} = 1$.
However, it is unclear whether such an elementary $ X $-harmonic sequence contribute towards a flat sequence from $ \Lambda_\bullet $ to $ \Lambda_\bullet $; a necessary condition for this is $c\,   1_\Gamma \bigodot \xi  = d\,  \xi \bigodot 1_\Gamma $.
A straight forward deduction from this condition will tell us that flat sequences are simple scalar multiples of the identity.

\comments{
\vspace{4mm}
(ii) Our second example of a $ 1 $-cell in $\textbf{UC}^{\,tr} \left(\Gamma_\bullet,\Gamma_\bullet\right) $ is given by $ \Lambda_k \coloneqq \t{id}_{\mcal M} \in \t{End} \left(\mcal M_k\right) $ with the unitary connection $ W_k \coloneqq W $ for all $ k $ where $ W $ is a fixed unitary in $ \t{End} \left(\Gamma \right) $.
Like example (i), the loop operators are independent of $ k $ in here as well; we denote it by $ S: \t{End} \left( \t{id}_{\mcal M} \right) \ra \t{End} \left( \t{id}_{\mcal M} \right) $.
In particular, $ S = X $ (as defined in example (i)) and hence independent of the unitary $ W $.
So, any BQFS from $ \Lambda_\bullet $ to $ \Lambda_\bullet $ is given by a bounded $ X $-harmonic sequence; the elementary ones share the same necessary condition as in example (i) for being flat.

\vspace{4mm}
}

In 
the above example, if we would have started with a finite connected undirected graph $ \Gamma $, then by \Cref{periodicOcneanucompact}, all BQFS from $ \Lambda_\bullet $ to $ \Lambda_\bullet $ would have been flat.

\subsection{Vertex models}\ 

Let $ \mcal M $ be the category of finite dimensional Hilbert spaces, and $ \Gamma \coloneqq \t{id}_{\mcal M} \otimes \ell^2 (X) \in \t{End} \left(\mcal M \right)  $, $ \Lambda \coloneqq \t{id}_{\mcal M} \otimes \ell^2 (Y) \in \t{End} \left(\mcal M \right) $ be two functors where $ X , Y $ are some nonempty finite sets.
Consider the $ 0 $-cell $ \Gamma_\bullet $ in $\textbf{UC}^\t{tr}$ defined by $ \mcal M_k \coloneqq \mcal M $ and $ \Gamma_k = \Gamma $ for all $ k$ where the weight of $ \C $ in $ \mcal M_k $ is $ \abs{X}^{-k} $.
For a $ 1 $-cell in $\textbf{UC}^\t{tr} \left( \Gamma_\bullet , \Gamma_\bullet \right)$, we consider $ \left\{\Lambda_k \coloneqq \Lambda \right\}_{k\geq 0} $ with the unitary connections $ W^k \coloneqq \t{id}_\bullet \otimes U F $ where $ U : \ell^2 (X) \otimes \ell^2(Y) \ra \ell^2(X) \otimes \ell^2(Y)  $ is a unitary and $ F : \ell^2 (Y) \otimes \ell^2(X) \ra \ell^2(X) \otimes \ell^2(Y)  $ is canonical flip map.
Note that $ \t{End} \left(\Lambda \right)  \cong M_Y (\C) $ and the loop operators are independent of $ k $; let us denote it by $ S: M_Y (\C) \ra M_Y (\C) $.
One can deduce the following two formula,\\
\[
\left(S\eta \right)_{y,y'} =  \abs{X}^{-1} \us{\substack{{x_1,x_2 \in X}\\{y_1,y_2\in Y}}}{\sum} \overline{U^{x_2 y_2}_{x_1 y}}\ \eta_{y_2 y_1}\ U^{x_2 y_1}_{x_1 y'} \ \ \t{ and } \ \ 
\left(S^{*}\eta \right)_{y,y'} =  \abs{X}^{-1} \us{\substack{{x_1,x_2 \in X}\\{y_1,y_2\in Y}}}{\sum} U^{x_1 y}_{x_2 y_2}\ \eta_{y_2 y_1} \ \overline {U^{x_1 y'}_{x_2 y_1}}
\]
for all $\eta \in M_{Y}(\C)$ and $y,y' \in Y $.
By \Cref{periodicOcneanucompact}, every BQFS from $ \Lambda_\bullet $ to $ \Lambda_\bullet $ becomes flat.